\documentclass[11pt]{article}

\makeatletter
\def\input@path{{styles/}}
\makeatother

\newcommand{\UsePackage}[1]{%
  \IfFileExists{styles/#1.sty}{%
      \usepackage{styles/#1}%
   }{%
      \IfFileExists{../styles/#1.sty}{%
         \usepackage{../styles/#1}%
      }{%
         \usepackage{#1}%
      }%
   }%
}

\usepackage[T1]{fontenc}
\usepackage{lmodern}
\usepackage{textcomp}
\usepackage[numbers]{natbib}
 
\usepackage{amsmath}%
\usepackage{amssymb}%
\usepackage[table]{xcolor}%

\usepackage{todonotes}
\usepackage[in]{fullpage}%

\usepackage[amsmath,thmmarks]{ntheorem}%
\theoremseparator{.}%

\usepackage{titlesec}%
\titlelabel{\thetitle. }%
\usepackage{xcolor}%
\usepackage{mleftright}%
\usepackage{xspace}%
\usepackage{graphicx}
\usepackage{hyperref}%

\usepackage{hyperref}%
\hypersetup{%
      unicode,
      breaklinks,%
      colorlinks=true,%
      urlcolor=[rgb]{0.25,0.0,0.0},%
      linkcolor=[rgb]{0.5,0.0,0.0},%
      citecolor=[rgb]{0,0.2,0.445},%
      filecolor=[rgb]{0,0,0.4},
      anchorcolor=[rgb]={0.0,0.1,0.2}%
}
\usepackage[ocgcolorlinks]{ocgx2}

\theoremseparator{.}%

\theoremstyle{plain}%
\newtheorem{theorem}{Theorem}[section]

\newtheorem{lemma}[theorem]{Lemma}

\newtheorem{corollary}[theorem]{Corollary}

\newtheorem{proposition}[theorem]{Proposition}

\newtheorem{assumption}[theorem]{Assumption}%

\theoremstyle{plain}%
\theoremheaderfont{\sf} \theorembodyfont{\upshape}%
\newtheorem*{remark:unnumbered}[theorem]{Remark}%
\newtheorem{remark}[theorem]{Remark}%
\newtheorem{definition}[theorem]{Definition}

\newcommand{\myqedsymbol}{\rule{2mm}{2mm}}

\theoremheaderfont{\em}%
\theorembodyfont{\upshape}%
\theoremstyle{nonumberplain}%
\theoremseparator{}%
\theoremsymbol{\myqedsymbol}%
\newtheorem{proof}{Proof:}%

\providecommand{\emphind}[1]{}%
\renewcommand{\emphind}[1]{\emph{#1}\index{#1}}

\definecolor{blue25emph}{rgb}{0, 0, 11}

\providecommand{\emphic}[2]{}
\renewcommand{\emphic}[2]{\textcolor{blue25emph}{%
      \textbf{\emph{#1}}}\index{#2}}

\providecommand{\emphi}[1]{}%
\renewcommand{\emphi}[1]{\emphic{#1}{#1}}

\definecolor{almostblack}{rgb}{0, 0, 0.3}

\providecommand{\emphw}[1]{}%
\renewcommand{\emphw}[1]{{\textcolor{almostblack}{\emph{#1}}}}%

\providecommand{\emphOnly}[1]{}%
\renewcommand{\emphOnly}[1]{\emph{\textcolor{blue25}{\textbf{#1}}}}

\newcommand{\HLink}[2]{\hyperref[#2]{#1~\ref*{#2}}}
\newcommand{\HLinkSuffix}[3]{\hyperref[#2]{#1\ref*{#2}{#3}}}

\providecommand{\deflab}[1]{}
\renewcommand{\deflab}[1]{\label{def:#1}}

\providecommand{\eqlab}[1]{}%
\renewcommand{\eqlab}[1]{\label{equation:#1}}

\newcommand{\remove}[1]{}%

\usepackage{amsopn}

\DeclareMathOperator*{\argmax}{arg\,max}
\DeclareMathOperator*{\argmin}{arg\,min}

\usepackage[inline]{enumitem}

\newlist{compactenumA}{enumerate}{5}%
\setlist[compactenumA]{topsep=0pt,itemsep=-1ex,partopsep=1ex,parsep=1ex,%
   label=(\Alph*)}%

\newlist{compactenuma}{enumerate}{5}%
\setlist[compactenuma]{topsep=0pt,itemsep=-1ex,partopsep=1ex,parsep=1ex,%
   label=(\alph*)}%

\newlist{compactenumI}{enumerate}{5}%
\setlist[compactenumI]{topsep=0pt,itemsep=-1ex,partopsep=1ex,parsep=1ex,%
   label=(\Roman*)}%

\newlist{compactenumi}{enumerate}{5}%
\setlist[compactenumi]{topsep=0pt,itemsep=-1ex,partopsep=1ex,parsep=1ex,%
   label=(\roman*)}%

\newlist{compactitem}{itemize}{5}%
\setlist[compactitem]{topsep=0pt,itemsep=-1ex,partopsep=1ex,parsep=1ex,%
   label=\ensuremath{\bullet}}%

\numberwithin{figure}{section}%
\numberwithin{table}{section}%
\numberwithin{equation}{section}%

\newcommand{\va}{{\mathbf{a}}}

\newcommand{\vs}{{\mathbf{s}}}

\newcommand{\vg}{{\mathbf{g}}}
\newcommand{\vu}{{\mathbf{u}}}
\newcommand{\vv}{{\mathbf{v}}}

\newcommand{\bv}{{\mathbf{v}}}
\newcommand{\vw}{{\mathbf{w}}}

\newcommand{\vx}{{\mathbf{x}}}
\newcommand{\bx}{{\mathbf{x}}}
\newcommand{\vy}{{\mathbf{y}}}
\newcommand{\by}{{\mathbf{y}}}

\newcommand{\bz}{{\mathbf{z}}}
\newcommand{\bq}{{\mathbf{q}}}
\newcommand{\bp}{{\mathbf{p}}}
\newcommand{\bQ}{{\mathbf{Q}}}
\newcommand{\bP}{{\mathbf{P}}}
\newcommand{\bG}{{\mathbf{G}}}
\newcommand{\vG}{{\mathbf{G}}}

\newcommand{\bI}{{\mathbf{I}}}
\newcommand{\vDelta}{{\boldsymbol{\Delta}}}
\newcommand{\vlambda}{{\boldsymbol{\lambda}}}
\newcommand{\vdelta}{{\boldsymbol{\delta}}}
\newcommand{\vxi}{{\boldsymbol{\xi}}}

\newcommand{\cX}{{\mathcal{X}}}
\newcommand{\cY}{{\mathcal{Y}}}

\usepackage{lipsum}
\usepackage{amsfonts}
\usepackage{graphicx}
\usepackage{epstopdf}
\usepackage{algorithmic}
\usepackage{multirow}
\usepackage{algorithm}
\ifpdf
\DeclareGraphicsExtensions{.eps,.pdf,.png,.jpg}
\else
\DeclareGraphicsExtensions{.eps}
\fi

\begin{document}

\title{Deterministic and Stochastic Accelerated Gradient Method for Convex Semi-Infinite Optimization
\thanks{This work is jointly supported by the University of Iowa
Jumpstarting Tomorrow Program and NSF award 2147253.}
}

\author{Yao Yao\thanks{University of Iowa, Iowa City, IA (yao-yao-2@uiowa.edu, qihang-lin@uiowa.edu).}
	\and Qihang Lin\footnotemark[2]
	\and Tianbao Yang\thanks{Texa A\&M University, College Station, TX (tianbao-yang@tamu.edu).}}

\date{}

\maketitle

\begin{abstract}
   This paper explores numerical methods for solving a convex differentiable semi-infinite program. We introduce a primal-dual gradient method which performs three updates iteratively: 
a momentum gradient ascend step to update the constraint parameters, a momentum gradient ascend step to update the dual variables, and a gradient descend step to update the primal variables. Our approach also extends to scenarios where gradients and function values are accessible solely through stochastic oracles. This method extends the recent primal-dual methods, for example, \cite{hamedani2021primal,boob2022stochastic}, for optimization with a finite number of constraints. We show the iteration complexity of the proposed method for finding an $\epsilon$-optimal solution under different convexity and concavity
assumptions on the functions.  
\end{abstract}


\section{Introduction}
\label{intro}
We consider the following constrained optimization problem
\begin{eqnarray}
\label{eq:SIP}
f^*&:=&\min_{\bx\in\cX} f(\bx)\quad\text{s.t.}\quad g_i(\bx,\by^i)\leq 0,~\forall \by^i\in \cY^i,~i=1,\dots,m.
\end{eqnarray}
Here, $f$ and $g_i$, $i=1,\dots,m$ are real-valued continuous functions, and $\cX\subset\mathbb{R}^p$ and $\cY^i\subset\mathbb{R}^q$, $i=1,\dots,m$ are convex closed sets. Since $\cY^i$ is not a finite set,  the number of constraints in~\eqref{eq:SIP} is infinite, so \eqref{eq:SIP} is a \emph{semi-infinite program} (SIP)~\cite[and references therein]{reemtsen1998semi,goberna2002linear,lopez2007semi,goberna2013semi}. 
SIP has been studied systematically since the 1970s and has a number of applications including robotics~\cite{marin1988optimal,hettich1991semi,vaz2004robot},
statistics~\cite{dall2001some}, machine learning~\cite{ozogur2009modelling,ozougur2010numerical,sonnenburg2006large}, stochastic programming~\cite{dentcheva2004semi}, robust optimization~\cite{mehrotra2014cutting,liu2015semi,sun2017quadratic,luo2017decomposition}, Markov decision process~\cite{de_farias_linear_2003,lin2020revisiting,pakiman2020self}, inventory control \cite{adelman_price-directed_2004,adelman2012computing,adelman2013unifying}, revenue management \cite{adelman_dynamic_2007}, queuing \cite{VFarias_VanRoy_2007, zhang2009approximate, Adelman_2013}, and health care \cite[ch. 4]{patrick2008dynamic,restrepo2008}. 

In this paper, we study the first-order method for finding an  \emph{$\epsilon$-optimal} solution for \eqref{eq:SIP}, which is a solution $\bar\bx\in\mathcal{X}$ satisfying
\begin{eqnarray}
\label{eq:epsilonopt_def}
f(\bar\bx)-f^*\leq \epsilon\text{ and }g_i^*(\bar\bx)\leq \epsilon,~i=1,\dots,m,
\end{eqnarray}
where
\begin{eqnarray}
\label{eq:gistar}
g_i^*(\bx):=\max_{\by^i\in\mathcal{Y}^i}g_i(\bx,\by^i),~i=1,\dots,m.
\end{eqnarray}
We mainly focus on the convex case although some extension to non-convex problems arising from machine learning applications will be discussed later. Hence, the following assumptions are made for problem \eqref{eq:SIP}.
\begin{assumption}[Convexity and compactness]
\label{assume:basic}
The following statements hold.
\begin{enumerate}
\item[a.] $\cX\subset\mathbb{R}^p$ and $\cY^i\subset\mathbb{R}^q$, $i=1,\dots,m$ are convex and closed. 
\item[b.] There exists $D_y\in\mathbb{R}$ such that $\|\by^i-\by^{i\prime}\|\leq D_y$ for any $\by^i,\by^{i\prime}\in\cY^i$ for $i=1,\dots,m$. 
\item[c.] $f$ is $\mu_f$-strongly convex in $\bx$ for $\mu_f\geq0$. 
\item[d.] $g_i$ is convex in $\bx$ for $i=1,\dots,m$. 
\item[e.] $g_i$ is $\mu_y$-strongly concave in $\by^i$ for $i=1,\dots,m$ and $\mu_y\geq0$.
\end{enumerate}
\end{assumption}
Note that $\mu_f$ and $\mu_y$ can be zero. As shown in this paper, the theoretical complexity of the studied method depends on whether $\mu_f$ or $\mu_y$ or both are positive. We can also assume $g_i$ is $\mu_x$-strongly convex in $\bx$ for $i=1,\dots,m$ with $\mu_x\geq0$. However, whether $\mu_x$ is positive or not does not affect the order of iteration complexity. This is consistent with the findings when there are finitely many constraints \cite{hamedani2021primal,boob2022stochastic}. Hence, we do not introduce $\mu_x$ for simplicity. Here, we assume $\cY^i$ has the same dimension for $i=1,\dots,m$ only to simplify the notation. In fact, all the results in this paper still hold when each $\cY^i$ has its own dimension. 

Let $\vlambda=(\lambda_1,\dots,\lambda_m)^\top\in\mathbb{R}_+^m$ be the Lagrangian multipliers of the constraints in \eqref{eq:SIP}. We assume that there exist Lagrangian multipliers that satisfy the optimality conditions together with the optimal solution of \eqref{eq:SIP}. 
\begin{assumption}[Existence of  optimal multipliers]
\label{assume:KKT}
For any optimal solution $\vx^*$  of \eqref{eq:SIP}, there exists $\vlambda^*\in\mathbb{R}_+^m$ such that
\begin{eqnarray*}
\vx^*\in\argmin_{\vx\in\cX}\left[f(\bx)+(\vlambda^*)^\top \vg^*(\bx)\right]
\text{ and }
\lambda_i^*g_i^*(\bx^*)=0, \quad i=1,\dots,m,
 \end{eqnarray*}
where 
\begin{eqnarray}
\label{eq:gvecstar}
\vg^*(\bx)=(g_1^*(\bx),\dots,g_m^*(\bx))^\top.
\end{eqnarray}
\end{assumption}

In addition, we assume $f$ and $g_i$, $i=1,\dots,m$, are real-valued and continuous. We consider both smooth and non-smooth cases. In the smooth case, let $\nabla f$ be the gradient of $f$ and let  $\nabla_x g_i$ and $\nabla_y g_i$ be the gradient of $g$ with respect to $\bx$ and $\by^i$, respectively. In the non-smooth case,  let $\partial f$ be the subdifferential of $f$ and let  $\partial_x g_i$ and $\partial_y g_i$ be the subdifferential of $g$ with respect to $\bx$ and $\by^i$, respectively. With a little abuse of notation, in the non-smooth case, we still use $\nabla f$, $\nabla_x g_i$ and $\nabla_y g_i$ to denote subgradients in  $\partial f$, $\partial_x g_i$ and $\partial_y g_i$, respectively. Let $\|\cdot\|$ denote the Euclidean norm throughout the paper. The following assumptions are also made. 
\begin{assumption}[Smoothness and Lipschitz continuity]
\label{assume:continuity}
There exist non-negative constants $L_f$, $L_g^{xx}$, $L_g^{yx}$, $L_g^{yy}$, $H_f$,  $H_g^x$, $H_g^y$, $M_g^{x}$ and $M_g^{y}$ such that the following statements hold. 
\begin{enumerate}
\item[a.] $f(\bx)\leq f(\bx')+\nabla f(\bx')^\top(\bx-\bx')+\frac{L_f}{2}\|\bx-\bx'\|^2+H_f\|\bx-\bx'\|$ for any $\bx,\bx'\in \cX$.
\item[b.] 
$g_i(\vx,\vy^i)-g_i(\vx',\vy^i)- \nabla_x g_i(\vx',\vy^i)^\top (\vx-\vx')\leq\frac{L_g^{xx}}{2} \|\vx-\vx'\|^2+ H_g^{x} \|\vx-\vx'\|$  for any $\bx,\bx'\in \cX$ and any $\by^i\in\cY^i$.
\item[c.]  $\|\nabla_y g_i(\bx,\by^i)-\nabla_y g_i(\bx',\by^{i\prime})\|\leq L_g^{yx} \|\bx-\bx'\|+L_g^{yy} \|\by^i-\by^{i\prime}\|+H_g^y$ for any $\bx,\bx'\in \cX$ and any $\by^i,\by^{i\prime}\in\cY^i$.
\item[d.]  $|g_i(\bx,\by^i)- g_i(\bx',\by^{i\prime})|\leq M_g^{x} \|\bx-\bx'\|+M_g^{y} \|\by^i-\by^{i\prime}\|$ for any $\bx,\bx'\in \cX$ and any $\by^i,\by^{i\prime}\in\cY^i$.
\end{enumerate}
\end{assumption}
This assumption encompasses both smooth and non-smooth cases. In fact, the smooth case occurs when $H_f=H_g^x=H_g^y=0$. In this case, Assumptions~\ref{assume:continuity}a, b and c require $\nabla f$, $\nabla_x g_i$ and $\nabla_y g_i$ are unique at each point and satisfy some Lipschitz properties. In fact, the non-smooth case occurs when $H_f$, $H_g^x$ and $H_g^y$ are non-zero while $L_f=L_g^{xx}=L_g^{yx}=L_g^{yy}=0$. In this case, Assumptions~\ref{assume:continuity}a, b and c imply that $\nabla f$, $\nabla_x g_i$ and $\nabla_y g_i$ are bounded.  Assumption~\ref{assume:continuity} d means $g_i$ is Lipschitz continuous, which implies
\begin{equation}
\label{eq:Lipgi}
\|\nabla_x g_i(\vx,\vy^i)\|\leq  M_g^{x}\text{ and } \|\nabla_y g_i(\vx,\vy^i)\|\leq  M_g^{y}
\end{equation}
for any $\bx\in \cX$ and any $\by^i,\by^{i\prime}\in\cY^i$.

We consider both the case where the gradients and function values above can be computed exactly and the case where they can only be approximated by stochastic oracles. In the latter case, we call \eqref{eq:SIP} a \emph{stochastic semi-infinite program} (SSIP), for which the following standard assumptions are made. 
\begin{assumption}[Stochastic oracles]
\label{assume:oracle}
There exist a random vector $\xi$ and measurable mappings $\nabla F(\vx,\xi)\in \mathbb{R}^p$, $G_i(\vx,\vy^i,\xi)\in\mathbb{R}$, $\nabla_x G_i(\vx,\vy^i,\xi)\in\mathbb{R}^p$, and $\nabla_y G_i(\vx,\vy^i,\xi)\in\mathbb{R}^q$ that provide stochastic oracles for $\nabla f(\vx)$, $ g_i(\vx,\vy^i)$, $\nabla_x g_i(\vx,\vy^i)$, and $\nabla_y g_i(\vx,\vy^i)$, respectively. Moreover, the following conditions hold.

Unbiasedness:
\begin{eqnarray}
    \label{eq:mean_gf}
    \mathbb{E}[\nabla F(\vx,\xi)]&=&\nabla f(\vx),\\\label{eq:mean_Gi}
    \mathbb{E}[G_i(\vx,\vy^i,\xi)]& =& g_i(\vx,\vy^i),\quad \quad i=1,\cdots,m\\\label{eq:mean_Ggx}
    \mathbb{E}[\nabla_x G_i(\vx,\vy^i,\xi)]& =& \nabla_x g_i(\vx,\vy^i),\quad \quad i=1,\cdots,m\\\label{eq:mean_Ggy}
    \mathbb{E}[\nabla_y G_i(\vx,\vy^i,\xi)]& =& \nabla_y g_i(\vx,\vy^i),\quad \quad i=1,\cdots,m.
\end{eqnarray}
Bounded variance:
\begin{eqnarray}
    \label{eq:var_gf}
    \mathbb{E}[\|\nabla F(\vx,\xi)-\nabla f(\vx)\|^2]&\leq& \sigma_{f'}^2,\\\label{eq:var_Gi}
    \mathbb{E}[\|G_i(\vx,\vy^i,\xi)-g_i(\vx,\vy^i)\|^2]& \leq& \sigma_{g}^2,\quad \quad i=1,\cdots,m\\\label{eq:var_Ggx}
    \mathbb{E}[\|\nabla_x G_i(\vx,\vy^i,\xi)-\nabla_x g_i(\vx,\vy^i)\|^2]& \leq& \sigma_{g'}^2,\quad \quad i=1,\cdots,m\\\label{eq:var_Ggy}
    \mathbb{E}[\|\nabla_y G_i(\vx,\vy^i,\xi)-\nabla_y g_i(\vx,\vy^i)\|^2]& \leq& \sigma_{g'}^2,\quad \quad i=1,\cdots,m,
\end{eqnarray}
where $ \sigma_{f'}$, $\sigma_{g}$ and $\sigma_{g'}$ are non-negative constants. 
\end{assumption}

Let $\cY:=\cY^1\times\cY^2\times\cdots\times\cY^m\subset\mathbb{R}^{mq}$ and 
$\by=(\by^{1\top},\dots,\by^{m\top})^\top\in \cY$. 
Consider the Lagrangian function of \eqref{eq:SIP}
\begin{equation}
\label{eq:Lagrangian}
    \phi(\bx,\vlambda,\by) = f(\bx)+\vlambda^\top \vg(\bx,\by).
\end{equation}
Given an optimal solution $\bx^*$ and the Lagrangian multipliers $\vlambda^*$ in Assumption~\ref{assume:KKT}, let $\by^*\in\argmax_{\vy\in\cY}\vg(\bx^*,\vy)$. It is known that $(\bx^*,\vlambda^*,\by^*)$ is a saddle-point of \eqref{eq:Lagrangian} in the sense that 
$$
\phi(\bx^*,\vlambda,\by)\leq \phi(\bx^*,\vlambda^*,\by^*)\leq \phi(\bx,\vlambda^*,\by^*)
$$
for any $(\bx,\vlambda,\by)\in\cX\times\mathbb{R}_+^m\times\cY$.
This motivates using a primal-dual gradient method for finding a saddle-point of $\phi(\bx,\vlambda,\by)$. However,  
Assumptions~\ref{assume:basic} only ensures the convexity of $\phi$ on $\bx$ but not the joint concavity of $\phi$ on $(\vlambda,\by)$, so one cannot directly apply the existing primal-dual methods for convex-concave min-max problems (e.g.,~\cite{nemirov09,nemirovski-2005-prox,hamedani2021primal,boob2022stochastic}) and their convergence analysis. 

That being said, \eqref{eq:SIP} is indeed tractable under the convexity and concavity assumptions from Assumptions~\ref{assume:basic}. In fact, thanks to the concavity of $g_i$ in $\by^i$, the lower-level problem $\max_{\by^i\in\cY^i}g_i(\bx,\by^i)$ for $i=1,\dots,m$ are relatively easy to solve. Suppose the optimal solution to each lower problem, denoted by $\by^i_*$ for $i=1,\dots,m$, is obtained. By Danskin's theorem, one can obtain $\nabla g_i^*(\bx)=\nabla_xg_i(\bx,\by^i_*)$ and then solve \eqref{eq:SIP} as a problem with finitely many constraints, $g_i^*(\bx)\leq 0$ for $i=1,\dots,m$, for example, using the methods in \cite{hamedani2021primal,boob2022stochastic}. When computing $\by^i_*$ exactly is not practical, one can first find an approximate solution to the lower-level problem,  denoted by $\widehat\by^i$ for $i=1,\dots,m$,  and then approximate $\nabla g_i^*(\bx)$ by $\nabla _xg_i(\bx,\widehat\by^i)$. This way, one can still solve \eqref{eq:SIP} as a finitely constrained problem using the inexact gradients of $g_i^*(\bx)$. However, this approach usually requires a small approximation error in terms of $\|\nabla g_i^*(\bx)-\nabla _xg_i(\bx,\widehat\by^i)\|$, which not only needs strong assumptions, for example,  strong concavity and smoothness of $g_i$ in $\by$, but also leads to high computational complexity for solving each lower-level problem. Moreover, this approach involves double loops, which is not easy to implement. 

To address these issues, we propose a first-order algorithm that only performs one gradient ascend step on $g_i$ to update $\by^i$ in each iteration, followed by primal-dual gradient steps to update $\vlambda$ and $\vx$ using the function values and gradients on the historical iterates. This approach involves only a single loop, making it easy to implement. Additionally, it does not assume strong convexity/concavity or smoothness and does not require solving the lower-level problems explicitly to any predetermined error.  Since $\vlambda$ and $\by$ are not updated jointly, our convergence analysis is unaffected by the non-concavity of $\phi$ in $(\vlambda,\by)$. Consequently, we theoretically analyze the number of iterations our method requires to find an $\epsilon$-optimal solution of \eqref{eq:SIP} under various settings, depending on the smoothness and convexity/concavity of the problem as well as the type of gradient oracles used (deterministic or stochastic). We summarize the iteration complexity of our method in Table~\ref{table:1}.

\begin{table}[t]
\centering
\begin{tabular}{|c |c |c | c| c| c| |c|} 
 \hline
Method & $\mu_f$ & $\mu_y$ & Gradients& Smoothness& Oracles for lower-level problem & Iters \\ [0.5ex] 
 \hline\hline
\multirow{7}{*}{Alg.~\ref{alg:PDHG}}  & $>0$ & $>0$ & \multirow{7}{*}{Deter. } &S& \multirow{7}{*}{$g_i$, $\nabla_x g_i$, $\nabla_y g_i$} & $O(\epsilon^{-0.5})$ \\ 
 & $>0$ & $>0$ &  &NS& & $O(\epsilon^{-1})$ \\ 
 & $>0$ & $=0$ &   &S or NS& & $O(\epsilon^{-1})$ \\
& $=0$ & $>0$ &   &S& & $\tilde O(\epsilon^{-1})$ \\
& $=0$ & $>0$ &   &NS& & $O(\epsilon^{-2})$ \\
 & $=0$ & $=0$ &   &S& & $O(\epsilon^{-1})$ \\ 
  & $=0$ & $=0$ &  &NS& & $O(\epsilon^{-2})$ \\ \hline
Alg.~\ref{alg:SPDHG} & $\geq 0$ & $\geq 0$ & Stoch. &S or NS& $G_i$, $\nabla_x G_i$, $\nabla_y G_i$& $O(\epsilon^{-2})$ \\ \hline
\multirow{2}{*}{\cite{Wei2020TheCA} }& $=0$ & NA & \multirow{2}{*}{Deter. }  &S or NS&\multirow{2}{*}{$g_i$, $\nabla_x g_i$, $\approx\argmax\limits_{\by^i\in\cY^i} g_i(\bx,\by^i)$}& $O(\epsilon^{-2})$ \\
& $>0$ & NA &  &S or NS& & $O(\epsilon^{-1})$ \\ \hline
\multirow{2}{*}{\cite{wei2020inexact,lin2020revisiting}} & \multirow{2}{*}{$=0$} & \multirow{2}{*}{NA} &Deter.  &\multirow{2}{*}{S or NS}& $g_i$, $\nabla_x g_i$ or $G_i$, $\nabla_x G_i$, & \multirow{2}{*}{$O(\epsilon^{-2})$ }\\ 
& &  & or Stoch. && sample from a density on $\cY^i$  &  \\ 
\hline
\end{tabular}
\caption{Comparison of the oracles used by different methods and their iteration complexity. Here, ``Deter.'' means ``deterministic'' and ``Stoch.'' means ``stochastic''. Besides, ``S'' and ``NS'' represent ``smooth'' and ``nonsmooth'', respectively, in the table. $\tilde O(\cdot)$ means some logarithmic factors of $\epsilon$ are omitted. }
\label{table:1}
\end{table}

This paper is not the first effort on developing first-order methods for SIP. Wei et al. in \cite{wei2020inexact} propose primal-dual gradient methods for SIP and SSIP, where the dual variables are probability densities on $\cY^i$'s. They require sampling from $\cY^i$'s using those densities whose computational cost is high when  $\cY^i$'s have a high dimension. Wei et al. also propose in \cite{Wei2020TheCA} a cooperative (switching) gradient method for SIP where, in each iteration, they assume a nearly optimal solution to $\max_{\by^i\in\cY^i} g_i(\bx,\by^i)$ can be found. On the contrary, our method only needs to compute the (stochastic) gradients of $g_i$ and $f$ per iteration, which in general has significantly lower computational cost than the oracles (sampling over a density or solving a maximization problem) needed in \cite{wei2020inexact,Wei2020TheCA}. Note that, because of the strong oracles they assume, $g_i$ does not need to be concave in $\by^i$ in \cite{wei2020inexact,Wei2020TheCA}.

Despite using weaker oracles,  our method needs the same or fewer iterations than  \cite{wei2020inexact,Wei2020TheCA} for finding an $\epsilon$-optimal solution. As presented in Table~\ref{table:1}, our method needs only $O(\epsilon^{-1})$ iterations for SIP when $\mu_f=0$ and $\mu_y>0$ while  the methods in \cite{wei2020inexact,Wei2020TheCA} need $O(\epsilon^{-2})$ iterations if applied to the same problems. When $\mu_f>0$ and $\mu_y=0$, our method needs $O(\epsilon^{-1})$ iterations while \cite{Wei2020TheCA} needs to further assume $g_i$ is strongly convex in $\bx$ to achieve the same complexity. Overall, this paper is the first work that establishes the iteration complexity of a single-loop method for convex SIP and SSIP under the scenarios listed in Table~\ref{table:1} using only (stochastic) gradient oracles.



\section{Related Works}
SIP has been systematically discussed in several monographs~\cite{reemtsen1998semi,reemtsen1998numerical,goberna2002linear,lopez2007semi,hettich1993semi,goberna2013semi} and has many applications as listed in the previous section. Most numerical algorithms for optimization with finitely many constraints can be somehow extended for SIP.  The examples include the penalty method~\cite{conn1987exact,lin2014new,xu2014solving}, the barrier method~\cite{kaliski1997logarithmic,luo1999complexity},  the interior-point method~\cite{ferris1989interior,todd1994interior}, the Lagrangian method~\cite{ruckmann2009augmented,coope1985projected}, the sequential quadratic programming method~\cite{tanaka1988globally,gramlich1995local}, and the trust-region method~\cite{tanaka1999trust}. However, most of these works only show an asymptotic convergence property or a local convergence rate. On the contrary, we establish the global non-asymptotic convergence rate for our algorithm and characterize its iteration complexity for finding an $\epsilon$-optimal solution. Luo et al.  \cite{luo1999complexity} also analyze the complexity of a logarithmic barrier method, but their study only focuses on a linear SIP instead of the nonlinear problem \eqref{eq:SIP}. 

Besides the aforementioned approaches stemmed from nonlinear programming, there exist techniques specifically developed for SIP based on different strategies for handling infinitely many constraints. The cutting-plane method~\cite{kortanek1993central,wu1998relaxed,betro2004accelerated,pang2016constrained,fang2001solving,oustry2023convex,cerulli2022convergent} solves the lower-level problem,  $\max_{\by^i\in\cY^i}g_i(\bx,\by^i)$, $i=1,\dots,m$ up to certain optimality gap in order to construct a cutting plane to update solutions. The computational complexity for solving the subproblems is non-negligible but not explicitly analyzed in those works. Our method does not directly solve $\max_{\by^i\in\cY^i}g_i(\bx,\by^i)$ for any targeted precision but only take a momentum gradient ascend step on $\by^i$ for $i=1,\dots,m$ in each iteration. 

The discretization method~\cite{teo2000computational,still2001discretization,xu2013nonlinear} chooses or samples a finite set $\widehat\cY^i \subset \cY^i$ and solve a relaxation of \eqref{eq:SIP} where the constraints are $g_i(\bx,\by^i)\leq 0$ for any $\by^i\in\widehat\cY^i$ and $i=1,\dots,m$. It is easy to implement but, to ensure a tight relaxation, the size of $\widehat\cY^i$ must typically increase exponentially with the dimension $q$, known as the curse of dimensionality. 
A related method is the exchange method~\cite{fang2001solving,zhang2010new,kortanek1993central,wu1998relaxed} which updates $\widehat\cY^i$ iteratively by including the (nearly) optimal solution of the lower-level problem and thus has the same computational issue as the cutting-plane method mentioned earlier. 

The reduction method~\cite{pereira2009reduction,hettich1993semi} assumes the optimal solution to the lower problem, namely, $\by^i(\bx)\in\argmax_{\by^i\in\cY^i}g_i(\bx,\by^i)$, is an implicit function of $\bx$ so the infinitely many constraints in \eqref{eq:SIP} can be replaced by finitely many constraints $g_i(\bx,\by^i(\bx))\leq 0$, $i=1,\dots,m$ only on $\bx$. However, this approach needs strong assumptions on the regularity of $g_i$ and  $\cY^i$ to ensure uniqueness and differentiability of $\by^i(\bx)$.

Compared to the existing approaches for SIP, the numerical tools for SSIP remain rare and are only developed for a few specific applications~\cite{guo2008ismisip,guo2009interval,lin2017revisiting}. This paper fills this gap by providing a stochastic gradient method for SSIP with theoretical complexity analysis. Suppose $\cY^i$ depends on $\bx$ as a set-value mapping $\cY^i(\bx):\mathcal{X}\rightarrow 2^{\mathbb{R}^q}$. Problem~\eqref{eq:SIP} becomes a generalized semi-infinite program~\cite{vazquez2008generalized}, which is beyond the scope of this paper and is our future research direction.

The algorithm proposed in this paper is motivated by the recent development on the  first-order methods for min-max optimization, constrained optimization and composite optimization~\cite{yao2023stochastic,chambolle2011first,chambolle2016ergodic,hamedani2021primal,zhang2020optimal,boob2022stochastic}. The accelerated primal-dual first-order method is originally developed by~\cite{chambolle2011first,chambolle2016ergodic} for a bilinear min-max problem. In their method, a momentum gradient direction is introduced to update the solutions, which improves the convergence rate. Their method is extended by~\cite{hamedani2021primal} for a non-bilinear min-max problem, which covers convex optimization with finitely many constraints as a special case. The method by~\cite{hamedani2021primal} is further modified by \cite{boob2022stochastic} by combining an extrapolation technique with a momentum gradient step and allowing stochastic gradient oracles. The same convergence rates are obtained by~\cite{hamedani2021primal} and \cite{boob2022stochastic} for deterministic convex constrained optimization. However, to achieve the claimed convergence rates, \cite{hamedani2021primal} does not require a known upper bound of the optimal Lagrangian while \cite{boob2022stochastic} does when the problem is not strongly convex. Besides, Hamedani and Aybat (2021) also introduce a back tracking technique to estimate the local Lipschitz constants which effectively improves the numerical performance.
Similar momentum gradient steps are also used for convex nested composite optimization~\cite{zhang2020optimal}.


Our method also uses a momentum gradient step similar to the ones in \cite{chambolle2011first,chambolle2016ergodic,hamedani2021primal,boob2022stochastic,zhang2020optimal} to update the constraint parameters $\by$ and uses the momentum gradient step with extrapolation in \cite{boob2022stochastic} to update the dual variable $\vlambda$. Different from \cite{hamedani2021primal,boob2022stochastic}, our method can be applied to \eqref{eq:SIP} with infinitely many constraints. In \cite{yao2023stochastic}, the authors studied a class of fairness constraints in machine learning that can be formulated as a stochastic composite constraint 
$$
d^*(\mathbb{E}_\xi\vG_i(\bx,\xi))=\max_{\by^i\in\cY^i}\left[(\by^i)^\top\mathbb{E}_\xi\vG_i(\bx,\xi)-d(\by^i)\right]\leq 0,
$$ 
where $\cY^i\subset\mathbb{R}_+^q$ is a convex  closed set, $\xi$ is a random variable, $d(\cdot)$ is a convex closed function whose Fenchel conjugate is $d^*(\cdot)$, and $\vG_i(\bx,\xi)$ is a vector-valued mapping that is convex in $\bx$ in each component. Problem \eqref{eq:SIP} covers their problem as a special case when $g_i(\bx,\by^i)=(\by^i)^\top\mathbb{E}\vG(\bx,\xi)-d(\by^i)$. First-order methods for SIP and SSIP have also be studied in \cite{wei2020inexact,Wei2020TheCA} where the algorithms need strong oracles per iteration. The differences between \cite{wei2020inexact,Wei2020TheCA} and this work is discussed in the end of previous section. We also refer readers to Table~\ref{table:1} for a clear comparison.






\section{Notations and Preliminaries}
\label{sec:notation}

We first introduce a few notations. Let $\cY:=\cY^1\times\cY^2\times\cdots\times\cY^m\subset\mathbb{R}^{mq}$, 
\begin{eqnarray}
\label{eq:gvecstarplus}
\vg_+^*(\bx)=(\max\{g_1^*(\bx),0\},\dots,\max\{g_m^*(\bx),0\})^\top,
\end{eqnarray}
\begin{eqnarray}
\label{eq:gvec}
\vg(\bx,\by)=(g_1(\bx,\by^{1}),\dots,g_m(\bx,\by^{m}))^\top,
\end{eqnarray}
\begin{align}
\label{eq:gradxygvec}
\nabla_x\vg(\bx,\by)=
\left[
\begin{array}{c}
\nabla_x g_1(\bx,\by^1)^\top \\
\nabla_x g_2(\bx,\by^2)^\top \\
 \vdots\\
 \nabla_x g_m(\bx,\by^m)^\top\\
\end{array}
\right]\in\mathbb{R}^{m\times p}
\end{align}
and
\begin{align}
\label{eq:gradygvec}
\nabla_y\vg(\bx,\by)=
\left[
\begin{array}{cccc}
\nabla_y g_1(\bx,\by^1)^\top & \mathbf{0} & \cdots & \mathbf{0}\\
\mathbf{0} & \nabla_y g_2(\bx,\by^2)^\top & \cdots & \vdots\\
 \mathbf{0} & \mathbf{0} & \ddots & \vdots\\
  \mathbf{0} & \mathbf{0} & \cdots & \nabla_y g_m(\bx,\by^m)^\top\\
\end{array}
\right]\in\mathbb{R}^{m\times mq}.
\end{align}
Similarly, using the notations from Assumption~\ref{assume:oracle}, the stochastic oracles of $\vg$, $\nabla_x\vg$ and $\nabla_y\vg$ can be defined as follows 
\begin{eqnarray}
\label{eq:Gvec}
\bG(\bx,\by,\xi)=(G_1(\bx,\by^{1},\xi),\dots,G_m(\bx,\by^{m},\xi))^\top,
\end{eqnarray}
\begin{align}
\label{eq:gradxyGvec}
\nabla_x\bG(\bx,\by,\xi)=
\left[
\begin{array}{c}
\nabla_xG_1(\bx,\by^1,\xi)^\top \\
\nabla_xG_2(\bx,\by^2,\xi)^\top \\
\vdots\\
\nabla_xG_m(\bx,\by^m,\xi)^\top\\
\end{array}
\right]\in\mathbb{R}^{m\times p}
\end{align}
and
\small
\begin{align}
\label{eq:gradyGvec}
\nabla_y\bG(\bx,\by,\xi)=
\left[
\begin{array}{cccc}
\nabla_yG_1(\bx,\by^1,\xi)^\top & \mathbf{0} & \cdots & \mathbf{0}\\
\mathbf{0} & \nabla_yG_2(\bx,\by^2,\xi)^\top & \cdots & \vdots\\
 \mathbf{0} & \mathbf{0} & \ddots & \vdots\\
  \mathbf{0} & \mathbf{0} & \cdots & \nabla_yG_m(\bx,\by^m,\xi)^\top\\
\end{array}
\right]\in\mathbb{R}^{m\times mq}.
\end{align}
\normalsize
Given any $(\bx',\by)\in\mathcal{X}\times\mathcal{Y}$, we define a vector-valued mapping $\ell_{\vg}(\cdot; \vx',\vy):\mathcal{X}\rightarrow \mathbb{R}^m$ as
\begin{equation}
\label{eq:ell}
\ell_{\vg}(\vx; \vx',\vy):=\vg(\vx',\vy)+\nabla_x \vg(\vx',\vy)(\vx-\vx').
\end{equation}
By the convexity of $g_i$ in $\bx$, the following inequality holds in a component-wise sense
\begin{equation}
\label{eq:ellg}
\vg(\bx,\by)\geq\ell_{\vg}(\vx; \vx',\vy), ~\forall \bx,\bx'\in\mathcal{X},~\by\in\mathcal{Y}.
\end{equation}
Using the oracles from Assumption~\ref{assume:oracle}, a stochastic estimation of $\ell_{\vg}(\vx; \vx',\vy)$ can be constructed, which is a stochastic mapping $\ell_{\vG}(\cdot,\xi;\vx',\vy):\mathcal{X}\rightarrow \mathbb{R}^m$ defined as
\begin{equation}
\label{eq:ell_stochastic}
\ell_{\vG}(\vx,\xi;\vx',\vy):=\vG(\vx',\vy,\xi)+\nabla_x\vG(\vx',\vy,\xi)(\vx-\vx').
\end{equation}
Given $\vlambda=(\lambda_1,\dots,\lambda_m)^\top\in\mathbb{R}_+^m$ and $\vy,\tilde\vy\in\cY$, we define $\|\vlambda\|_1$ as the $\ell_1$-norm and define
$$
\|\by-\tilde\vy\|_{\vlambda}^2:=\sum_{i=1}^m\lambda_i\|\by^i-\tilde\vy^i\|^2.
$$

\section{Algorithm}
\label{sec:alg}
Recall the Lagrangian function $\phi(\bx,\vlambda,\by)$ in \eqref{eq:Lagrangian}. Suppose $(\vx_k,\vy_k,\vlambda_k)$ are the primal and dual solutions of \eqref{eq:Lagrangian} generated at the $k$th iteration of an algorithm. To generate the next dual solutions, the existing approach (e.g. \cite{nemirov09,nemirovski-2005-prox,hamedani2021primal}) for a saddle-point problem typically will update $(\vy_k,\vlambda_k)$ together using the gradient or the momentum of the gradient of $\phi$ with respect to $\vy$ and $\vlambda$. However, the convergence of such a method to the global optimal solution becomes unclear because of the non-concavity of $\phi$ in $(\vy,\vlambda)$. Note that $\phi$ is still concave in $\vy$ while $\vlambda$ is fixed, and vice versa, so a potential approach is to update $\vy$ and $\vlambda$ sequentially.  This leads to the algorithm in Algorithm~\ref{alg:PDHG}, which we will explain in details next. 

In particular, one can update $\vy_k^i$ for each $i$ first by taking a gradient step with a momentum term towards maximizing $g_i(\bx_k,\by^i)$ over $\by^i\in \cY^i$, namely, 
\begin{eqnarray*}
\vy_{k+1}^i = \argmin_{\by^i\in\cY^i} -\left\langle \nabla_yg_i(\bx_k,\by_k^i)+\theta_k \left[\nabla_yg_i(\bx_k,\by_k^i)-\nabla_yg_i(\bx_{k-1},\by_{k-1}^i)\right],\by\right\rangle+\frac{\sigma_k}{2}\|\by^i-\by_k^i\|^2.
\end{eqnarray*}
Here, $\theta_k\geq0$ is a momentum parameter and $\sigma_k\geq0$ is a step size. Note that $\nabla_yg_i(\bx_k,\by_k^i)-\nabla_yg_i(\bx_{k-1},\by_{k-1}^i)$ can be viewed as the momentum of the $\nabla_yg_i$ between two consecutive iterations. When $\theta_k=0$, the updating equation above is just the gradient ascend step to maximize $g_i(\bx_k,\by^i)$ over $\by^i\in \cY^i$. When $\theta_k>0$, the updating direction $\nabla_yg_i(\bx_k,\by_k^i)+\theta_k \left[\nabla_yg_i(\bx_k,\by_k^i)-\nabla_yg_i(\bx_{k-1},\by_{k-1}^i)\right]$ can be used to further accelerate the gradient methods as known in literature~\cite{hamedani2021primal,mokhtari2020unified,chambolle2016ergodic,chambolle2011first,zhang2020optimal}. This update can be performed for $i=1,\dots,m$ independently and is what Line 4 and 5 do in Algorithm~\ref{alg:PDHG}. 

Similarly, one can update $\vlambda_k$ by taking a gradient step with a momentum term towards maximizing $\phi(\bx_k,\by_{k+1},\vlambda)$ over $\vlambda\in\mathbb{R}^m_+$, namely, 
\begin{align}
\label{eq:updatelambda_old}
\vlambda_{k+1} = \argmin_{\vlambda \geq \textbf{0}} - \left\langle \vg(\vx_k,\vy_{k+1})+\theta_k\left[\vg(\vx_k,\vy_k)-\vg(\vx_{k-1},\vy_k)\right], \vlambda\right\rangle+\frac{\gamma_k}{2}\|\vlambda-\vlambda_k\|^2.
\end{align}
Similar to the updating step for $\vy^i$, $\theta_k\geq0$ is a momentum parameter and $\gamma_k\geq0$ is a step size. In our algorithm, we implement an variant of \eqref{eq:updatelambda_old} for updating $\vlambda_k$, where $\vg$ is replaced by its linear approximation $\ell_{\vg}$ in \eqref{eq:ell}. In particular, we update $\vlambda$ using
\begin{align}
\label{eq:updatelambda_new}
\vlambda_{k+1} = \argmin_{\vlambda \geq \textbf{0}} - \left\langle \ell_{\vg}(\bx_k;\bx_{k-1},\by_{k+1})+\theta_k \left[\ell_{\vg}(\bx_k;\bx_{k-1},\by_k)-\ell_{\vg}(\bx_{k-1};\bx_{k-2},\by_k)\right], \vlambda\right\rangle+\frac{\gamma_k}{2}\|\vlambda-\vlambda_k\|^2.
\end{align}
This method is originally proposed by  \cite{boob2022stochastic} for optimization with finitely many constraints and, here, we generalize it for \eqref{eq:SIP}. This is the updates performed in Lines 6 and 7 do in Algorithm~\ref{alg:PDHG}.


Once we have updated $(\vy_{k+1},\vlambda_{k+1})$, we can generate $\vx_{k+1}$ by a gradient step from $\vx_k$ using the gradient of $\phi$ with respect to $\bx$ at $(\vx_k,\vy_{k+1},\vlambda_{k+1})$. This is performed in Line 8 of Algorithm~\ref{alg:PDHG} where $\tau_k$ is a step size. After $K$ iterations, the algorithm is terminated and returns the weighted average of $\vx_{k+1}$ for $k=0,\dots,K-1$ using $t_k\geq0$ as the weight. 

As shown in  \cite{hamedani2021primal} and \cite{boob2022stochastic} for finitely constrained problems, there are subtle differences between \eqref{eq:updatelambda_old} and  \eqref{eq:updatelambda_new} in their impacts to the algorithm's convergence analysis. When \eqref{eq:updatelambda_old} is applied, the convergence analysis in \cite{hamedani2021primal}  either assumes an upper bound $\Lambda$ of $\|\vlambda^*\|$ is known and $\vlambda$ is projected to a bounded domain $\{\vlambda\geq \textbf{0}| \|\vlambda\|\leq \Lambda \}$ in each iteration, or uses induction to show that $\vlambda_{k}$ is bounded for all iterations and uses a backtracking step to select $\tau_k$. On the contrary, when \eqref{eq:updatelambda_new} is applied, the algorithm in \cite{boob2022stochastic} does not require knowing $\Lambda$ but only achieve the optimal convergence rate when $\Lambda$ is known.

\begin{algorithm}[t]
\caption{Accelerated gradient method for SIP (AGSIP)}\label{alg:PDHG}
\begin{algorithmic}[1]
\STATE{ \textbf{Inputs:} $\bx_0\in\cX$, $\by_0\in\cY$, $\vlambda_0\in\mathbb{R}_+^m$, $K\in\mathbb{Z}_+$,
$\sigma_k$, $\gamma_k$, $\tau_k$, $\theta_k$ and $t_k\geq0$ for $k=0,1,\dots,K-1$.}
\STATE{ $\bx_{-2}=\bx_{-1}=\bx_{0}$, $\by_{-1}=\by_{0}$}
\FOR{$k=0,1\cdots,K-1$}
\STATE{  $\vu_k^i = \nabla_yg_i(\bx_k,\by_k^i)+\theta_k \left[\nabla_yg_i(\bx_k,\by_k^i)-\nabla_yg_i(\bx_{k-1},\by_{k-1}^i)\right]$, $i=1,\dots,m$}
\STATE{  $\by_{k+1}^i = \argmin_{\by^i\in\cY^i} -\langle\vu_k^i,\by^i\rangle+\frac{\sigma_k}{2}\|\by^i-\by_k^i\|^2$, $i=1,\dots,m$}
\STATE{  $\vv_k = \ell_{\vg}(\bx_k;\bx_{k-1},\by_{k+1})+\theta_k \left[\ell_{\vg}(\bx_k;\bx_{k-1},\by_k)-\ell_{\vg}(\bx_{k-1};\bx_{k-2},\by_k)\right]$}
\STATE{  $\vlambda_{k+1} = \argmin_{\vlambda \geq \textbf{0}} - \left\langle \vv_k, \vlambda\right\rangle+\frac{\gamma_k}{2}\|\vlambda-\vlambda_k\|^2$}
\STATE{  $\bx_{k+1} = \argmin_{\bx\in\cX} \left\langle \nabla f(\vx_k)+\nabla_x\vg(\vx_k,\vy_{k+1})^\top \vlambda_{k+1}, \bx\right\rangle+\frac{\tau_k}{2}\|\bx-\bx_k\|^2$}
\ENDFOR
\RETURN $\bar\bx_K=\frac{\sum_{k=0}^{K-1}t_k\bx_{k+1}}{\sum_{k=0}^{K-1}t_k}$
\end{algorithmic}
\end{algorithm}

When only the stochastic oracles are available, we can still implement Algorithm~\ref{alg:PDHG} by replacing the gradients and function values in each iteration using the corresponding stochastic estimators in \eqref{eq:Gvec}, \eqref{eq:gradxyGvec}, \eqref{eq:gradyGvec} and \eqref{eq:ell_stochastic}. However, for the theoretical analysis to go through, it is critical to ensure the independence among those stochastic estimators used in iteration $k$ conditioning on $(\vx_l,\vy_l,\vlambda_l)$ for $l=0,\dots,k$. To do so, we samples six independent samples of $\xi$, denoted by $\xi_k^i$ $i=1,2,\dots,6$, and use them to construct the stochastic oracles for $\nabla_y\vg(\vx_k,\vy_k)$, $\nabla_y\vg(\vx_{k-1},\vy_{k-1})$, $\ell_{\vg}(\bx_k;\bx_{k-1},\by_{k+1})$, $\ell_{\vg}(\bx_k;\bx_{k-1},\by_k)$, $\ell_{\vg}(\bx_{k-1};\bx_{k-2},\by_k)$,  and $\nabla f(\vx_k)+\nabla_x\vg(\vx_k,\vy_{k+1})^\top \vlambda_{k+1}$, respectively. We want to point out that we generate six samples of $\xi$ in each iteration only to index the stochastic oracles more clearly. All of our results can be still obtained by setting 
\begin{equation}
\label{eq:threesamples}
\xi_k^1=\xi_k^2 \text{ and }\xi_k^3=\xi_k^4=\xi_k^5
\end{equation}
in each iteration, so we actually need only three samples, i.e., $\xi_k^1$, $\xi_k^3$ and $\xi_k^6$. We will only analyze the convergence property of Algorithm~\ref{alg:SPDHG} because it covers Algorithm~\ref{alg:PDHG} as a special case when deterministic estimators are used.

\begin{algorithm}[t]
\caption{Stochastic gradient method for SIP (SGSIP) }\label{alg:SPDHG}
\begin{algorithmic}[1]
\STATE{ \textbf{Inputs:} $\bx_0\in\cX$, $\by_0\in\cY$ and $\vlambda_0\in\mathbb{R}_+^m$}
\STATE{ $\bx_{-2}=\bx_{-1}=\bx_{0}$, $\by_{-1}=\by_{0}$}
\FOR{$k=0,1\cdots,$}
\STATE{ Sample $\xi_k^i$, $i=1,\dots,6$ from the distribution of $\xi$.}
\STATE{  $\vu_k^i = \nabla_yG_i(\bx_k,\by_k^i,\xi_k^1)+\theta_k \left[\nabla_yG_i(\bx_k,\by_k^i,\xi_k^1)-\nabla_yG_i(\bx_{k-1},\by_{k-1}^i,\xi_{k}^2)\right]$, $i=1,\dots,m$}
\STATE{  $\by_{k+1}^i = \argmin_{\by^i\in\cY^i} -\langle\vu_k^i,\by^i\rangle+\frac{\sigma_k}{2}\|\by^i-\by_k^i\|^2$, $i=1,\dots,m$}
\STATE{  $\vv_k = \ell_{\vG}(\bx_k,\xi_k^3;\bx_{k-1},\by_{k+1})+\theta_k \left[\ell_{\vG}(\bx_k,\xi_k^4;\bx_{k-1},\by_k)-\ell_{\vG}(\bx_{k-1},\xi_k^5;\bx_{k-2},\by_k)\right]$}
\STATE{  $\vlambda_{k+1} = \argmin_{\vlambda \geq \textbf{0}} - \left\langle \vv_k, \vlambda\right\rangle+\frac{\gamma_k}{2}\|\vlambda-\vlambda_k\|^2$}
\STATE{  $\bx_{k+1} = \argmin_{\bx\in\cX} \left\langle \nabla F(\vx_k,\xi_k^6)+\nabla_x\vG(\vx_k,\vy_{k+1},\xi_k^6)^\top \vlambda_{k+1}, \bx\right\rangle+\frac{\tau_k}{2}\|\bx-\bx_k\|^2$}
\ENDFOR
\RETURN $\bar\bx_K=\frac{\sum_{k=0}^{K-1}t_k\bx_{k+1}}{\sum_{k=0}^{K-1}t_k}$
\end{algorithmic}
\end{algorithm}

\section{Bounding Primal-Dual Gap}
In this section, we present the main results on the convergence properties of Algorithms~\ref{alg:PDHG} and~\ref{alg:SPDHG}. 
Let $\bz=(\bx,\vlambda,\by)$ and $\bz_k=(\bx_k,\vlambda_k,\by_k)$ for $k\geq0$. Following \cite{zhang2020optimal}, we write 
$\phi(\bx_{k+1},\vlambda,\by) - \phi(\bx,\vlambda_{k+1},\by_{k+1})$ as the summation of three gaps and derive an upper bound on gap. In particular, according to \eqref{eq:Lagrangian}, we have
\begin{eqnarray}
\label{eq:Qs}
       ~~~\phi(\bx_{k+1},\vlambda,\by) - \phi(\bx,\vlambda_{k+1},\by_{k+1})
       &=&  Q_2(\bz_{k+1},\bz) + Q_1(\bz_{k+1},\bz)+ Q_0(\bz_{k+1},\bz),
\end{eqnarray}
where 
\begin{eqnarray}
\label{eq:Qs2}
        Q_2(\bz_{k+1},\bz)&:=&~\phi(\bx_{k+1},\vlambda,\by) - \phi(\bx_{k+1},\vlambda,\by_{k+1}) \\\nonumber
        &=& ~\vlambda^{\top}(\vg(\bx_{k+1},\by)-\vg(\bx_{k+1},\by_{k+1})) \\\label{eq:Qs1}
         Q_1(\bz_{k+1},\bz)&:=&~\phi(\bx_{k+1},\vlambda,\by_{k+1}) - \phi(\bx_{k+1},\vlambda_{k+1},\by_{k+1})\\\nonumber
         &=&~(\vlambda-\vlambda_{k+1})^{\top}\vg(\bx_{k+1},\by_{k+1})\\\label{eq:Qs0}
        Q_0(\bz_{k+1},\bz)&:=& ~\phi(\bx_{k+1},\vlambda_{k+1},\by_{k+1}) - \phi(\bx,\vlambda_{k+1},\by_{k+1})\\\nonumber
       & =&~ f(\bx_{k+1})-f(\bx)+\vlambda_{k+1}^\top(\vg(\bx_{k+1},\by_{k+1})-\vg(\bx,\by_{k+1})).
\end{eqnarray}
The convergence properties of Algorithm~\ref{alg:SPDHG} are then derived by bounding $\phi(\bx_{k+1},\vlambda,\by) - \phi(\bx,\vlambda_{k+1},\by_{k+1})$ from above by summing up some upper bounds of $Q_0$, $Q_1$ and $Q_2$, and choosing parameters $\sigma_k$, $\gamma_k$, $\tau_k$, $\theta_k$ and $t_k\geq 0$ for $k=0,1,\dots$ appropriately according to the smoothness of the problem, strong convexity/concavity of the problem, and the types (deterministic or stochastic) of oracle used.

The following lemma is a conclusion from Lemmas 2.1 and 6.1 in~\cite{nemirov09} and is needed to obtain the upper bounds of $Q_0$, $Q_1$ and $Q_2$. Its proof is provided for the sack of completeness.
\begin{lemma}
\label{eq:threepoint}
Given a closed convex set $\mathcal{W}\subset\mathbb{R}^d$ and a point $\vw_0\in\mathcal{W}$, a sequence $\{\vw_k\}_{k\geq0}$ is generated by
$$
\vw_{k+1} = \argmin_{\vw\in\mathcal{W}}\left\langle \vs_k+\vdelta_k, \vw\right\rangle+\frac{\eta_k}{2}\|\vw-\vw_k\|^2,~k=0,\dots,
$$
where $\eta_k>0$, $\vs_k\in\mathbb{R}^d$ and $\vdelta_k\in\mathbb{R}^d$. It holds, for any $\vw\in\mathcal{W}$, that
\begin{eqnarray}
\label{eq:threepoint_sk1}
\left\langle \vs_k+\vdelta_k, \vw_{k+1}-\vw\right\rangle+\frac{\eta_k}{2}\|\vw-\vw_{k+1}\|^2
\leq -\frac{\eta_k}{2}\|\vw_{k+1}-\vw_k\|^2+
\frac{\eta_k}{2}\|\vw-\vw_k\|^2.
\end{eqnarray}
Additionally, given $\tilde\vw_0=\vw_0$, a sequence $\{\tilde\vw_k\}_{k\geq0}$ is generated by
$$
\tilde\vw_{k+1} = \argmin_{\vw\in\mathcal{W}}-\left\langle \vdelta_k, \vw\right\rangle+\frac{\eta_k}{2}\|\vw- \tilde\vw_k\|^2,~k=0,1,\dots.
$$
It holds, for any $\vw\in\mathcal{W}$, that
\begin{eqnarray}
\label{eq:threepoint_sk4}
&&\left\langle \vs_k, \vw_{k+1}-\vw\right\rangle+\frac{\eta_k}{2}\|\vw-\vw_{k+1}\|^2+\frac{\eta_k}{2}\|\vw-\tilde\vw_{k+1}\|^2\\\nonumber
&\leq &-\frac{\eta_k}{4}\|\vw_{k+1}-\vw_k\|^2+
\frac{\eta_k}{2}\|\vw-\vw_k\|^2+\frac{\eta_k}{2}\|\vw-\tilde\vw_k\|^2\\\nonumber
&&
+\left\langle \vdelta_k, \tilde\vw_k-\vw_k\right\rangle+ \frac{3}{2\eta_k}\|\vdelta_k\|^2.
\end{eqnarray}
\end{lemma}
\begin{proof}
Inequality \eqref{eq:threepoint_sk1} is a direct outcome of the $\eta_k$-strong convexity of $\left\langle \vs_k+\vdelta_k, \cdot\right\rangle+\frac{\eta_k}{2}\|\cdot-\vw_k\|^2$ and the definition of $\vw_{k+1}$. Moreover, the Young's inequality gives
\begin{eqnarray}
\nonumber
&&\left\langle \vs_k+\vdelta_k, \vw_{k+1}-\vw\right\rangle\\\nonumber
&=&
\left\langle \vs_k, \vw_{k+1}-\vw\right\rangle
+\left\langle \vdelta_k, \vw_k-\vw\right\rangle
+\left\langle \vdelta_k, \vw_{k+1}-\vw_k\right\rangle\\\label{eq:threepoint_sk2}
&\geq&
\left\langle \vs_k, \vw_{k+1}-\vw\right\rangle
+\left\langle \vdelta_k, \vw_k-\vw\right\rangle
-\frac{1}{\eta_k}\|\vdelta_k\|^2
-\frac{\eta_k}{4}\|\vw_{k+1}-\vw_k\|^2.
\end{eqnarray}
According to Lemma 2.1 in \cite{nemirov09}, we have
\begin{eqnarray}
\label{eq:threepoint_sk3}
\frac{\eta_k}{2}\|\vw-\tilde\vw_{k+1}\|^2
\leq \frac{1}{2\eta_k}\|\vdelta_k\|^2-
\left\langle \vdelta_k, \vw-\tilde\vw_k\right\rangle+\frac{\eta_k}{2}\|\vw-\tilde\vw_k\|^2
\end{eqnarray}
for any $\vw\in\mathcal{W}$. Adding \eqref{eq:threepoint_sk1}, \eqref{eq:threepoint_sk2} and \eqref{eq:threepoint_sk3} and organizing terms give \eqref{eq:threepoint_sk4}.
\qedsymbol{}
\end{proof}

The convergence analysis of Algorithm~\ref{alg:SPDHG} in the stochastic case will involve stochastic errors, which are the differences between the stochastic oracles and their deterministic counterparts. We denote those errors involved in iteration $k$ of Algorithm~\ref{alg:SPDHG}  as follows. 
\begin{equation}
\label{eqn:deltas}
\begin{split}
\vdelta_k^{i,1}=&\nabla_yG_i(\bx_k,\by_k^i,\xi_k^1)-\nabla_yg_i(\bx_k,\by_k^i),~i=1,\dots,m\\
\vDelta_k^{1}=&\nabla_y\vG(\bx_k,\by_k,\xi_k^1)-\nabla_y\vg(\bx_k,\by_k)\\
\vdelta_k^{i,2}=&\nabla_yG_i(\bx_{k-1},\by_{k-1}^i,\xi_k^2)-\nabla_yg_i(\bx_{k-1},\by_{k-1}^i),~i=1,\dots,m\\
\vDelta_k^{2}=&\nabla_y\vG(\bx_{k-1},\by_{k-1},\xi_k^2)-\nabla_y\vg(\bx_{k-1},\by_{k-1})\\
\vdelta_k^{3}=&\vG(\vx_{k-1},\vy_{k+1},\xi_k^3)-\vg(\vx_{k-1},\vy_{k+1})\\
\vDelta_k^{3}=&\nabla_x\vG(\vx_{k-1},\vy_{k+1},\xi_k^3)-\nabla_x\vg(\vx_{k-1},\vy_{k+1})\\
\vdelta_k^{4}=&\vG(\vx_{k-1},\vy_{k},\xi_k^4)-\vg(\vx_{k-1},\vy_{k})\\
\vDelta_k^{4}=&\nabla_xG(\vx_{k-1},\vy_{k},\xi_k^4)-\nabla_x\vg(\vx_{k-1},\vy_{k})\\
\vdelta_k^{5}=&\vG(\vx_{k-2},\vy_{k},\xi_k^5)-\vg(\vx_{k-2},\vy_{k})\\
\vDelta_k^{5}=&\nabla_xG(\vx_{k-2},\vy_{k},\xi_k^5)-\nabla_x\vg(\vx_{k-2},\vy_{k})\\
\vdelta_k^{6}=&\nabla F(\vx_{k},\xi_k^6)-\nabla f(\vx_{k})\\
\vDelta_k^{6}=&\nabla_xG(\vx_{k},\vy_{k+1},\xi_k^6)-\nabla_x \vg(\vx_{k},\vy_{k+1}).
\end{split}
\end{equation}
Here, $\vdelta_k^i$'s are vectors while $\vDelta_k^i$'s are matrices. The superscript $i$ of $\vdelta_k^i$ and $\vDelta_k^i$ indicates that they depend on sample $\xi_k^i$ for $i=1,\dots,6$. Again, we introduce six samples $\xi_k^i$, $i=1,\dots,6$, in each iteration mainly to index the error terms in \eqref{eqn:deltas} more clearly through superscripts. Our theoretical analysis remains valid even with just three samples satisfying \eqref{eq:threesamples}.

Our analysis will involve the following momenta of the gradients 
\begin{align}
\label{eq:qk}
\bar\bq_{k} =& \ell_{\vg}(\vx_{k};\vx_{k-1},\vy_{k})-\ell_{\vg}(\vx_{k-1};\vx_{k-2},\vy_{k}),\\
\label{eq:pk}
\bar\bp_{k} =&  \nabla_y g_i(\bx_{k},\by_{k}^i)-\nabla_y g_i(\bx_{k-1},\by_{k-1}^i),~i=1,\dots,m,\\
\label{eq:Pk}
\bar\bP_{k} =&  \nabla_y \vg(\bx_{k},\by_{k})-\nabla_y \vg(\bx_{k-1},\by_{k-1})
\end{align}
for $k=0,1,\dots$. Here, $\bar\bq_{k}$ is the linear approximation of $\vg(\bx_{k},\by_{k})-\vg(\bx_{k-1},\by_k)$, a momentum of the gradient of $\phi$ with respect to $\vlambda$, while $\bar\bp_{k}$ and $\bar\bP_{k}$ are the momenta of $\nabla_yg_i$ and $\nabla_y\vg$, respectively. Since $\bx_{-2}=\bx_{-1}=\bx_{0}$ and $\by_{-1}=\by_{0}$ in our algorithms, we have
\begin{equation}
\label{eq:barqpP0}
\bar\bq_{0}=\mathbf{0},~ \bar\bp_{0}=\mathbf{0} \text{ and }\bar\bP_{0}=\mathbf{0}.
\end{equation}
The stochastic estimations of $\bar\bq_{k}$, $\bar\bp_{k}$ and $\bar\bP_{k}$ are, respectively,  
\begin{align}
\label{eq:barqk}
\bq_{k} =& \ell_{\vg}(\vx_{k},\xi_k^4;\vx_{k-1},\vy_{k})-\ell_{\vG}(\vx_{k-1},\xi_k^5;\vx_{k-2},\vy_k)\\
\label{eq:barpk}
\bp_{k} =& \nabla_yG_i(\bx_k,\by_k^i,\xi_k^1)-\nabla_yG_i(\bx_{k-1},\by_{k-1}^i,\xi_k^2),~i=1,\dots,m,\\
\label{eq:barPk}
\bP_{k} =& \nabla_y\vG(\bx_k,\by_k,\xi_k^1)-\nabla_y\vG(\bx_{k-1},\by_{k-1},\xi_k^2)
\end{align}
for $k=0,1,\dots$.

\subsection{An upper bound of $Q_2$}
We first define an auxiliary sequence $\{\tilde \vy_k^i\}_{k\geq0}$ for $i=1,\dots,m$, where $\tilde\vy_0^i=\vy_0^i$ and $\tilde\vy_k^i$ for $k\geq1$ are generated as 
\begin{eqnarray}
\label{eq:tildeyk}
\tilde\vy_{k+1}^i=\argmin_{\vy^i}\left\langle \vdelta_k^{i,1}+\theta_k \left(\vdelta_k^{i,1}-\vdelta_k^{i,2}\right),\vy^i\right\rangle+\frac{\sigma_k}{2}\|\by^i-\tilde\by_k^i\|^2,
\end{eqnarray}
where $\vdelta_k^{i,1}$ and $\vdelta_k^{i,2}$ are defined in \eqref{eqn:deltas} and $\sigma_k$ is as in Algorithm~\ref{alg:SPDHG}.
For $i=1,2,\dots,m$, applying conclusion \eqref{eq:threepoint_sk4} in Lemma~\ref{eq:threepoint} to Line 6 of Algorithm~\ref{alg:SPDHG} with the following instantiation
\begin{eqnarray}
\label{eq:threepoint_instance_yki1}
&\vs_k=-\nabla_yg_i(\bx_k,\by_k^i)-\theta_k \left[\nabla_yg_i(\bx_k,\by_k^i)-\nabla_yg_i(\bx_{k-1},\by_{k-1}^i)\right],\\
\label{eq:threepoint_instance_yki2}
&\vdelta_k=-\vdelta_k^{i,1}-\theta_k \left(\vdelta_k^{i,1}-\vdelta_k^{i,2}\right),\\
\label{eq:threepoint_instance_yki3}
&\mathcal{W}=\mathcal{Y}^i,~
\vw_0= \vy_0^i \text{ and }
\eta_k=\sigma_k,
\end{eqnarray}
we have, for any $\vy^i\in\mathcal{Y}^i$, that
\begin{eqnarray}
\label{eq:threepoint_yki}
\left\langle \vs_k, \vy_{k+1}^i-\vy^i\right\rangle&\leq &-\frac{\sigma_k}{4}\|\vy_{k+1}^i-\vy_k^i\|^2+
\frac{\sigma_k}{2}\|\vy^i-\vy_k^i\|^2-\frac{\sigma_k}{2}\|\vy^i-\vy_{k+1}^i\|^2\\\nonumber
&&+\frac{\sigma_k}{2}\|\vy^i-\tilde\vy_k^i\|^2-\frac{\sigma_k}{2}\|\vy^i-\tilde\vy_{k+1}^i\|^2
+\left\langle \vdelta_k, \tilde\vy_k^i-\vy_k^i\right\rangle+\frac{3}{2\sigma_k}\|\vdelta_k\|^2.
\end{eqnarray}

By the $\mu_y$-strong concavity of $g_i$ in $\by^i$, it holds, for any $\by^i \in\cY^i$, that
\begin{eqnarray}
\nonumber
  &&g_i(\bx_{k+1},\by^i)-g_i(\bx_{k+1},\by_{k+1}^i)\\\nonumber
  &\leq& -\langle\nabla_y g_i(\bx_{k+1},\by_{k+1}^i), \by_{k+1}^i-\by^i\rangle-\frac{\mu_y}{2}\|\by_{k+1}^i-\by^i\|^2\\\label{eqn:tri_y_1}
&= &\left\langle \vs_k, \vy_{k+1}^i-\vy^i\right\rangle+\langle \nabla_yg_i(\bx_k,\by_k^i)-\nabla_y g_i(\bx_{k+1},\by_{k+1}^i),\by_{k+1}^i-\by^i \rangle\\\nonumber
    && + \theta_k\langle \nabla_y g_i(\bx_{k},\by_{k}^i)-\nabla_y g_i(\bx_{k-1},\by_{k-1}^i),\by_{k}^i-\by^i \rangle\\\nonumber
    && +\theta_k\langle \nabla_y g_i(\bx_{k},\by_{k}^i)-\nabla_y g_i(\bx_{k-1},\by_{k-1}^i),\by_{k+1}^i-\by_{k}^i\rangle
    -\frac{\mu_y}{2}\|\by_{k+1}^i-\by^i\|^2,
\end{eqnarray}
where the equality holds because of the instance of $\vs_k$ in \eqref{eq:threepoint_instance_yki1}.
We next bound the term $ \theta_k\langle \nabla_y g_i(\bx_{k},\by_{k}^i)-\nabla_y g_i(\bx_{k-1},\by_{k-1}^i),\by_{k+1}^i-\by_{k} ^i\rangle$ on the right-hand side of \eqref{eqn:tri_y_1} as follows
\begin{align}
\nonumber
   & \theta_k\langle \nabla_y g_i(\bx_{k},\by_{k}^i)-\nabla_y g_i(\bx_{k-1},\by_{k-1}^i),\by_{k+1}^i-\by_{k} ^i\rangle\\\nonumber
    \leq& \theta_k\|\nabla_y g_i(\bx_{k},\by_{k}^i)-\nabla_y g_i(\bx_{k-1},\by_{k-1}^i)\|\|\vy_{k+1}^i-\vy_k^i\|\\\nonumber
    \leq& \theta_k\left(
    L^{yx}_g\|\vx_{k}-\vx_{k-1}\|+L^{yy}_g\|\vy_{k}^i-\vy_{k-1}^i\|
    +H_g^y\right)\|\vy_{k+1}^i-\vy_k^i\|\\\nonumber
\leq& \frac{2\theta_k^2}{\sigma_k}\left(
    L^{yx}_g\|\vx_{k}-\vx_{k-1}\|+L^{yy}_g\|\vy_{k}^i-\vy_{k-1}^i\| \right)^2+\frac{\sigma_k}{8}\|\vy_{k+1}^i-\vy_{k}^i\|^2 +\theta_kH_g^y\|\vy_{k+1}^i-\vy_k^i\|      \\\label{eq:eqn:tri_y_0}
      \leq& \frac{4(\theta_kL^{yx}_g)^2}{\sigma_k}\|\vx_{k}-\vx_{k-1}\|^2
       +\frac{4(\theta_kL^{yy}_g)^2}{\sigma_k}\|\vy_{k}^i-\vy_{k-1}^i\|^2 +\frac{\sigma_k}{8}\|\vy_{k+1}^i-\vy_{k}^i\|^2+\theta_kH_g^y\|\vy_{k+1}^i-\vy_k^i\|   ,
\end{align}
where the first inequality is by Cauchy-Schwarz inequality, the second by Assumption~\ref{assume:continuity}.c,  the third  by Young's inequality, and the last by Jensen's inequality. 


Recall \eqref{eq:pk} and \eqref{eq:threepoint_instance_yki2}. Applying \eqref{eq:threepoint_yki} and \eqref{eq:eqn:tri_y_0}  to \eqref{eqn:tri_y_1} and organizing terms lead to
\begin{align}
\nonumber
&g_i(\bx_{k+1},\by^i)-g_i(\bx_{k+1},\by_{k+1}^i)\\\nonumber
\leq & -\langle \bar\bp_{k+1},\by_{k+1}^i-\by^i \rangle+ \theta_k\langle \bar\bp_{k},\by_{k}^i-\by^i \rangle -\frac{\mu_y}{2}\|\by_{k+1}^i-\by^i\|^2\\\nonumber
    & +\frac{\sigma_k}{2}\left[ \|\vy^i-\vy_k^i\|^2-\|\vy^i-\vy_{k+1}^i\|^2+ \|\vy^i-\tilde\vy_k^i\|^2-\|\vy^i-\tilde\vy_{k+1}^i\|^2\right]\\\nonumber
   &+\frac{4(\theta_kL^{yx}_g)^2}{\sigma_k}\|\vx_{k}-\vx_{k-1}\|^2
       +\frac{4(\theta_kL^{yy}_g)^2}{\sigma_k}\|\vy_{k}^i-\vy_{k-1}^i\|^2-\frac{\sigma_k}{8}\|\vy_{k+1}^i-\vy_{k}^i\|^2+\theta_kH_g^y\|\vy_{k+1}^i-\vy_k^i\|\\\label{eqn:tri_y_2}
    &+ \left\langle \vdelta_k^{i,1}+\theta_k \left(\vdelta_k^{i,1}-\vdelta_k^{i,2}\right),\vy_k^i-\tilde\vy_k^i\right\rangle+\frac{3}{2\sigma_k}\left\|\vdelta_k^{i,1}+\theta_k \left(\vdelta_k^{i,1}-\vdelta_k^{i,2}\right)\right\|^2.
\end{align}

By Cauchy-Schwarz inequality, we have 
\begin{equation}
\label{eq:changeylambda}
\sum_{i=1}^m\lambda_i\|\vy_{k+1}^i-\vy_{k}^i\|
\leq \|\vlambda\|_1^{\frac{1}{2}}\|\vy_{k+1}-\vy_{k}\|_\vlambda.
\end{equation}
Recall \eqref{eq:Pk} and 
the definition of $Q_2$ in \eqref{eq:Qs2}. Multiplying \eqref{eqn:tri_y_2} by $\lambda_i\geq0$ and summing it up for $i=1,\dots,m$ give
\begin{align}
\label{eq:Q2ub1_old}
&Q_2(\bz_{k+1},\bz)   = \vlambda^\top(\vg(\bx_{k+1},\by)-\vg(\bx_{k+1},\by_{k+1})) \\\nonumber
\leq & -\vlambda^\top\bar\bP_{k+1}(\by_{k+1}-\by)+ \theta_k\vlambda^\top\bar\bP_k(\by_{k}-\by)-\frac{\mu_y}{2}\|\by_{k+1}-\by\|_{\vlambda}^2 \\\nonumber
& +\frac{\sigma_k}{2}\left[ \|\vy-\vy_k\|_{\vlambda}^2-\|\vy-\vy_{k+1}\|_{\vlambda}^2 +\|\vy-\tilde\vy_k\|_{\vlambda}^2-\|\vy-\tilde\vy_{k+1}\|_{\vlambda}^2\right]\\\nonumber
& +\frac{4(\theta_kL^{yx}_g)^2}{\sigma_k}\|\vlambda\|_1\|\vx_{k}-\vx_{k-1}\|^2
       +\frac{4(\theta_kL^{yy}_g)^2}{\sigma_k}\|\vy_{k}-\vy_{k-1}\|_{\vlambda}^2\\\nonumber
&-\frac{\sigma_k}{8}\|\vy_{k+1}-\vy_{k}\|_{\vlambda}^2+\theta_kH_g^y \|\vlambda\|_1^{\frac{1}{2}}\|\vy_{k+1}-\vy_{k}\|_\vlambda\\\nonumber
&+ \vlambda^\top\left[\vDelta_k^1+\theta_k \left(\vDelta_k^1-\vDelta_k^2\right)\right](\vy_k-\tilde\vy_k)+\frac{3}{2\sigma_k}\left\|\vDelta_k^1+\theta_k \left(\vDelta_k^1-\vDelta_k^2\right)\right\|_{\vlambda}^2
\end{align}
for any $\vlambda\geq0$ and any $\vy$, where \eqref{eq:changeylambda} is applied to obtain term $\theta_kH_g^y \|\vlambda\|_1^{\frac{1}{2}}\|\vy_{k+1}-\vy_{k}\|_\vlambda$ on the right-hand side.


\subsection{An upper bound of $Q_1$}
We first define an auxiliary sequence $\{\tilde\vlambda_k\}_{k\geq0}$, where $\tilde\vlambda_0^i=\vlambda_0$ and $\tilde\vlambda_k$ for $k\geq1$ are generated as 
\begin{align}
\label{eq:tildelambdak}
\tilde\vlambda_{k+1}=\argmin_{\vlambda\geq\mathbf{0}}&\left\langle\ell_{\vG}(\vx_k,\xi_k^3;\vx_{k-1},\vy_{k+1})-\ell_{\vg}(\vx_k;\vx_{k-1},\vy_{k+1})+\theta_k( \bq_{k}-\bar\bq_{k}),\vlambda\right\rangle+\frac{\gamma_k}{2}\|\vlambda-\tilde\vlambda_k\|^2,
\end{align}
where $\ell_{\vg}$, $\ell_{\vG}$,  $\bq_{k}$ and $\bar\bq_{k}$ are defined in \eqref{eq:ell}, \eqref{eq:ell_stochastic},  \eqref{eq:qk} and \eqref{eq:barqk}, respectively, and $\gamma_k$ is as in Algorithm~\ref{alg:SPDHG}. Applying conclusion \eqref{eq:threepoint_sk4} in Lemma~\ref{eq:threepoint} to Line 8 of Algorithm~\ref{alg:SPDHG} with the following instantiation
\begin{eqnarray}
\label{eq:threepoint_instance_lambdak1}
&\vs_k=-\ell_{\vg}(\bx_k;\bx_{k-1},\by_{k+1})-\theta_k \bar\bq_{k},\\
\label{eq:threepoint_instance_lambdak2}
&\vdelta_k=-\ell_{\vG}(\vx_k,\xi_k^3;\vx_{k-1},\vy_{k+1})+\ell_{\vg}(\vx_k;\vx_{k-1},\vy_{k+1})-\theta_k( \bq_{k}-\bar\bq_{k}),\\
\label{eq:threepoint_instance_lambdak3}
&\mathcal{W}=\mathbb{R}_+^m,~
\vw_0= \vlambda_0,\text{ and }
\eta_k=\gamma_k,
\end{eqnarray}
we have, for any $\vlambda\in\mathbb{R}_+^m$, that
\begin{align}
\label{eq:threepoint_lambdak}
\left\langle \vs_k, \vlambda_{k+1}-\vlambda\right\rangle
\leq &-\frac{\gamma_k}{4}\|\vlambda_{k+1}-\vlambda_k\|^2+
\frac{\gamma_k}{2}\|\vlambda-\vlambda_k\|^2-\frac{\gamma_k}{2}\|\vlambda-\vlambda_{k+1}\|^2\\\nonumber
&+\frac{\gamma_k}{2}\|\vlambda-\tilde\vlambda_k\|^2-\frac{\gamma_k}{2}\|\vlambda-\tilde\vlambda_{k+1}\|^2+\left\langle \vdelta_k, \tilde\vlambda_k-\vlambda_k\right\rangle+\frac{3}{2\gamma_k}\|\vdelta_k\|^2.
\end{align}

Recall \eqref{eq:qk}. Using the instance of $\vs_k$ in \eqref{eq:threepoint_instance_lambdak1}, we have 
\begin{eqnarray}
\label{eq:Q1sub1}
&&(\vlambda-\vlambda_{k+1})^{\top}\vg(\bx_{k+1},\by_{k+1})\\\nonumber
&=&\left\langle \vs_k, \vlambda_{k+1}-\vlambda\right\rangle+\langle \ell_{\vg}(\vx_{k+1};\vx_k,\vy_{k+1})-\vg(\bx_{k+1},\by_{k+1}),\vlambda_{k+1}-\vlambda\rangle \\\nonumber
&&-\langle \bar\bq_{k+1},\vlambda_{k+1}-\vlambda\rangle + \theta_k \langle \bar\bq_{k},\vlambda_{k}-\vlambda\rangle + \theta_k \langle \bar\bq_{k},\vlambda_{k+1}-\vlambda_{k}\rangle.
\end{eqnarray}
By Assumption~\ref{assume:continuity}.b, we have 
\begin{equation*}
    g_i(\vx_{k+1},\vy_{k+1}^i)-g_i(\vx_k,\vy_{k+1}^i)- \nabla_x g_i(\vx_k,\vy_{k+1}^i)^\top (\vx_{k+1}-\vx_k)\leq\frac{L_g^{xx}}{2} \|\vx_{k+1}-\vx_k\|^2+ H_g^{x} \|\vx_{k+1}-\vx_k\|.
\end{equation*}
Multiplying the inequality above by any $\lambda^i\geq 0$ and summing it for $i=1,\dots,m$ give
\begin{eqnarray}
\label{eq:Q0eq5}
\vlambda^\top \left[\vg(\vx_{k+1},\vy_{k+1})-\ell_{\vg}(\vx_{k+1};\vx_k,\vy_{k+1})\right]\leq \frac{L_g^{xx}\|\vlambda\|_1}{2}\|\vx_{k+1}-\vx_k\|^2+H_g^{x}\|\vlambda\|_1 \|\vx_{k+1}-\vx_k\|.
\end{eqnarray}
Applying this inequality to the right-hand side of \eqref{eq:Q1sub1} leads to 
\begin{align}
\label{eq:Q1sub2}
(\vlambda-\vlambda_{k+1})^{\top}\vg(\bx_{k+1},\by_{k+1})
\leq&\left\langle \vs_k, \vlambda_{k+1}-\vlambda\right\rangle+\langle \ell_{\vg}(\vx_{k+1};\vx_k,\vy_{k+1})-\vg(\bx_{k+1},\by_{k+1}),\vlambda_{k+1}\rangle \\\nonumber
&+\frac{L_g^{xx}\|\vlambda\|_1}{2}\|\vx_{k+1}-\vx_k\|^2+H_g^{x}\|\vlambda\|_1 \|\vx_{k+1}-\vx_k\| \\\nonumber
&-\langle \bar\bq_{k+1},\vlambda_{k+1}-\vlambda\rangle + \theta_k \langle \bar\bq_{k},\vlambda_{k}-\vlambda\rangle + \theta_k \langle \bar\bq_{k},\vlambda_{k+1}-\vlambda_{k}\rangle.
\end{align}
Recall definition \eqref{eq:ell} and the fact that Assumption~\ref{assume:continuity} d implies \eqref{eq:Lipgi}. According to \eqref{eq:qk}, we have
\small
\begin{align}
\nonumber
&\|\bar\bq_k\| \\\nonumber
=&\|\ell_{\vg}(\vx_{k};\vx_{k-1},\vy_{k})-\ell_{\vg}(\vx_{k-1};\vx_{k-2},\vy_{k})\|\\\nonumber
\leq &\sqrt{\sum_{i=1}^m\left[g_i(\vx_{k-1},\by_{k}^i)-g_i(\vx_{k-2},\by_{k}^i)+\nabla_xg_i(\vx_{k-1},\by_{k}^i)^\top(\vx_{k}-\vx_{k-1})-\nabla_xg_i(\vx_{k-2},\by_{k}^i)^\top(\vx_{k-1}-\vx_{k-2})\right]^2}\\\nonumber
\leq&\sqrt{\sum_{i=1}^m\left(M_g^x\|\vx_{k-1}-\vx_{k-2}\|+M_g^x\|\vx_{k}-\vx_{k-1}\|+M_g^x\|\vx_{k-1}-\vx_{k-2}\|\right)^2}
\\\label{eq:barqk_Youngineq}
= &\sqrt{m}\left(2M_g^x\|\vx_{k-1}-\vx_{k-2}\|+M_g^x\|\vx_{k}-\vx_{k-1}\|\right).
\end{align}
\normalsize
We next bound the term $\theta_k\langle \bar\bq_{k},\vlambda_{k+1}-\vlambda_{k}\rangle$ on the right-hand side of \eqref{eq:Q1sub2} as follows
\begin{eqnarray}
\label{eq:eq:Q1ub0}
&&\theta_k\langle \bar\bq_{k},\vlambda_{k+1}-\vlambda_{k}\rangle \\\nonumber
     &\leq& \frac{\gamma_k}{8}\|\vlambda_{k+1}-\vlambda_k\|^2+\frac{2\theta_k^2}{\gamma_k}\|\bar\bq_{k}\|^2\\\nonumber
     &\leq& \frac{\gamma_k}{8}\|\vlambda_{k+1}-\vlambda_k\|^2 + \frac{4m(\theta_kM_g^x)^2}{\gamma_k}\|\bx_k-\bx_{k-1}\|^2+\frac{16m(\theta_kM_g^x)^2}{\gamma_k}\|\bx_{k-1}-\bx_{k-2}\|^2,
\end{eqnarray}
where the first inequality is Young's inequality and the second inequality holds because of \eqref{eq:barqk_Youngineq} and Jensen's inequality.

Recall \eqref{eq:threepoint_instance_lambdak2} and the definition of $Q_1$ in \eqref{eq:Qs1}. Applying \eqref{eq:threepoint_lambdak} and \eqref{eq:eq:Q1ub0}  to \eqref{eq:Q1sub2} and organizing terms lead to
\begin{eqnarray}
\label{eq:Q1ub1}
        ~~~~&&Q_1(\bz_{k+1},\bz)=(\vlambda-\vlambda_{k+1})^{\top}\vg(\bx_{k+1},\by_{k+1})\\\nonumber
        &\leq & \langle\ell_{\vg}(\vx_{k+1};\vx_k,\vy_{k+1})-\vg(\vx_{k+1},\vy_{k+1}), \vlambda_{k+1}\rangle
        +\frac{L_g^{xx}\|\vlambda\|_1}{2}\|\vx_{k+1}-\vx_k\|^2+H_g^{x}\|\vlambda\|_1 \|\vx_{k+1}-\vx_k\|\\\nonumber
        &&-\langle \bar\bq_{k+1},\vlambda_{k+1}-\vlambda\rangle + \theta_k \langle \bar\bq_{k},\vlambda_{k}-\vlambda\rangle \\\nonumber
        && + \frac{\gamma_k}{2}\left[ \|\vlambda-\vlambda_k\|^2-\|\vlambda-\vlambda_{k+1}\|^2+\|\vlambda-\tilde\vlambda_k\|^2-\|\vlambda-\tilde\vlambda_{k+1}\|^2\right]\\\nonumber
        &&+ \frac{4m(\theta_kM_g^x)^2}{\gamma_k}\|\bx_k-\bx_{k-1}\|^2+\frac{16m(\theta_kM_g^x)^2}{\gamma_k}\|\bx_{k-1}-\bx_{k-2}\|^2-\frac{\gamma_k}{8}\|\vlambda_{k+1}-\vlambda_k\|^2 \\\nonumber
  &&+\langle\ell_{\vG}(\vx_k,\xi_k^3;\vx_{k-1},\vy_{k+1})-\ell_{\vg}(\vx_k;\vx_{k-1},\vy_{k+1})+\theta_k( \bq_{k}-\bar\bq_{k}),\vlambda_k-\tilde\vlambda_k\rangle\\\nonumber
&&+\frac{3}{2\gamma_k}\left\|\ell_{\vG}(\vx_k,\xi_k^3;\vx_{k-1},\vy_{k+1})-\ell_{\vg}(\vx_k;\vx_{k-1},\vy_{k+1})+\theta_k( \bq_{k}-\bar\bq_{k})\right\|^2
\end{eqnarray}
for any $\vlambda\geq0$.

\subsection{An upper bound of $Q_0$}
We first define an auxiliary sequence $\{\tilde \vx_k\}_{k\geq0}$, where $\tilde\vx_0=\vx_0$ and $\tilde\vx_k$ for $k\geq1$ are generated as 
\begin{eqnarray}
\label{eq:tildexk}
\tilde\bx_{k+1}=\argmin_{\bx\in\cX}-\langle \vdelta_k^6+(\vDelta_k^6)^\top\vlambda_{k+1}, \bx\rangle+\frac{\tau_k}{2}\|\bx-\tilde\bx_k\|^2,
\end{eqnarray}
where $ \vdelta_k^6$ and $ \vDelta_k^6$ are defined in \eqref{eqn:deltas} and $\tau_k$ is as in Algorithm~\ref{alg:SPDHG}. 
Applying conclusion \eqref{eq:threepoint_sk1} of Lemma~\ref{eq:threepoint} to Line 9 of Algorithm~\ref{alg:SPDHG} with the following instantiation
\begin{eqnarray}
\label{eq:threepoint_instance_xk1}
&\vs_k=\nabla f(\vx_k)+\nabla_x\vg(\vx_k,\vy_{k+1})^\top \vlambda_{k+1},\\
\label{eq:threepoint_instance_xk2}
&\vdelta_k= \vdelta_k^6+(\vDelta_k^6)^\top\vlambda_{k+1},\\
\label{eq:threepoint_instance_xk3}
&\mathcal{W}=\cX,~
\vw_0= \vx_0, \text{ and }
\eta_k=\tau_k,
\end{eqnarray}
we have, for any $\vx\in\cX$, that
\begin{eqnarray}
\label{eq:threepoint_xk_old}
\left\langle \vs_k+\vdelta_k, \vx_{k+1}-\vx\right\rangle&\leq &-\frac{\tau_k}{2}\|\vx_{k+1}-\vx_k\|^2+
\frac{\tau_k}{2}\|\vx-\vx_k\|^2-\frac{\tau_k}{2}\|\vx-\vx_{k+1}\|^2.
\end{eqnarray}
By Young's inequality, \eqref{eq:threepoint_xk} further implies
\begin{eqnarray}
\nonumber
\left\langle \vs_k, \vx_{k+1}-\vx\right\rangle&\leq &-\frac{\tau_k}{2}\|\vx_{k+1}-\vx_k\|^2+
\frac{\tau_k}{2}\|\vx-\vx_k\|^2-\frac{\tau_k}{2}\|\vx-\vx_{k+1}\|^2\\\nonumber
&&-\left\langle \vdelta_k, \vx_{k+1}-\vx_k\right\rangle-\left\langle \vdelta_k, \vx_k-\vx\right\rangle\\\nonumber
&\leq &-\frac{\tau_k}{4}\|\vx_{k+1}-\vx_k\|^2+
\frac{\tau_k}{2}\|\vx-\vx_k\|^2-\frac{\tau_k}{2}\|\vx-\vx_{k+1}\|^2\\\label{eq:threepoint_xk}
&&+\frac{\| \vdelta_k\|}{\tau_k}+\left\langle \vdelta_k, \vx-\vx_k\right\rangle.
\end{eqnarray}

By $\mu_f$-strong convexity of $f$, it holds, for any $\vx\in\cX$, that
\begin{equation}
\label{eq:Q0eq1}
\langle \nabla f(\vx_{k}),\bx -\bx_{k} \rangle \leq f(\vx)-f(\vx_{k}) -\frac{\mu_f}{2}\|\bx_{k}-\bx\|^2.
\end{equation}
By convexity of $g_i$ in $\bx$, it holds, for any $\vx\in\cX$, that
\begin{equation*}
   0\leq g_i(\vx,\vy_{k+1}^i)- g_i(\vx_k,\vy_{k+1}^i)-\nabla_x g_i(\vx_k,\vy_{k+1}^i)^\top (\vx-\vx_k).
\end{equation*}
Multiplying the inequality above by $\lambda_{k+1}^i\geq 0$ and summing it for $i=1,\dots,m$ give
\begin{eqnarray}
\label{eq:Q0eq2}
   0&\leq &\vlambda_{k+1}^\top(\vg(\bx,\by_{k+1})-\vg(\bx_k,\by_{k+1}))-\langle \nabla_x\vg(\vx_k,\vy_{k+1})^\top \vlambda_{k+1},\bx-\bx_k \rangle.
\end{eqnarray}

By \eqref{eq:Q0eq1}, \eqref{eq:Q0eq2} and the instance of $\vs_k$ in \eqref{eq:threepoint_instance_xk1}, we have 
\begin{eqnarray}
\nonumber
&&f(\bx_{k+1})-f(\bx)+\vlambda_{k+1}^\top(\vg(\bx_{k+1},\by_{k+1})-\vg(\bx,\by_{k+1}))\\\nonumber
&=&f(\bx_{k+1})-f(\bx)+\vlambda_{k+1}^\top(\vg(\bx_{k+1},\by_{k+1})-\vg(\bx,\by_{k+1}))\\\nonumber
&&+\left\langle \vs_k, \vx_{k+1}-\vx\right\rangle-\left\langle \nabla f(\vx_k)+\nabla_x\vg(\vx_k,\vy_{k+1})^\top \vlambda_{k+1}, \vx_{k+1}-\vx\right\rangle\\\nonumber
&\leq&f(\bx_{k+1})-f(\bx_k)+\vlambda_{k+1}^\top(\vg(\bx_{k+1},\by_{k+1})-\vg(\bx,\by_{k+1}))-\frac{\mu_f}{2}\|\bx_k-\bx\|^2\\\nonumber
&&+\left\langle \vs_k, \vx_{k+1}-\vx\right\rangle-\left\langle\nabla_x\vg(\vx_k,\vy_{k+1})^\top \vlambda_{k+1}, \vx_{k+1}-\vx\right\rangle -\left\langle \nabla f(\vx_k), \vx_{k+1}-\vx_k\right\rangle\\\label{eqn:tri_x_1}
&\leq&f(\bx_{k+1})-f(\bx_k)+\vlambda_{k+1}^\top(\vg(\bx_{k+1},\by_{k+1})-\vg(\bx_k,\by_{k+1}))-\frac{\mu_f}{2}\|\bx_k-\bx\|^2\\\nonumber
&&+\left\langle \vs_k, \vx_{k+1}-\vx\right\rangle-\left\langle\nabla_x\vg(\vx_k,\vy_{k+1})^\top \vlambda_{k+1}, \vx_{k+1}-\vx_k\right\rangle-\left\langle \nabla f(\vx_k), \vx_{k+1}-\vx_k\right\rangle,
\end{eqnarray}
where the first inequality is by \eqref{eq:Q0eq1} and the second is by \eqref{eq:Q0eq2}. By Assumption~\ref{assume:continuity} a, \eqref{eqn:tri_x_1} further implies
\begin{align}
\nonumber
&f(\bx_{k+1})-f(\bx)+\vlambda_{k+1}^\top(\vg(\bx_{k+1},\by_{k+1})-\vg(\bx,\by_{k+1}))\\\nonumber
\leq&\vlambda_{k+1}^\top(\vg(\bx_{k+1},\by_{k+1})-\vg(\bx_k,\by_{k+1}))-\frac{\mu_f}{2}\|\bx_k-\bx\|^2\\\nonumber
&+\left\langle \vs_k, \vx_{k+1}-\vx\right\rangle-\left\langle\nabla_x\vg(\vx_k,\vy_{k+1})^\top \vlambda_{k+1}, \vx_{k+1}-\vx_k\right\rangle+\frac{L_f}{2}\|\vx_{k+1}-\vx_{k}\|^2+H_f\|\vx_{k+1}-\vx_{k}\|\\\label{eqn:tri_x_2}
=&\langle\vg(\bx_{k+1},\by_{k+1})-\ell_{\vg}(\vx_{k+1};\vx_k,\vy_{k+1}), \vlambda_{k+1}\rangle-\frac{\mu_f}{2}\|\bx_k-\bx\|^2\\\nonumber
&+\left\langle \vs_k, \vx_{k+1}-\vx\right\rangle+\frac{L_f}{2}\|\vx_{k+1}-\vx_{k}\|^2+H_f\|\vx_{k+1}-\vx_{k}\|,
\end{align}
where the equality is because of  \eqref{eq:ell}.

Recall \eqref{eq:threepoint_instance_xk2} and the definition of $Q_0$ in \eqref{eq:Qs0}. Applying \eqref{eq:threepoint_xk}  to \eqref{eqn:tri_x_2} gives
\begin{eqnarray}
\label{eq:Q0eq4}
        &&Q_0(\bz_{k+1},\bz)= f(\bx_{k+1})-f(\bx)+\vlambda_{k+1}^\top(\vg(\bx_{k+1},\by_{k+1})-\vg(\bx,\by_{k+1}))  \\\nonumber
        &\leq & \langle\vg(\bx_{k+1},\by_{k+1})-\ell_{\vg}(\vx_{k+1};\vx_k,\vy_{k+1}), \vlambda_{k+1}\rangle\\\nonumber
        && -\frac{\mu_f}{2}\|\bx-\bx_{k}\|^2+ \left(\frac{L_f}{2}-\frac{\tau_k}{4}\right)\|\vx_{k+1}-\vx_{k}\|^2+H_f\|\vx_{k+1}-\vx_{k}\|\\\nonumber
        &&+ \frac{\tau_k}{2}\left[ \|\bx-\bx_k\|^2-\|\bx-\bx_{k+1}\|^2 \right]\\\nonumber
         &&+\langle \vdelta_k^6+(\vDelta_k^6)^\top\vlambda_{k+1},  \bx- \bx_k\rangle+\frac{1}{\tau_k}\left\|\vdelta_k^6+(\vDelta_k^6)^\top\vlambda_{k+1}\right\|^2.
\end{eqnarray}



\subsection{Summing up the upper bounds}
Recall that \eqref{eq:Q2ub1_old}, \eqref{eq:Q1ub1} and \eqref{eq:Q0eq4} provide an upper bound on $Q_2(\bz_{k+1},\bz)$, $Q_1(\bz_{k+1},\bz)$ and $Q_0(\bz_{k+1},\bz)$, respectively. In light of \eqref{eq:Qs}, adding these three inequalities and reorganizing terms lead to
\begin{eqnarray}
\label{eq:Qsum1}
&&\phi(\bx_{k+1},\vlambda,\by) - \phi(\bx,\vlambda_{k+1},\by_{k+1})\\\nonumber
     &\leq & \mathcal{T}_{\text{crossterm},k}(\vlambda,\by) +\mathcal{T}_{\text{change},k}(\by)+\mathcal{T}_{\text{distance},k}(\bx,\vlambda,\by)
     +\mathcal{T}_{\text{bias},k}(\bx,\vlambda)+\mathcal{T}_{\text{variance},k}(\vlambda).
\end{eqnarray}
Here, 
\begin{eqnarray}
\label{eq:crosstermk}
\mathcal{T}_{\text{crossterm},k}(\by,\vlambda)
&:=& -\langle \bar\bq_{k+1},\vlambda_{k+1}-\vlambda\rangle+ \theta_k \langle \bar\bq_{k},\vlambda_{k}-\vlambda\rangle \\\nonumber
&& -\vlambda^\top\bar\bP_{k+1} (\by_{k+1}-\by) +\theta_k\vlambda^\top\bar\bP_{k} (\by_{k}-\by),
\end{eqnarray}
\begin{eqnarray}
\label{eq:change}
\mathcal{T}_{\text{change},k}(\vlambda)
&:=& \left[\frac{4m(\theta_kM_g^x)^2}{\gamma_k}+\frac{ 4(\theta_kL^{yx}_g)^2\|\vlambda\|_1}{\sigma_k}\right]\|\bx_k-\bx_{k-1}\|^2\\\nonumber
&&+\frac{16m(\theta_kM_g^x)^2}{\gamma_k}\|\bx_{k-1}-\bx_{k-2}\|^2+ \left[\frac{L_f}{2}+\frac{L_g^{xx}\|\vlambda\|_1}{2}-\frac{\tau_k}{4}\right]\|\vx_{k+1}-\vx_{k}\|^2\\\nonumber
&& - \frac{\gamma_k}{8}\|\vlambda_{k+1}-\vlambda_k\|^2+\frac{4(\theta_kL^{yy}_g)^2}{\sigma_k}\|\vy_{k}-\vy_{k-1}\|_{\vlambda}^2-\frac{\sigma_k}{8}\|\vy_{k+1}-\vy_k\|_{\vlambda}^2\\\nonumber
&&+\left[H_g^{x}\|\vlambda\|_1 +H_f\right]\|\vx_{k+1}-\vx_{k}\|+\theta_kH_g^y \|\vlambda\|_1^{\frac{1}{2}}\|\vy_{k+1}-\vy_{k}\|_\vlambda,
\end{eqnarray}
\begin{eqnarray}
\label{eq:distance}
\mathcal{T}_{\text{distance},k}(\bx,\by,\vlambda)
&:=& \frac{\tau_k-\mu_f}{2}\|\bx-\bx_k\|^2-\frac{\tau_k}{2}\|\bx-\bx_{k+1}\|^2\\\nonumber
&&+\frac{\gamma_k}{2}\left[\|\vlambda-\vlambda_k\|^2- \|\vlambda-\vlambda_{k+1}\|^2\right] +\frac{\gamma_k}{2}\left[ \|\vlambda-\tilde\vlambda_k\|^2-\|\vlambda-\tilde\vlambda_{k+1}\|^2\right]\\\nonumber
&& +\frac{\sigma_k}{2}\|\vy-\vy_k\|_{\vlambda}^2-\frac{\sigma_k+\mu_y}{2}\|\vy-\vy_{k+1}\|_{\vlambda}^2+\frac{\sigma_k}{2}\left[ \|\vy-\tilde\vy_k\|_{\vlambda}^2-\|\vy-\tilde\vy_{k+1}\|_{\vlambda}^2\right],
\end{eqnarray}
\begin{eqnarray}
\label{eq:zeromean}
\mathcal{T}_{\text{bias},k}(\bx,\vlambda)
&:=&\langle \vdelta_k^6+(\vDelta_k^6)^\top\vlambda_{k+1}, \bx- \bx_k\rangle+ \vlambda^\top\left[\vDelta_k^1+\theta_k \left(\vDelta_k^1-\vDelta_k^2\right)\right](\vy_k-\tilde\vy_k)\\\nonumber
&&+\langle\ell_{\vG}(\vx_k,\xi_k^3;\vx_{k-1},\vy_{k+1})-\ell_{\vg}(\vx_k;\vx_{k-1},\vy_{k+1})+\theta_k( \bq_{k}-\bar\bq_{k}),\vlambda_k-\tilde\vlambda_k\rangle,
\end{eqnarray}
and
\begin{eqnarray}
\label{eq:variance}
\mathcal{T}_{\text{variance},k} (\vlambda)
&:=&\frac{1}{\tau_k}\left\|\vdelta_k^6+(\vDelta_k^6)^\top\vlambda_{k+1}\right\|^2 +\frac{3}{2\sigma_k}\left\|\vDelta_k^1+\theta_k \left(\vDelta_k^1-\vDelta_k^2\right)\right\|_{\vlambda}^2\\\nonumber
&&+\frac{3}{2\gamma_k}\left\|\ell_{\vG}(\vx_k,\xi_k^3;\vx_{k-1},\vy_{k+1})-\ell_{\vg}(\vx_k;\vx_{k-1},\vy_{k+1})+\theta_k( \bq_{k}-\bar\bq_{k})\right\|^2.
\end{eqnarray}

Suppose the gradient oracles are deterministic, namely, all quantities in \eqref{eqn:deltas} are zeros. In this case,
\begin{equation}
\label{eq:zeromean_and_variance_deterministic}
\mathcal{T}_{\text{bias},k}(\bx,\vlambda)=\mathcal{T}_{\text{variance},k}(\vlambda)=0,
\end{equation}
and, by the definitions of $\tilde\by_k$, $\tilde\vlambda_k$ and $\tilde\bx_k$ in \eqref{eq:tildeyk}, \eqref{eq:tildelambdak} and \eqref{eq:tildexk}. We also have 
\begin{equation}
\label{eq:tildetermzero}
\tilde\bx_k=\bx_0,\quad\tilde\by_k=\by_0,\text{ and }\tilde\vlambda_k=\vlambda_0, \quad \forall k\geq0.
\end{equation}

Recall that the output of Algorithms~\ref{alg:PDHG} and \ref{alg:SPDHG} is $\bar\vx_{K}=\frac{1}{\sum_{k=0}^{K-1}t_k}\sum_{k=0}^{K-1}t_k\vx_{k+1}$. 
By Assumptions~\ref{assume:basic} c and d, we have
\small
\begin{align}
\nonumber
&f(\bar\vx_{K})+\vlambda^\top \vg(\bar\vx_{K},\by)-f(\bx^*)-\frac{1}{\sum_{k=0}^{K-1}t_k}\sum_{k=0}^{K-1}t_k\vlambda_{k+1}^\top \vg^*(\bx^*)\\\nonumber
\leq&f(\bar\vx_{K})+\vlambda^\top \vg(\bar\vx_{K},\by)-f(\bx^*)-\frac{1}{\sum_{k=0}^{K-1}t_k}\sum_{k=0}^{K-1}t_k\vlambda_{k+1}^\top \vg(\bx^*,\by_{k+1})\\\label{eq:Qsum1_averaged}
\leq&\frac{1}{\sum_{k=0}^{K-1}t_k}\sum_{k=0}^{K-1}t_k[\phi(\bx_{k+1},\vlambda,\by) - \phi(\bx^*,\vlambda_{k+1},\by_{k+1})]
\end{align}
\normalsize
for any $\vlambda\in\mathbb{R}^m_+$ and any $\vy\in\mathcal{Y}$, where the first inequality is because $\vg^*(\bx^*)\geq \vg(\bx^*,\by_{k+1})$ and $\vlambda_{k+1}\geq 0$, and the second is because of Assumptions~\ref{assume:basic} c and d.

Since $\vg^*(\bx^*)\leq \mathbf{0}$, choosing $\vlambda=\vlambda_0=\mathbf{0}$ and $\by=\by_0$ in \eqref{eq:Qsum1_averaged} leads to an upper bound on the optimality gap of  $\bar\vx_{K}$, that is, 
\small
\begin{align}
\label{eq:Qsum_averaged_obj}
f(\bar\vx_{K})-f(\bx^*)\leq\frac{1}{\sum_{k=0}^{K-1}t_k}\sum_{k=0}^{K-1}t_k[\phi(\bx_{k+1},\mathbf{0},\by_0) - \phi(\bx^*,\vlambda_{k+1},\by_{k+1})].
\end{align}
\normalsize
Let $\vlambda^*$ be the optimal multiplier corresponding to $\bx^*$ in Assumption~\ref{assume:KKT}. By the definitions of $\bx^*$ and $\vlambda^*$, we have
\begin{align}
\label{eq:choiceinQ2_prepare1}
0\leq f(\bar\vx_{K})+(\vlambda^*)^\top \vg^*(\bar\vx_{K})-f(\bx^*)-(\vlambda^*)^\top \vg^*(\bx^*)=f(\bar\vx_{K})+(\vlambda^*)^\top \vg^*(\bar\vx_{K})-f(\bx^*).
\end{align}
Let 
\begin{align}
\label{eq:lambda_for_barxK}
\widehat\vlambda_K:=\argmax_{\vlambda}\left\{(\vlambda-\vlambda^*)^\top\vg^*(\bar\vx_{K})\big|\vlambda\geq\vlambda^*,\|\vlambda-\vlambda^*\|_1\leq 1\right\}
\end{align}
and $\widehat\by_K=(\widehat\by_K^{1\top},\dots,\widehat\by_K^{m\top})^\top\in \cY$, where
\begin{align}
\label{eq:y_for_barxK}
\widehat\by^i_K\in\argmax_{\by^i\in\mathcal{Y}^i}g_i(\bar\vx_{K},\by^i),~i=1,\dots,m.
\end{align}
It can be easily verified that
\begin{align}
\label{eq:lambda_for_barxK_bound}
\|\widehat\vlambda_K\|_1\leq \|\vlambda^*\|_1+1,
\end{align}
\begin{align}
\label{eq:lambda_for_barxK_property}
(\widehat\vlambda_K)^\top\vg^*(\bar\vx_{K})=\|[\vg^*(\bar\vx_{K})]_+\|_\infty+(\vlambda^*)^\top\vg^*(\bar\vx_{K})
\end{align}
and
\begin{align}
\label{eq:y_for_barxK_property}
\vg(\bar\vx_{K},\widehat\by_K)=\vg^*(\bar\vx_{K}).
\end{align}
Let $\by^i=\widehat\by^i_K$ for $i=1,\dots,m$ and $\vlambda=\widehat\vlambda_K$ in \eqref{eq:Qsum1_averaged}. We obtain an upper bound on the infeasibility of  $\bar\vx_{K}$, that is, 
\small
\begin{align}
\nonumber
\|[\vg^*(\bar\vx_{K})]_+\|_\infty
=&(\widehat\vlambda_K)^\top\vg^*(\bar\vx_{K})-(\vlambda^*)^\top\vg^*(\bar\vx_{K})\\\nonumber
\leq&f(\bar\vx_{K})+\widehat\vlambda_K^\top \vg(\bar\vx_{K},\widehat\by_K)-f(\bx^*)-\frac{1}{\sum_{k=0}^{K-1}t_k}\sum_{k=0}^{K-1}t_k\vlambda_{k+1}^\top \vg^*(\bx^*)\\\label{eq:Qsum_averaged_cst}
\leq&\frac{1}{\sum_{k=0}^{K-1}t_k}\sum_{k=0}^{K-1}t_k[\phi(\bx_{k+1},\widehat\vlambda_K,\widehat\by_K) - \phi(\bx^*,\vlambda_{k+1},\by_{k+1})],
\end{align}
\normalsize
where the equality is from \eqref{eq:lambda_for_barxK_property}, the first inequality is because of \eqref{eq:choiceinQ2_prepare1}, \eqref{eq:y_for_barxK_property} and the fact that $\vg^*(\bx^*)\leq \mathbf{0}$, and the last inequality is from \eqref{eq:Qsum1_averaged}.

\section{Convergence Analysis}
In this section, we present our main theoretical results on the convergence rate of Algorithms~\ref{alg:PDHG} and~\ref{alg:SPDHG} under different settings of $f$ and $\vg$.  In our analysis, we use $\mathbb{E}_k$ to denote the conditional expectation conditioning on $(\xi_l^i)_{i=1}^6$ for $l=0,1,\dots,k$.  Motivated by \eqref{eq:Qsum1}, \eqref{eq:Qsum_averaged_obj} and \eqref{eq:Qsum_averaged_cst}, we will first provide an upper bound of the summation of $\mathcal{T}_{\text{crossterm},k}$,  $\mathcal{T}_{\text{change},k}$, $\mathcal{T}_{\text{distance},k}$, $\mathcal{T}_{\text{bias},k}$ and $\mathcal{T}_{\text{variance},k}$ over $k=0,1,\dots,K-1$ in the next proposition, which is the key step to derive the convergence rates of the optimality gap and infeasibility of  $\bar\vx_{K}$.

\begin{proposition}
\label{thm:mainprop_deterministic}
Suppose  parameters $t_k$, $\theta_k$, $\tau_k$, $\sigma_k$ and $\gamma_k$ in Algorithm~\ref{alg:SPDHG} satisfies the following conditions
\small
\begin{align}
\label{eq:thetaktk}
&\theta_{k+1}t_{k+1}=t_k\\
\label{eq:determine_cond2}
&\frac{t_kL_f}{2}+\frac{6mt_{k+1}\theta_{k+1}^2(M_g^x)^2}{\gamma_{k+1}}+\frac{24mt_{k+2}\theta_{k+2}^2(M_g^x)^2}{\gamma_{k+2}}+\frac{18m\sigma_{g'}^2t_{k+1}}{\gamma_{k+1}}+\frac{12m\sigma_{g'}^2t_{k+2}}{\gamma_{k+2}}- \frac{t_k\tau_k}{16}\leq0,\\
\label{eq:determine_cond3}
&\frac{7t_{k+1}(\theta_{k+1}L^{yy}_g)^2}{\sigma_{k+1}}
- \frac{t_k\sigma_k}{8}\leq0,\\
\label{eq:determine_cond4_a}
&  t_{k+1}\tau_{k+1}-t_{k+1}\mu_f-t_{k}\tau_{k}\leq0,\\
\label{eq:determine_cond4_b}
& t_{k+1}\sigma_{k+1}- t_k\sigma_k-t_k\mu_y\leq 0,\\
\label{eq:determine_cond4_c}
&t_{k+1}\gamma_{k+1}-t_k\gamma_k+\frac{8m\sigma_{g'}^2t_k}{\tau_k} \leq 0
\end{align}
\normalsize
for $k=0,1,\dots,K$. It holds, for any deterministic $\bx\in\cX$, stochastic $\by\in\cY$ and stochastic $\vlambda\in\mathbb{R}_+^m$ satisfying $\|\vlambda\|_1\leq \Lambda$ almost surely, that
\small
\begin{align}
\nonumber
&\mathbb{E}\sum_{k=0}^{K-1}t_k[\phi(\bx_{k+1},\vlambda,\by) - \phi(\bx,\vlambda_{k+1},\by_{k+1})]\\\label{eq:Qsum5}
\leq&\frac{t_0\tau_0-t_0\mu_f}{2}\|\bx-\bx_0\|^2+t_0\gamma_0\mathbb{E}\|\vlambda-\vlambda_0\|^2+t_0\sigma_0\mathbb{E}\|\vy-\vy_0\|_{\vlambda}^2+\mathbb{E}\mathcal{T}_{\Lambda,K}\\\nonumber
&+\sum_{k=0}^{K-1}\frac{t_{k+1}\gamma_{k+1}-t_k\gamma_k}{2}\mathbb{E}\|\vlambda-\tilde\vlambda_{k+1}\|^2  +\sum_{k=0}^{K-1}\frac{t_{k+1}\sigma_{k+1}-t_k\sigma_k}{2}\mathbb{E}\|\vy-\tilde\vy_{k+1}\|_{\vlambda}^2\\\nonumber
&-\frac{t_K\gamma_K}{2}\mathbb{E}\|\vlambda-\tilde\vlambda_K\|^2-\frac{t_K\sigma_K}{2}\mathbb{E}\|\vy-\tilde\vy_K\|_{\vlambda}^2  +4\left[(H_g^{x})^2\Lambda^2+(H_f)^2\right]\sum_{k=0}^{K-1}\frac{t_k}{\tau_k}+7(H_g^y)^2\Lambda\sum_{k=0}^{K}\frac{t_k\theta_k^2}{\sigma_k}\\\nonumber
&+\Lambda
\sqrt{10m\sigma_{g'}^2D_y^2\sum_{k=0}^{K-1}t_k^2}+\left(2\sigma_{f'}^2+4\sigma_{g'}^2\Lambda^2\right)\sum_{k=0}^{K-1}\frac{t_k}{\tau_k}+15m\Lambda\sigma_{g'}^2\sum_{k=0}^{K-1}\frac{t_k}{\sigma_k}+30m\sigma_g^2\sum_{k=0}^{K-1}\frac{t_k}{\gamma_k}
\end{align}
\normalsize
where 
\small
\begin{eqnarray}
\label{eq:TK}
\mathcal{T}_{\Lambda,K}:=\sum_{k=0}^{K-1}\left[\frac{ 7t_{k+1}(\theta_{k+1}L^{yx}_g)^2\Lambda}{\sigma_{k+1}}+\frac{t_kL_g^{xx}\Lambda}{2}-\frac{t_k\tau_k}{16}\right]\mathbb{E}\|\bx_{k+1}-\bx_k\|^2.
\end{eqnarray}
\normalsize
\end{proposition}
\begin{proof}
For simplicity of notation, we will hide the arguments of $\mathcal{T}_{\text{crossterm},k}$,  $\mathcal{T}_{\text{change},k}$, $\mathcal{T}_{\text{distance},k}$, $\mathcal{T}_{\text{bias},k}$ and $\mathcal{T}_{\text{variance},k}$ in this proof.

Recall \eqref{eq:thetaktk} and \eqref{eq:barqpP0}. Multiplying $\mathcal{T}_{\text{crossterm},k}$ in \eqref{eq:crosstermk} by $t_k\geq0$ and summing it up for $k=0,1,\dots,K-1$, we obtain 
\small
\begin{align}
\label{eq:crossterm_sum1}
&\sum_{k=0}^{K-1}t_k\mathcal{T}_{\text{crossterm},k}\\\nonumber
=&\sum_{k=0}^{K-1}t_k\left[\theta_k \langle \bar\bq_{k},\vlambda_{k}-\vlambda\rangle - \langle \bar\bq_{k+1},\vlambda_{k+1}-\vlambda\rangle \right]+\sum_{k=0}^{K-1}t_k\left[\theta_k\vlambda^\top\bar\bP_{k}(\by_{k}-\by)-\vlambda^\top\bar\bP_{k+1}(\by_{k+1}-\by) \right]\\\nonumber
=&t_0\theta_0 \langle \bar\bq_{0},\vlambda_{0}-\vlambda\rangle-t_K\theta_K \langle \bar\bq_{K},\vlambda_{K}-\vlambda\rangle +t_0\theta_0\vlambda^\top\bar\bP_{0}(\by_{0}-\by)-t_K\theta_K\vlambda^\top\bar\bP_{K}(\by_{K}-\by)\\\nonumber
=&-t_K\theta_K \langle \bar\bq_{K},\vlambda_{K}-\vlambda\rangle-t_K\theta_K\vlambda^\top\bar\bP_{K}(\by_{K}-\by).
\end{align}
\normalsize
Applying Young's inequality, we have
\small
\begin{eqnarray}
\label{eq:last_term1}
&&-t_K\theta_K\langle \bar\bq_{K},\vlambda_{K}-\vlambda\rangle\\\nonumber
&\leq& \frac{t_K\theta_K^2}{\gamma_K}\|\bar\bq_{K}\|^2 +\frac{t_K\gamma_K}{4}\|\vlambda-\vlambda_K\|^2\\\nonumber
&\leq& \frac{8mt_K\theta_K^2(M_g^x)^2}{\gamma_K}\|\vx_{K-1}-\vx_{K-2}\|^2+\frac{2mt_K\theta_K^2(M_g^x)^2}{\gamma_K}\|\vx_{K}-\vx_{K-1}\|^2 +\frac{t_K\gamma_K}{4}\|\vlambda-\vlambda_K\|^2
\end{eqnarray}
\normalsize
where the second inequality is from \eqref{eq:barqk_Youngineq} and Jensen's inequality. Similar, we also have
\small
\begin{eqnarray}
\label{eq:last_term2}
&&-t_K\theta_K\vlambda^\top\bar\bP_{K}(\by_{K}-\by)\\\nonumber
&\leq& \frac{t_K\theta_K^2}{\sigma_K}\sum_{i=1}^m\lambda_i\|\nabla_yg_i(\bx_{K},\by_{K})-\nabla_y g_i(\bx_{K-1},\by_{K-1})\|^2 +\frac{t_K\sigma_K}{4}\|\vy-\vy_K\|_{\vlambda}^2\\\nonumber
&\leq& \frac{3t_K\theta_K^2}{\sigma_K}\|\vlambda\|_1(L_g^{yx})^2\|\bx_{K-1}-\bx_{K}\|^2 
+\frac{3t_K\theta_K^2}{\sigma_K}(L_g^{yy})^2\|\by_{K-1}-\by_{K}\|_{\vlambda}^2
+\frac{3t_K\theta_K^2}{\sigma_K}\|\vlambda\|_1(H_g^y)^2\\\nonumber
&&+\frac{t_K\sigma_K}{4}\|\vy-\vy_K\|_{\vlambda}^2,
\end{eqnarray}
\normalsize
where the first inequlity comes from Young's inequality and the second is by Assumption~\ref{assume:continuity}c and Jensen's inequality. 
Applying \eqref{eq:last_term1} and \eqref{eq:last_term2} to \eqref{eq:crossterm_sum1} leads to 
\small
\begin{eqnarray}
\nonumber
\sum_{k=0}^{K-1}t_k\mathcal{T}_{\text{crossterm},k}
&\leq &\left(\frac{3t_K\theta_K^2(L_g^{yx})^2\|\vlambda\|_1}{\sigma_K}+\frac{2mt_K\theta_K^2(M_g^x)^2}{\gamma_K}\right)\|\vx_{K}-\vx_{K-1}\|^2\\\nonumber
&&+\frac{8mt_K\theta_K^2(M_g^x)^2}{\gamma_K}\|\vx_{K-1}-\vx_{K-2}\|^2+\frac{3t_K\theta_K^2(L_g^{yy})^2}{\sigma_K}\|\by_{K-1}-\by_{K}\|_{\vlambda}^2\\\nonumber
&&+\frac{t_K\gamma_K}{4}\|\vlambda-\vlambda_K\|^2+\frac{t_K\sigma_K}{4}\|\vy-\vy_K\|_{\vlambda}^2+\frac{3t_K\theta_K^2}{\sigma_K}\|\vlambda\|_1(H_g^y)^2\\\nonumber
&\leq &\sum_{k=0}^{K-1}\left(\frac{3t_{k+1}\theta_{k+1}^2(L_g^{yx})^2\|\vlambda\|_1}{\sigma_{k+1}}+\frac{2mt_{k+1}\theta_{k+1}^2(M_g^x)^2}{\gamma_{k+1}}+\frac{8mt_{k+2}\theta_{k+2}^2(M_g^x)^2}{\gamma_{k+2}}\right)\|\vx_{k+1}-\vx_{k}\|^2\\\nonumber
&&+\sum_{k=0}^{K-1}\frac{3t_{k+1}\theta_{k+1}^2(L_g^{yy})^2}{\sigma_{k+1}}\|\by_{k}-\by_{k+1}\|_{\vlambda}^2\\\label{eq:crossterm_sum2}
&&+\frac{t_K\gamma_K}{4}\|\vlambda-\vlambda_K\|^2+\frac{t_K\sigma_K}{4}\|\vy-\vy_K\|_{\vlambda}^2+\sum_{k=0}^{K-1}\frac{3t_{k+1}\theta_{k+1}^2}{\sigma_{k+1}}\|\vlambda\|_1(H_g^y)^2,
\end{eqnarray}
\normalsize
where the second inequality is obtained by adding non-negative terms  to the right-hand side.



Multiplying $\mathcal{T}_{\text{change},k}$ in\eqref{eq:change}  by $t_k\geq0$ and summing it up for $k=0,1,\dots,K-1$, we obtain 
\small
\begin{eqnarray}
\label{eq:change_sum1_old}
\quad\quad&&\sum_{k=0}^{K-1}t_k\mathcal{T}_{\text{change},k}\\\nonumber
&\leq& \sum_{k=0}^{K-1}\left[\frac{t_kL_f}{2}+\frac{4mt_{k+1}(\theta_{k+1}M_g^x)^2}{\gamma_{k+1}}+\frac{16mt_{k+2}(\theta_{k+2}M_g^x)^2}{\gamma_{k+2}}-\frac{t_k\tau_k}{8}\right]\|\bx_{k+1}-\bx_k\|^2\\\nonumber
&&+\sum_{k=0}^{K-1}\left[\frac{ 4t_{k+1}(\theta_{k+1}L^{yx}_g)^2\|\vlambda\|_1}{\sigma_{k+1}}+\frac{t_kL_g^{xx}\|\vlambda\|_1}{2}-\frac{t_k\tau_k}{8}\right]\|\bx_{k+1}-\bx_k\|^2\\\nonumber
&& +\sum_{k=0}^{K-1}\left[\frac{4t_{k+1}(\theta_{k+1}L^{yy}_g)^2}{\sigma_{k+1}}-\frac{t_k\sigma_k}{8}\right]\|\vy_{k+1}-\vy_k\|_{\vlambda}^2- \sum_{k=0}^{K-1}\frac{\gamma_k}{8}\|\vlambda_{k+1}-\vlambda_k\|^2\\\nonumber
&&+\sum_{k=0}^{K-1}t_k\left[H_g^{x}\|\vlambda\|_1 +H_f\right]\|\vx_{k+1}-\vx_{k}\|+\sum_{k=0}^{K-1}t_k\theta_kH_g^y \|\vlambda\|_1^{\frac{1}{2}}\|\vy_{k+1}-\vy_{k}\|_\vlambda
\end{eqnarray}
\normalsize
By Young's inequality, we have
\small
\begin{eqnarray}
\label{eq:change_sum1_orderone}
&&\left[H_g^{x}\|\vlambda\|_1 +H_f\right]\|\vx_{k+1}-\vx_{k}\|+\theta_kH_g^y \|\vlambda\|_1^{\frac{1}{2}}\|\vy_{k+1}-\vy_{k}\|_\vlambda\\\nonumber
&\leq&\frac{4(H_g^{x})^2\|\vlambda\|_1^2}{\tau_k}+ \frac{\tau_k}{16}\|\vx_{k+1}-\vx_{k}\|^2+\frac{4(H_f)^2}{\tau_k}+ \frac{\tau_k}{16}\|\vx_{k+1}-\vx_{k}\|^2\\\nonumber
&&+\frac{4(\theta_kH_g^y)^2\|\vlambda\|_1}{\sigma_k} +\frac{\sigma_k}{16}\|\vy_{k+1}-\vy_{k}\|_\vlambda^2.
\end{eqnarray}
\normalsize
Applying \eqref{eq:change_sum1_orderone} to \eqref{eq:change_sum1_old} gives
\small
\begin{eqnarray}
\nonumber
&&\sum_{k=0}^{K-1}t_k\mathcal{T}_{\text{change},k}\\\label{eq:change_sum1}
&\leq& \sum_{k=0}^{K-1}\left[\frac{t_kL_f}{2}+\frac{4mt_{k+1}(\theta_{k+1}M_g^x)^2}{\gamma_{k+1}}+\frac{16mt_{k+2}(\theta_{k+2}M_g^x)^2}{\gamma_{k+2}}-\frac{t_k\tau_k}{16}\right]\|\bx_{k+1}-\bx_k\|^2\\\nonumber
&&+\sum_{k=0}^{K-1}\left[\frac{ 4t_{k+1}(\theta_{k+1}L^{yx}_g)^2\|\vlambda\|_1}{\sigma_{k+1}}+\frac{t_kL_g^{xx}\|\vlambda\|_1}{2}-\frac{t_k\tau_k}{16}\right]\|\bx_{k+1}-\bx_k\|^2\\\nonumber
&& +\sum_{k=0}^{K-1}\left[\frac{4t_{k+1}(\theta_{k+1}L^{yy}_g)^2}{\sigma_{k+1}}-\frac{t_k\sigma_k}{8}\right]\|\vy_{k+1}-\vy_k\|_{\vlambda}^2 - \sum_{k=0}^{K-1}\frac{\gamma_k}{16}\|\vlambda_{k+1}-\vlambda_k\|^2\\\nonumber
&& +\sum_{k=0}^{K-1}\left[\frac{4t_k(H_g^{x})^2\|\vlambda\|_1^2}{\tau_k}+\frac{4t_k(H_f)^2}{\tau_k}+\frac{4t_k(\theta_kH_g^y)^2\|\vlambda\|_1}{\sigma_k}\right].
\end{eqnarray}
\normalsize

Multiplying $\mathcal{T}_{\text{distance},k}$ in \eqref{eq:distance} by $t_k\geq0$ and summing it up for $k=0,1,\dots,K-1$, we obtain 
\small
\begin{eqnarray}
\nonumber
&&\sum_{k=0}^{K-1}t_k\mathcal{T}_{\text{distance},k} \\\nonumber
&=& \frac{t_0\tau_0-t_0\mu_f}{2}\|\bx-\bx_0\|^2+\frac{t_0\gamma_0}{2}\|\vlambda-\vlambda_0\|^2+\frac{t_0\sigma_0}{2}\|\vy-\vy_0\|_{\vlambda}^2   \\\nonumber
&&+\sum_{k=0}^{K-1}\frac{t_{k+1}\tau_{k+1}-t_{k+1}\mu_f-t_k\tau_k}{2}\|\bx-\bx_{k+1}\|^2+\sum_{k=0}^{K-1}\frac{t_{k+1}\gamma_{k+1}-t_k\gamma_k}{2}\|\vlambda-\vlambda_{k+1}\|^2 \\\nonumber
&& +\sum_{k=0}^{K-1}\frac{t_{k+1}\sigma_{k+1}-t_k\sigma_k-t_k\mu_y}{2}\|\vy-\vy_{k+1}\|_{\vlambda}^2 \\\nonumber
&&- \frac{t_K\tau_K-t_K\mu_f}{2}\|\bx-\bx_K\|^2-\frac{t_K\gamma_K}{2}\|\vlambda-\vlambda_K\|^2-\frac{t_K\sigma_K}{2}\|\vy-\vy_K\|_{\vlambda}^2   \\\nonumber
&&+\frac{t_0\gamma_0}{2}\|\vlambda-\tilde\vlambda_0\|^2+\frac{t_0\sigma_0}{2}\|\vy-\tilde\vy_0\|_{\vlambda}^2   \\\nonumber
&&+\sum_{k=0}^{K-1}\frac{t_{k+1}\gamma_{k+1}-t_k\gamma_k}{2}\|\vlambda-\tilde\vlambda_{k+1}\|^2  +\sum_{k=0}^{K-1}\frac{t_{k+1}\sigma_{k+1}-t_k\sigma_k}{2}\|\vy-\tilde\vy_{k+1}\|_{\vlambda}^2 \\\nonumber
&&-\frac{t_K\gamma_K}{2}\|\vlambda-\tilde\vlambda_K\|^2-\frac{t_K\sigma_K}{2}\|\vy-\tilde\vy_K\|_{\vlambda}^2   \\\label{eq:distance_sum}
&\leq& \frac{t_0\tau_0-t_0\mu_f}{2}\|\bx-\bx_0\|^2+t_0\gamma_0\|\vlambda-\vlambda_0\|^2+t_0\sigma_0\|\vy-\vy_0\|_{\vlambda}^2\\\nonumber
&&+\sum_{k=0}^{K-1}\frac{t_{k+1}\tau_{k+1}-t_{k+1}\mu_f-t_k\tau_k}{2}\|\bx-\bx_{k+1}\|^2+\sum_{k=0}^{K-1}\frac{t_{k+1}\gamma_{k+1}-t_k\gamma_k}{2}\|\vlambda-\vlambda_{k+1}\|^2 \\\nonumber
&& +\sum_{k=0}^{K-1}\frac{t_{k+1}\sigma_{k+1}-t_k\sigma_k-t_k\mu_y}{2}\|\vy-\vy_{k+1}\|_{\vlambda}^2 \\\nonumber
&&-\frac{t_K\gamma_K}{2}\|\vlambda-\vlambda_K\|^2-\frac{t_K\sigma_K}{2}\|\vy-\vy_K\|_{\vlambda}^2  -\frac{t_K\gamma_K}{2}\|\vlambda-\tilde\vlambda_K\|^2-\frac{t_K\sigma_K}{2}\|\vy-\tilde\vy_K\|_{\vlambda}^2 \\\nonumber
&&+\sum_{k=0}^{K-1}\frac{t_{k+1}\gamma_{k+1}-t_k\gamma_k}{2}\|\vlambda-\tilde\vlambda_{k+1}\|^2  +\sum_{k=0}^{K-1}\frac{t_{k+1}\sigma_{k+1}-t_k\sigma_k}{2}\|\vy-\tilde\vy_{k+1}\|_{\vlambda}^2,
\end{eqnarray}
\normalsize
where the inequality is obtained by  dropping non-negative terms and using the fact that $\tilde\by_0=\by_0$ and $\tilde\vlambda_0=\vlambda_0$.


Next, we will bound $\mathbb{E}\sum_{k=0}^{K-1}t_k\mathcal{T}_{\text{bias},k}$ from above. It is easy to show that, for any (deterministic) $\bx\in\mathcal{X}$, 
\small
\begin{eqnarray}
\label{eq:zeromean_6}
\mathbb{E}\langle \vdelta_k^6+(\vDelta_k^6)^\top\vlambda_{k+1}, \bx- \bx_k\rangle
=\mathbb{E}\mathbb{E}_{k-1}\left[\langle \vdelta_k^6+(\vDelta_k^6)^\top\vlambda_{k+1}, \bx- \bx_k\rangle\Big| \vlambda_{k+1},\vy_{k+1}\right]=0,
\end{eqnarray}
\normalsize
where the last equality holds because of \eqref{eq:mean_gf} and \eqref{eq:mean_Ggx}, 
and
\small
\begin{eqnarray}
\nonumber
&&\mathbb{E}\langle\ell_{\vG}(\vx_k,\xi_k^3;\vx_{k-1},\vy_{k+1})-\ell_{\vg}(\vx_k;\vx_{k-1},\vy_{k+1})+\theta_k( \bq_{k}-\bar\bq_{k}),\vlambda_k-\tilde\vlambda_k\rangle\\\label{eq:zeromean_345}
&=&\mathbb{E}\mathbb{E}_{k-1}\left[\langle\ell_{\vG}(\vx_k,\xi_k^3;\vx_{k-1},\vy_{k+1})-\ell_{\vg}(\vx_k;\vx_{k-1},\vy_{k+1})+\theta_k( \bq_{k}-\bar\bq_{k}),\vlambda_k-\tilde\vlambda_k\rangle\Big| \vy_{k+1}\right]\\\nonumber
&=&0,
\end{eqnarray}
\normalsize
where the last equality holds because of \eqref{eq:mean_Gi} and \eqref{eq:mean_Ggx}. 
Given any $0\leq l<k\leq K-1$, we have
\small
\begin{eqnarray*}
&&\mathbb{E}\left\{\left[\vdelta_k^{i,1}+\theta_k \left(\vdelta_k^{i,1}-\vdelta_k^{i,2}\right)\right]^\top(\vy_k^i-\tilde\vy_k^i)
\cdot
\left[\vdelta_l^{i,1}+\theta_l \left(\vdelta_l^{i,1}-\vdelta_l^{i,2}\right)\right]^\top(\vy_l^i-\tilde\vy_l^i)
\right\}\\
&=&\mathbb{E}\mathbb{E}_{k-1}\left\{\left[\vdelta_k^{i,1}+\theta_k \left(\vdelta_k^{i,1}-\vdelta_k^{i,2}\right)\right]^\top(\vy_k^i-\tilde\vy_k^i)
\cdot
\left[\vdelta_l^{i,1}+\theta_l \left(\vdelta_l^{i,1}-\vdelta_l^{i,2}\right)\right]^\top(\vy_l^i-\tilde\vy_l^i)
\right\}\\
&=&0,
\end{eqnarray*}
\normalsize
where the last equality holds because of  \eqref{eq:mean_Ggx} and how $\tilde\by_k$ is generated, i.e., \eqref{eq:tildeyk}. This equation implies 
\small
\begin{eqnarray*}
&&\mathbb{E}\left(\sum_{k=0}^{K-1}t_k\left[\vdelta_k^{i,1}+\theta_k \left(\vdelta_k^{i,1}-\vdelta_k^{i,2}\right)\right]^\top(\vy_k^i-\tilde\vy_k^i)\right)^2\\
&=&\sum_{k=0}^{K-1}t_k^2\mathbb{E}\left(\left[\vdelta_k^{i,1}+\theta_k \left(\vdelta_k^{i,1}-\vdelta_k^{i,2}\right)\right]^\top(\vy_k^i-\tilde\vy_k^i)\right)^2\\
&\leq& \mathbb{E}\sum_{k=0}^{K-1}t_k^2\mathbb{E}\left(\left\|\vdelta_k^{i,1}+\theta_k \left(\vdelta_k^{i,1}-\vdelta_k^{i,2}\right)\right\|^2\|\vy_k^i-\tilde\vy_k^i\|^2\right)\\
&\leq& \sum_{k=0}^{K-1}t_k^2\mathbb{E}\left\|\vdelta_k^{i,1}+\theta_k \left(\vdelta_k^{i,1}-\vdelta_k^{i,2}\right)\right\|^2D_y^2\\
&\leq& 10\sigma_{g'}^2D_y^2\sum_{k=0}^{K-1}t_k^2,
\end{eqnarray*}
\normalsize
where the first inequality by the Cauchy-Schwartz inequality, the second inequality comes from Assumption~\ref{assume:basic}b, and the last is because $\mathbb{E}\left\|\vdelta_k^{i,1}+\theta_k \left(\vdelta_k^{i,1}-\vdelta_k^{i,2}\right)\right\|^2\leq (2(1+\theta_k)^2+2\theta_k^2)\sigma_{g'}^2\leq10\sigma_{g'}^2$ given \eqref{eq:var_Ggy} and the fact that $\theta_k\in[0, 1]$. Summing the inequality above for $i=1,2,\dots,m$ gives 
\small
\begin{eqnarray*}
&&\mathbb{E}\left\|\sum_{k=0}^{K-1}t_k\left[\vDelta_k^1+\theta_k \left(\vDelta_k^1-\vDelta_k^2\right)\right](\vy_k-\tilde\vy_k)\right\|_\infty^2\\
&\leq &\mathbb{E}\left\|\sum_{k=0}^{K-1}t_k\left[\vDelta_k^1+\theta_k \left(\vDelta_k^1-\vDelta_k^2\right)\right](\vy_k-\tilde\vy_k)\right\|^2\\
&\leq &
10m\sigma_{g'}^2D_y^2\sum_{k=0}^{K-1}t_k^2.
\end{eqnarray*}
\normalsize
Using this inequality and the Cauchy-Schwartz inequality, we can show that
\small
\begin{eqnarray}
\nonumber
&&\sum_{k=0}^{K-1}t_k\mathbb{E}\vlambda^\top\left[\vDelta_k^1+\theta_k \left(\vDelta_k^1-\vDelta_k^2\right)\right](\vy_k-\tilde\vy_k)\\\nonumber
&\leq &\sqrt{\mathbb{E}\|\vlambda\|_1^2}
\sqrt{\mathbb{E}\left\|\sum_{k=0}^{K-1}t_k\left[\vDelta_k^1+\theta_k \left(\vDelta_k^1-\vDelta_k^2\right)\right](\vy_k-\tilde\vy_k)\right\|_\infty^2}\\\label{eq:zeromean_12}
&\leq &\sqrt{\mathbb{E}\|\vlambda\|_1^2}
\sqrt{10m\sigma_{g'}^2D_y^2\sum_{k=0}^{K-1}t_k^2}.
\end{eqnarray}
\normalsize
Recall \eqref{eq:zeromean} and the fact that  $\|\vlambda\|_1\leq \Lambda$ almost surely. Applying \eqref{eq:zeromean_6}, \eqref{eq:zeromean_345} and \eqref{eq:zeromean_12}, we have
\small
\begin{eqnarray}
\label{eq:zeromean_sum_stoc}
\mathbb{E}\sum_{k=0}^{K-1}t_k\mathcal{T}_{\text{bias},k}
&\leq&\sqrt{\mathbb{E}\|\vlambda\|_1^2}\sqrt{10m\sigma_{g'}^2D_y^2\sum_{k=0}^{K-1}t_k^2}
\leq\Lambda
\sqrt{10m\sigma_{g'}^2D_y^2\sum_{k=0}^{K-1}t_k^2}
\end{eqnarray}
\normalsize

Next, we want to bound $\mathbb{E}\sum_{k=0}^{K-1}t_k\mathcal{T}_{\text{variance},k}$ from above.
Recall the definitions in \eqref{eq:ell}, \eqref{eq:ellg}, \eqref{eq:qk} and \eqref{eq:barqk}. By Jensen's inequality, we have
\small
\begin{eqnarray}
\nonumber
&&\mathbb{E}\left\|\ell_{\vG}(\vx_k,\xi_k^3;\vx_{k-1},\vy_{k+1})-\ell_{\vg}(\vx_k;\vx_{k-1},\vy_{k+1})+\theta_k( \bq_{k}-\bar\bq_{k})\right\|^2\\\nonumber
&\leq&2\mathbb{E}\|\ell_{\vG}(\vx_k,\xi_k^3;\vx_{k-1},\vy_{k+1})-\ell_{\vg}(\vx_k;\vx_{k-1},\vy_{k+1})\|^2+2\mathbb{E}\|\bq_{k}-\bar\bq_{k}\|^2\\\nonumber
&=&2\mathbb{E}\|\vG(\vx_{k-1},\vy_{k+1},\xi_k^3)+\nabla_x\vG(\vx_{k-1},\vy_{k+1},\xi_k^3)(\vx_{k}-\vx_{k-1})\\\nonumber
&&\quad\quad-\vg(\vx_{k-1},\vy_{k+1})+\nabla_x \vg(\vx_{k-1},\vy_{k+1})(\vx_{k}-\vx_{k-1})\|^2\\\nonumber
&&+2\mathbb{E}\|\ell_{\vg}(\vx_{k};\vx_{k-1},\vy_{k})-\ell_{\vg}(\vx_{k-1};\vx_{k-2},\vy_{k})-\ell_{\vG}(\vx_{k},\xi_k^4;\vx_{k-1},\vy_k)+\ell_{\vG}(\vx_{k-1},\xi_k^5;\vx_{k-2},\vy_k)\|^2\\\nonumber
&\leq&4\mathbb{E}\|\vG(\vx_{k-1},\vy_{k+1},\xi_k^3)-\vg(\vx_{k-1},\vy_{k+1})\|^2\\\nonumber
&&+4\mathbb{E}\|\nabla_x\vG(\vx_{k-1},\vy_{k+1},\xi_k^3)(\vx_{k}-\vx_{k-1})-\nabla_x \vg(\vx_{k-1},\vy_{k+1})(\vx_{k}-\vx_{k-1})\|^2\\\nonumber
&&+8\mathbb{E}\|\vG(\vx_{k-1},\vy_{k},\xi_k^4)-\vg(\vx_{k-1},\vy_{k})\|^2\\\nonumber
&&+8\mathbb{E}\|\nabla_x\vG(\vx_{k-1},\vy_{k},\xi_k^4)(\vx_{k}-\vx_{k-1})-\nabla_x \vg(\vx_{k-1},\vy_{k})(\vx_{k}-\vx_{k-1})\|^2\\\nonumber
&&+8\mathbb{E}\|\vG(\vx_{k-2},\vy_{k},\xi_k^5)-\vg(\vx_{k-2},\vy_{k})\|^2\\\nonumber
&&+8\mathbb{E}\|\nabla_x\vG(\vx_{k-2},\vy_{k},\xi_k^5)(\vx_{k-1}-\vx_{k-2})-\nabla_x \vg(\vx_{k-2},\vy_{k})(\vx_{k-1}-\vx_{k-2})\|^2\\\label{eq:variance1}
&\leq&20m\sigma_g^2+12m\sigma_{g'}^2\mathbb{E}\|\vx_{k}-\vx_{k-1}\|^2+8m\sigma_{g'}^2\mathbb{E}\|\vx_{k-1}-\vx_{k-2}\|^2,
\end{eqnarray}
\normalsize
where the last inequality holds because of Cauchy-Schwartz inequality, \eqref{eq:var_Gi} and \eqref{eq:var_Ggx}. Let $\vlambda_{k+1}=(\lambda_{k+1,1},\dots,\lambda_{k+1,m})^\top$. By \eqref{eq:var_Ggx} and Jensen's inequality, we have 
\small
\begin{eqnarray}
\nonumber
\mathbb{E}\|(\vDelta_k^6)^\top\vlambda_{k+1}\|^2
&\leq&
\mathbb{E}\sum_{i=1}^m\frac{\lambda_{k+1,i}\|\nabla_x G_i(\vx_{k},\vy_{k+1}^i,\xi_k^6)-\nabla_x g_i(\vx_{k},\vy_{k+1}^i)\|^2}{\|\vlambda_{k+1}\|_1}\|\vlambda_{k+1}\|_1^2\\\nonumber
&\leq&\mathbb{E}\mathbb{E}_{k-1}\left[\sum_{i=1}^m\frac{\lambda_{k+1,i}\|\nabla_x G_i(\vx_{k},\vy_{k+1}^i,\xi_k^6)-\nabla_x g_i(\vx_{k},\vy_{k+1}^i)\|^2}{\|\vlambda_{k+1}\|_1}\|\vlambda_{k+1}\|_1^2\Big| \vlambda_{k+1},\vy_{k+1}\right]\\\nonumber
&\leq&\sigma_{g'}^2\mathbb{E}\|\vlambda_{k+1}\|_1^2
\end{eqnarray}
\normalsize
According to this inequality and \eqref{eq:var_gf}, we have 
\small
\begin{eqnarray}
\nonumber
\mathbb{E}\left\|\vdelta_k^6+(\vDelta_k^6)^\top\vlambda_{k+1}\right\|^2
&\leq& 2\mathbb{E}\|\vdelta_k^6\|^2+ 2\mathbb{E}\|(\vDelta_k^6)^\top\vlambda_{k+1}\|^2\\\nonumber
&\leq&2\sigma_{f'}^2+2\sigma_{g'}^2\mathbb{E}\|\vlambda_{k+1}\|_1^2\\\nonumber
&\leq&2\sigma_{f'}^2+4\sigma_{g'}^2\mathbb{E}\|\vlambda_{k+1}-\vlambda\|_1^2+4\sigma_{g'}^2\mathbb{E}\|\vlambda\|_1^2.\\\label{eq:variance2}
&\leq&2\sigma_{f'}^2+4m\sigma_{g'}^2\mathbb{E}\|\vlambda_{k+1}-\vlambda\|^2+4\sigma_{g'}^2\Lambda^2.
\end{eqnarray}
\normalsize
Since $\|\vlambda\|_1\leq \Lambda$ almost surely, by \eqref{eq:var_Ggy} and the fact that $\theta_k\in[0, 1]$,  we have, for $i=1,2,\dots,m$, 
\small
$$
\mathbb{E}\lambda_i\left\|\vdelta_k^{i,1}+\theta_k \left(\vdelta_k^{i,1}-\vdelta_k^{i,2}\right)\right\|^2\leq \Lambda(2(1+\theta_k)^2+2\theta_k^2)\sigma_{g'}^2\leq10\Lambda\sigma_{g'}^2,
$$
\normalsize
which implies  
\small
\begin{eqnarray}
\label{eq:variance3}
\mathbb{E}\left\|\vDelta_k^1+\theta_k \left(\vDelta_k^1-\vDelta_k^2\right)\right\|_{\vlambda}^2
\leq 10m\Lambda\sigma_{g'}^2.
\end{eqnarray}
\normalsize
Recall \eqref{eq:variance}. Applying \eqref{eq:variance1}, \eqref{eq:variance2} and \eqref{eq:variance3}, we have
\small
\begin{eqnarray}
\label{eq:variance_sum_stoc}
&&\mathbb{E}\sum_{k=0}^{K-1}t_k\mathcal{T}_{\text{variance},k}\\\nonumber
&\leq&\left(2\sigma_{f'}^2+4\sigma_{g'}^2\Lambda^2\right)\sum_{k=0}^{K-1}\frac{t_k}{\tau_k}+15m\Lambda\sigma_{g'}^2\sum_{k=0}^{K-1}\frac{t_k}{\sigma_k}+30m\sigma_g^2\sum_{k=0}^{K-1}\frac{t_k}{\gamma_k}\\\nonumber
&&+4m\sigma_{g'}^2\sum_{k=0}^{K-1}\frac{t_k}{\tau_k}\mathbb{E}\|\vlambda_{k+1}-\vlambda\|^2+\sum_{k=0}^{K-1}\left[\frac{18m\sigma_{g'}^2t_{k+1}}{\gamma_{k+1}}+\frac{12m\sigma_{g'}^2t_{k+2}}{\gamma_{k+2}}\right]\mathbb{E}\|\vx_{k+1}-\vx_k\|^2.
\end{eqnarray}
\normalsize
Adding \eqref{eq:crossterm_sum2}, \eqref{eq:change_sum1}, \eqref{eq:distance_sum}, \eqref{eq:zeromean_sum_stoc} and \eqref{eq:variance_sum_stoc} gives the following inequality
\small
\begin{align}
\nonumber
&\mathbb{E}\sum_{k=0}^{K-1}t_k[\phi(\bx_{k+1},\vlambda,\by) - \phi(\bx,\vlambda_{k+1},\by_{k+1})]\\\label{eq:Qsum5_stoc}
\leq&\frac{t_0\tau_0-t_0\mu_f}{2}\|\bx-\bx_0\|^2+t_0\gamma_0\|\vlambda-\vlambda_0\|^2+t_0\sigma_0\|\vy-\vy_0\|_{\vlambda}^2\\\nonumber
&+\sum_{k=0}^{K-1}\left[\frac{t_kL_f}{2}+\frac{6mt_{k+1}\theta_{k+1}^2(M_g^x)^2}{\gamma_{k+1}}+\frac{24mt_{k+2}\theta_{k+2}^2(M_g^x)^2}{\gamma_{k+2}}+\frac{18m\sigma_{g'}^2t_{k+1}}{\gamma_{k+1}}+\frac{12m\sigma_{g'}^2t_{k+2}}{\gamma_{k+2}}-\frac{t_k\tau_k}{16}\right]\\\nonumber
&\quad\quad\quad\times\mathbb{E}\|\vx_{k+1}-\vx_{k}\|^2\\\nonumber
&+\sum_{k=0}^{K-1}\mathbb{E}\left\{\left[\frac{ 7t_{k+1}(\theta_{k+1}L^{yx}_g)^2\|\vlambda\|_1}{\sigma_{k+1}}+\frac{t_kL_g^{xx}\|\vlambda\|_1}{2}-\frac{t_k\tau_k}{16}\right]\|\bx_{k+1}-\bx_k\|^2\right\}\\\nonumber
& +\sum_{k=0}^{K-1}\left[\frac{7t_{k+1}(\theta_{k+1}L^{yy}_g)^2}{\sigma_{k+1}}-\frac{t_k\sigma_k}{8}\right]\mathbb{E}\|\vy_{k+1}-\vy_k\|_{\vlambda}^2 - \sum_{k=0}^{K-1}\frac{\gamma_k}{16}\mathbb{E}\|\vlambda_{k+1}-\vlambda_k\|^2\\\nonumber
& +\sum_{k=0}^{K-1}\left[\frac{4t_k(H_g^{x})^2\mathbb{E}\|\vlambda\|_1^2}{\tau_k}+\frac{4t_k(H_f)^2}{\tau_k}+\frac{4t_k(\theta_kH_g^y)^2\mathbb{E}\|\vlambda\|_1}{\sigma_k}+\frac{3t_{k+1}(\theta_{k+1}H_g^y)^2\|\vlambda\|_1}{\sigma_{k+1}}\right]\\\nonumber
&+\sum_{k=0}^{K-1}\frac{t_{k+1}\tau_{k+1}-t_{k+1}\mu_f-t_k\tau_k}{2}\mathbb{E}\|\bx-\bx_{k+1}\|^2+\sum_{k=0}^{K-1}\left(\frac{t_{k+1}\gamma_{k+1}-t_k\gamma_k}{2}+\frac{4m\sigma_{g'}^2t_k}{\tau_k}\right)\mathbb{E}\|\vlambda-\vlambda_{k+1}\|^2 \\\nonumber
& +\sum_{k=0}^{K-1}\frac{t_{k+1}\sigma_{k+1}-t_k\sigma_k-t_k\mu_y}{2}\mathbb{E}\|\vy-\vy_{k+1}\|_{\vlambda}^2-\frac{t_K\gamma_K}{4}\mathbb{E}\|\vlambda-\vlambda_K\|^2-\frac{t_K\sigma_K}{4}\mathbb{E}\|\vy-\vy_K\|_{\vlambda}^2 \\\nonumber
&+\sum_{k=0}^{K-1}\frac{t_{k+1}\gamma_{k+1}-t_k\gamma_k}{2}\mathbb{E}\|\vlambda-\tilde\vlambda_{k+1}\|^2  +\sum_{k=0}^{K-1}\frac{t_{k+1}\sigma_{k+1}-t_k\sigma_k}{2}\mathbb{E}\|\vy-\tilde\vy_{k+1}\|_{\vlambda}^2\\\nonumber
&-\frac{t_K\gamma_K}{2}\mathbb{E}\|\vlambda-\tilde\vlambda_K\|^2-\frac{t_K\sigma_K}{2}\mathbb{E}\|\vy-\tilde\vy_K\|_{\vlambda}^2 \\\nonumber
&+\Lambda
\sqrt{10m\sigma_{g'}^2D_y^2\sum_{k=0}^{K-1}t_k^2}+\left(2\sigma_{f'}^2+4\sigma_{g'}^2\Lambda^2\right)\sum_{k=0}^{K-1}\frac{t_k}{\tau_k}+15m\Lambda\sigma_{g'}^2\sum_{k=0}^{K-1}\frac{t_k}{\sigma_k}+30m\sigma_g^2\sum_{k=0}^{K-1}\frac{t_k}{\gamma_k}.
\end{align}
\normalsize
Conclusion \eqref{eq:Qsum5} is thus obtained from \eqref{eq:Qsum5_stoc} after applying conditions \eqref{eq:determine_cond2}, \eqref{eq:determine_cond3}, \eqref{eq:determine_cond4_a}, \eqref{eq:determine_cond4_b} and \eqref{eq:determine_cond4_c}, applying condition $\|\vlambda\|_1\leq \Lambda$, and dropping non-positive terms on the right-hand side of  \eqref{eq:Qsum5_stoc}.
\qedsymbol{}
\end{proof}

\subsection{Convergence properties in deterministic case}
In this section, we present the convergence properties of Algorithm~\ref{alg:PDHG}, which is a deterministic special case of Algorithm~\ref{alg:SPDHG} when
$$
\sigma_{f}=\sigma_{g'}=\sigma_{g}=0
$$
in Assumption~\ref{assume:oracle}, and thus is subject to Proposition~\ref{thm:mainprop_deterministic}. Note that \eqref{eq:tildetermzero} holds in this case, which simplifies \eqref{eq:Qsum5}. The following corollary of Proposition~\ref{thm:mainprop_deterministic} thus applies to Algorithm~\ref{alg:PDHG}, which provides upper bounds on the objective gap and constraint violation of $\bar\vx_{K}$.
\begin{corollary}
\label{thm:mainprop_deterministic_simplified}
Suppose  parameters $t_k$, $\theta_k$, $\tau_k$, $\sigma_k$ and $\gamma_k$ in Algorithm~\ref{alg:PDHG} satisfies \eqref{eq:thetaktk}, \eqref{eq:determine_cond3}, \eqref{eq:determine_cond4_a}, \eqref{eq:determine_cond4_b}, and the following conditions
\small
\begin{align}
\label{eq:determine_cond2_simplified}
&\frac{t_kL_f}{2}+\frac{6mt_{k+1}\theta_{k+1}^2(M_g^x)^2}{\gamma_{k+1}}+\frac{24mt_{k+2}\theta_{k+2}^2(M_g^x)^2}{\gamma_{k+2}}- \frac{t_k\tau_k}{16}\leq0,\\
\label{eq:determine_cond4_c_simplified}
&t_{k+1}\gamma_{k+1}-t_k\gamma_k\leq 0
\end{align}
\normalsize
for $k=0,1,\dots,K$. It holds that
\small
\begin{align}
\label{eq:choiceinQ1}
f(\bar\vx_{K})-f(\bx^*)
\leq&\frac{1}{\sum_{k=0}^{K-1}t_k}\frac{t_0\tau_0}{2}\|\vx^*-\vx_0\|^2
+\frac{1}{\sum_{k=0}^{K-1}t_k}\sum_{k=0}^{K-1}\frac{4t_k(H_f)^2}{\tau_k}
\end{align}
\normalsize
and
\small
\begin{align}
\label{eq:choiceinQ2}
\|[\vg^*(\bar\vx_{K})]_+\|_\infty \leq&
\frac{1}{\sum_{k=0}^{K-1}t_k}\left[\frac{t_0\tau_0}{2}\|\bx-\bx_0\|^2+\frac{t_0\gamma_0}{2} (\|\vlambda^*\|_1+1)^2+\frac{t_0\sigma_0}{2}D_y^2(\|\vlambda^*\|_1+1)+ \mathcal{T}_{\|\vlambda^*\|_1+1,K}\right]\\\nonumber
&+\frac{4\left[(H_g^{x})^2(\|\vlambda^*\|_1+1)^2+(H_f)^2\right]}{\sum_{k=0}^{K-1}t_k}\sum_{k=0}^{K-1}\frac{t_k}{\tau_k}+\frac{7(H_g^y)^2(\|\vlambda^*\|_1+1)}{\sum_{k=0}^{K-1}t_k}\sum_{k=0}^{K}\frac{t_k\theta_k^2}{\sigma_k},
\end{align}
\normalsize
where $\mathcal{T}_{\|\vlambda^*\|_1+1,K}$ is defined in \eqref{eq:TK} with $\Lambda=\|\vlambda^*\|_1+1$.
\end{corollary}
\begin{proof}
Conditions \eqref{eq:determine_cond2_simplified} and \eqref{eq:determine_cond4_c_simplified} imply conditions  \eqref{eq:determine_cond2} and \eqref{eq:determine_cond4_c} in the deterministic case. Therefore, by Proposition~\ref{thm:mainprop_deterministic}, it holds, for any deterministic $\bx\in\cX$, $\by\in\cY$ and  $\vlambda\in\mathbb{R}_+^m$ satisfying $\|\vlambda\|_1\leq \Lambda$, that
\small
\begin{align}
\nonumber
& \sum_{k=0}^{K-1}t_k[\phi(\bx_{k+1},\vlambda,\by) - \phi(\bx,\vlambda_{k+1},\by_{k+1})]\\\label{eq:Qsum5_simplified}
\leq&\frac{t_0\tau_0}{2}\|\bx-\bx_0\|^2+\frac{t_0\gamma_0}{2} \|\vlambda-\vlambda_0\|^2+\frac{t_0\sigma_0}{2} \|\vy-\vy_0\|_{\vlambda}^2+ \mathcal{T}_{\Lambda,K}\\\nonumber
& +4\left[(H_g^{x})^2\Lambda^2+(H_f)^2\right]\sum_{k=0}^{K-1}\frac{t_k}{\tau_k}+7(H_g^y)^2\Lambda\sum_{k=0}^{K}\frac{t_k\theta_k^2}{\sigma_k},
\end{align}
\normalsize
where $\mathcal{T}_{\Lambda,K}$ is defined in \eqref{eq:TK}. Recall that $\vlambda_0=\mathbf{0}$ and observe that $\mathcal{T}_{0,K}\leq 0$. Applying \eqref{eq:Qsum5_simplified} to the right-hand side of \eqref{eq:Qsum_averaged_obj} with $\Lambda=0$ leads to \eqref{eq:choiceinQ1}. Recall \eqref{eq:lambda_for_barxK}, \eqref{eq:y_for_barxK} and \eqref{eq:lambda_for_barxK_bound}. Inequality \eqref{eq:choiceinQ2} is thus obtained by applying \eqref{eq:Qsum5_simplified} to the right-hand side of \eqref{eq:Qsum_averaged_cst} with $\Lambda= \|\vlambda^*\|_1+1$ and using the facts that $\|\widehat\vlambda_K-\vlambda_0\|=\|\widehat\vlambda_K\|\leq\|\vlambda^*\|_1+1$ and that $\|\widehat\by_K-\vy_0\|_{\widehat\vlambda_K}^2\leq \|\widehat\vlambda_K\|_1D_y^2$. 
\qedsymbol{}
\end{proof}

By choosing parameters $t_k$, $\theta_k$, $\tau_k$, $\sigma_k$ and $\gamma_k$ in Corollary~\ref{thm:mainprop_deterministic_simplified}  appropriately, we can obtain the convergence properties of Algorithm~\ref{alg:PDHG} in different scenarios in the following theorems.

\begin{theorem}
\label{thm:maintheorem_deterministic}
Suppose $\mu_f>0$, $\mu_y>0$, and parameters $t_k$, $\theta_k$, $\tau_k$, $\sigma_k$ and $\gamma_k$ in Algorithm~\ref{alg:PDHG} are chosen as
\begin{eqnarray*}
    &t_k=k+k_0+1,\quad \theta_k=\frac{k+k_0}{k+k_0+1},\\
    &\tau_k =\frac{(k+k_0+1) \mu_f}{2},\quad \sigma_k =\frac{(k+k_0) \mu_y}{2},\quad    \gamma_k=\frac{3456m(M_g^x)^2}{\mu_f(k+k_0+1)},
\end{eqnarray*}
where $k_0\geq \max\left\{32 L_f/\mu_f,\sqrt{224}L^{yy}_g/\mu_y\right\}$. Then 
\small
\begin{align}
\label{eq:objective_deterministic1}
f(\bar\vx_{K})-f(\bx^*)\leq \frac{(k_0+1)^2\mu_f\|\vx^*-\vx_0\|^2}{2K(K+1+2k_0)}+\frac{16(H_f)^2}{\mu_f(K+2k_0+1)}
\end{align}
\normalsize
and
\small
\begin{align}
\label{eq:constraint_deterministic1}
\|[\vg^*(\bar\vx_{K})]_+\|_\infty
\leq&\frac{(k_0+1)^2\mu_f\|\vx^*-\vx_0\|^2}{2K(K+1+2k_0)}+\frac{3456m (M_g^x)^2(\|\vlambda^*\|_1+1)^2}{\mu_f(K(K+1+2k_0))}+\frac{k_0(k_0+1)\mu_yD_y^2(\|\vlambda^*\|_1+1)}{2K(K+1+2k_0)}\\\nonumber
&+ \frac{16\left[(H_g^{x})^2(\|\vlambda^*\|_1+1)^2+(H_f)^2\right]}{\mu_f((K+1)+2k_0)}+\frac{56(H_g^y)^2(\|\vlambda^*\|_1+1)}{\mu_y((K+1)+2k_0)} + \frac{2\mathcal{T}_{\|\vlambda^*\|_1+1,K}}{K(K+1+2k_0)},
\end{align}
\normalsize
where $\mathcal{T}_{\|\vlambda^*\|_1+1,K}$ is defined in \eqref{eq:TK} and is bounded by a constant independent of $K$, i.e.,  
$\mathcal{T}_{\|\vlambda^*\|_1+1,K}=O(1)$. In addition, $\mathcal{T}_{\|\vlambda^*\|_1+1,K}\leq 0$ when
\begin{eqnarray}
\label{eq:k0removehatK}
k_0\geq \max\left\{ \sqrt{\frac{448(L^{yx}_g)^2(\|\vlambda^*\|_1+1)}{\mu_y\mu_f}},~~\frac{32L_g^{xx}(\|\vlambda^*\|_1+1)}{\mu_f}\right\}-1.
\end{eqnarray}
\end{theorem}
\begin{proof}
It is easy to verify that \eqref{eq:thetaktk}, \eqref{eq:determine_cond4_a}, \eqref{eq:determine_cond4_b} and \eqref{eq:determine_cond4_c_simplified} hold.
According to the definitions of $t_k$, $\theta_k$, $\tau_k$, $\sigma_k$ and $\gamma_k$, we have 
\begin{eqnarray}
\nonumber
&&\frac{t_kL_f}{2}+\frac{6mt_{k+1}(\theta_{k+1}M_g^x)^2}{\gamma_{k+1}}+\frac{24mt_{k+2}(\theta_{k+2}M_g^x)^2}{\gamma_{k+2}}\\\nonumber
&=&\frac{t_kL_f}{2}+\frac{6m(k+k_0+1)^2(M_g^x)^2}{3456 m(M_g^x)^2/\mu_f}+\frac{24m(k+k_0+2)^2(M_g^x)^2}{3456m (M_g^x)^2/\mu_f}\\\label{eq:determine_cond2_verified}
&\leq&\frac{t_k\tau_k}{32}+\frac{54m(k+k_0+1)^2(M_g^x)^2}{3456 m(M_g^x)^2/\mu_f}=\frac{t_k\tau_k}{16},
\end{eqnarray}
where the inequality is because $\tau_k\geq k_0\mu_f/2\geq 16 L_f$ and $k+k_0+2\leq 2(k+k_0+1)$. This verifies \eqref{eq:determine_cond2_simplified}. In addition, 
\begin{eqnarray}
\label{eq:determine_cond3_verified}
\frac{7t_{k+1}(\theta_{k+1}L^{yy}_g)^2}{\sigma_{k+1}}
=\frac{14(k+k_0+1)(L^{yy}_g)^2}{(k+k_0+2)\mu_y}
\leq \frac{k_0^2\mu_y}{16}
\leq \frac{t_k\sigma_k}{8},
\end{eqnarray}
where the first inequality is because $k_0\geq \sqrt{224}L^{yy}_g/\mu_y$. This verifies \eqref{eq:determine_cond3}. Hence, \eqref{eq:choiceinQ1} and \eqref{eq:choiceinQ2} hold by Corollary~\ref{thm:mainprop_deterministic_simplified}. 

By the definitions of $t_k$, $\theta_k$, $\tau_k$, $\sigma_k$ and $\gamma_k$, we have
\begin{align}
\label{eq:sumtk}
\sum_{k=0}^{K-1}t_k=&\frac{K(K+1+2k_0)}{2}\\
\label{eq:sumtktauk}
\sum_{k=0}^{K-1}\frac{t_k}{\tau_k}=&\frac{2K}{\mu_f}\\
\label{eq:sumtkthetaksigmak}
\sum_{k=0}^{K}\frac{t_k\theta_k^2}{\sigma_k}\leq&\frac{2(K+1)}{\mu_y}\leq \frac{4K}{\mu_y}.
\end{align}
Conclusions \eqref{eq:objective_deterministic1} and \eqref{eq:constraint_deterministic1} by applying \eqref{eq:sumtk}, \eqref{eq:sumtktauk} and \eqref{eq:sumtkthetaksigmak} as well as the definition of $t_0$, $\tau_0$, $\sigma_0$ and $\gamma_0$ to the right-hand sides of \eqref{eq:choiceinQ1} and \eqref{eq:choiceinQ2}.

To show the second part of the conclusion,  we observe that 
\small
\begin{align}
\nonumber
\mathcal{T}_{\|\vlambda^*\|_1+1,K}
=& \sum_{k=0}^{K-1}\left[\frac{ 7t_{k+1}(\theta_{k+1}L^{yx}_g)^2(\|\vlambda^*\|_1+1)}{\sigma_{k+1}}+\frac{t_kL_g^{xx}(\|\vlambda^*\|_1+1)}{2}-\frac{t_k\tau_k}{16}\right]\|\bx_{k+1}-\bx_k\|^2\\\label{eq:Tlambda_deterministic}
\leq & \sum_{k=0}^{\widehat K-1}\left[\frac{ 7t_{k+1}(\theta_{k+1}L^{yx}_g)^2(\|\vlambda^*\|_1+1)}{\sigma_{k+1}}+\frac{t_kL_g^{xx}(\|\vlambda^*\|_1+1)}{2}-\frac{t_k\tau_k}{16}\right]\|\bx_{k+1}-\bx_k\|^2,
\end{align}
\normalsize
where 
\begin{eqnarray}
\label{eq:k0hatK_deterministic}
\widehat K=\left\lceil\max\left\{ \sqrt{\frac{448(L^{yx}_g)^2(\|\vlambda^*\|_1+1)}{\mu_y\mu_f}},~~\frac{32L_g^{xx}(\|\vlambda^*\|_1+1)}{\mu_f}\right\}-k_0-1\right\rceil.
\end{eqnarray}
In fact, when $k\geq\widehat K$, it can be verified easily that 
\begin{eqnarray*}
\frac{ 7t_{k+1}(\theta_{k+1}L^{yx}_g)^2(\|\vlambda^*\|_1+1)}{\sigma_{k+1}}+\frac{t_kL_g^{xx}(\|\vlambda^*\|_1+1)}{2} \leq \frac{t_k\tau_k}{32}+\frac{t_k\tau_k}{32} =\frac{t_k\tau_k}{16},
\end{eqnarray*}
so the summands in the definition of $\mathcal{T}_{\|\vlambda^*\|_1+1,K}$ with indexes $k\geq\widehat K$ are less than or equal to zero. The right-hand side of  \eqref{eq:Tlambda_deterministic} is a constant independent of $K$, so $\mathcal{T}_{\|\vlambda^*\|_1+1,K}=O(1)$. In addition, we have $\widehat K\leq 0$ and thus $\mathcal{T}_{\|\vlambda^*\|_1+1,K}\leq 0$ when \eqref{eq:k0removehatK} holds.
\qedsymbol{}
\end{proof}

Suppose the problem is smooth, which is the case when $H_g^{x}=H_g^{y}=H_f=0$. Theorem~\ref{thm:maintheorem_deterministic} shows that Algorithm~\ref{alg:PDHG} finds an $\epsilon$-optimal solution for \eqref{eq:SIP} in $O(\epsilon^{-0.5})$ iterations when $\mu_f>0$ and $\mu_y>0$. This iteration complexity increases to $O(\epsilon^{-1})$ when the problem is nonsmooth, i.e., when at least one of $H_g^{x}$, $H_g^{y}$ and $H_f$ is positive.

\begin{theorem}
\label{thm:maintheorem_deterministic_muy0}
Suppose $\mu_f>0$, $\mu_y=0$, and parameters $t_k$, $\theta_k$, $\tau_k$, $\sigma_k$ and $\gamma_k$ in Algorithm~\ref{alg:PDHG} are chosen as
\begin{eqnarray*}
    &t_k=k+k_0+1,\quad \theta_k=\frac{k+k_0}{k+k_0+1},\\
    &\tau_k =\frac{(k+k_0+1) \mu_f}{2},\quad \sigma_k =\frac{\max\{\sqrt{56}(K+k_0)L^{yy}_g,(H_g^y/D_y)K^{\frac{1}{2}}(K+k_0)\}}{k+k_0+1},\quad    \gamma_k=\frac{3456m (M_g^x)^2}{\mu_f(k+k_0+1)},
\end{eqnarray*}
where $k_0\geq 18 L_f/\mu_f$. Then \eqref{eq:objective_deterministic1} holds and 
\small
\begin{align}
\label{eq:constraint_deterministic1_muy0}
\|[\vg^*(\bar\vx_{K})]_+\|_\infty
\leq&\frac{(k_0+1)^2\mu_f\|\vx^*-\vx_0\|^2}{2K(K+1+2k_0)}+\frac{3456m (M_g^x)^2(\|\vlambda^*\|_1+1)^2}{\mu_fK(K+1+2k_0)}
+\frac{\sqrt{56}L^{yy}_gD_y^2(\|\vlambda^*\|_1+1)}{K}\\\nonumber
&+\frac{16\left[(H_g^{x})^2(\|\vlambda^*\|_1+1)^2+(H_f)^2\right]}{\mu_f(K+1+2k_0)}+\frac{29H_g^yD_y(\|\vlambda^*\|_1+1)}{\sqrt{K}} + \frac{2\mathcal{T}_{\|\vlambda^*\|_1+1,K}}{K(K+1+2k_0)},
\end{align}
\normalsize
where $\mathcal{T}_{\|\vlambda^*\|_1+1,K}$ is defined in \eqref{eq:TK} and is bounded by a constant independent of $K$, i.e.,  
$\mathcal{T}_{\|\vlambda^*\|_1+1,K}=O(1)$. In addition, $\mathcal{T}_{\|\vlambda^*\|_1+1,K}\leq 0$ when
\begin{eqnarray}
\label{eq:k0removehatK_muy0}
k_0\geq \max\left\{ \frac{112(L^{yx}_g)^2(\|\vlambda^*\|_1+1)}{\sqrt{14}L^{yy}_g\mu_f},~~\frac{16L_g^{xx}(\|\vlambda^*\|_1+1)}{\mu_f}\right\}-1.
\end{eqnarray}
\end{theorem}
\begin{proof}
It is easy to verify that \eqref{eq:thetaktk}, \eqref{eq:determine_cond4_a}, \eqref{eq:determine_cond4_b} and \eqref{eq:determine_cond4_c_simplified} hold.
According to the definitions of $t_k$, $\theta_k$, $\tau_k$, $\sigma_k$ and $\gamma_k$, \eqref{eq:determine_cond2_verified} still holds. So does \eqref{eq:determine_cond2_simplified}. In addition,
\begin{eqnarray*}
\frac{7t_{k+1}(\theta_{k+1}L^{yy}_g)^2}{\sigma_{k+1}}
\leq\frac{7(k+k_0+1)^2(L^{yy}_g)^2}{\sqrt{56}(K+k_0)L^{yy}_g}
\leq \frac{\sqrt{56}(K+k_0)L^{yy}_g}{8}
\leq \frac{t_k\sigma_k}{8},
\end{eqnarray*}
where the first and third inequalities are because of the definitions of $t_k$, $\theta_k$ and $\sigma_k$, and the second inequality is because $k\leq K-1$. This verifies \eqref{eq:determine_cond3}. Hence, \eqref{eq:choiceinQ1} and \eqref{eq:choiceinQ2} hold by Corollary~\ref{thm:mainprop_deterministic_simplified}. 

By the definitions of $t_k$, $\theta_k$, $\tau_k$, $\sigma_k$ and $\gamma_k$, we can show \eqref{eq:sumtk} and \eqref{eq:sumtktauk} as well as
\begin{align}
\label{eq:sumtkthetaksigmak_muy0}
\sum_{k=0}^{K}\frac{t_k\theta_k^2}{\sigma_k}= &\sum_{k=0}^{K}\frac{(k+k_0)^2}{(H_g^y/D_y)K^{\frac{1}{2}}(K+k_0)}
\leq \frac{2D_yK^{\frac{1}{2}}(K+k_0)}{H_g^y},
\end{align}
where the inequality is because $\sum_{k=0}^{K}(k+k_0)^2\leq(K+1)(K+k_0)^2\leq 2K(K+k_0)^2$. Conclusions \eqref{eq:objective_deterministic1} and \eqref{eq:constraint_deterministic1_muy0} are obtained by applying \eqref{eq:sumtk}, \eqref{eq:sumtktauk} and \eqref{eq:sumtkthetaksigmak_muy0} as well as the definition of $t_0$, $\tau_0$, $\sigma_0$ and $\gamma_0$ to the right-hand sides of \eqref{eq:choiceinQ1} and \eqref{eq:choiceinQ2}.

To show the second part of the conclusion,  we observe that \eqref{eq:Tlambda_deterministic} holds with 
\begin{eqnarray}
\label{eq:k0hatK_deterministic_muy0}
\widehat K=\left\lceil\max\left\{ \frac{112(L^{yx}_g)^2(\|\vlambda^*\|_1+1)}{\sqrt{14}L^{yy}_g\mu_f},~~\frac{16L_g^{xx}(\|\vlambda^*\|_1+1)}{\mu_f}\right\}-k_0-1\right\rceil.
\end{eqnarray}
In fact, when $k\geq\widehat K$, we have 
\begin{eqnarray*}
&&\frac{ 7t_{k+1}(\theta_{k+1}L^{yx}_g)^2(\|\vlambda^*\|_1+1)}{\sigma_{k+1}}+\frac{t_kL_g^{xx}(\|\vlambda^*\|_1+1)}{2}-\frac{t_k\tau_k}{16}\\\nonumber
&\leq&\frac{7(k+k_0+1)^2(L^{yx}_g)^2(\|\vlambda^*\|_1+1)}{\sqrt{56}(K+k_0)L^{yy}_g}+\frac{(k+k_0+1)L_g^{xx}(\|\vlambda^*\|_1+1)}{2}-\frac{(k+k_0+1)^2\mu_f}{16} \\\nonumber
&\leq&\frac{7(k+k_0+1)(L^{yx}_g)^2(\|\vlambda^*\|_1+1)}{\sqrt{56}L^{yy}_g}+\frac{(k+k_0+1)L_g^{xx}(\|\vlambda^*\|_1+1)}{2}-\frac{(k+k_0+1)^2\mu_f}{16} \\\nonumber
&\leq &0,
\end{eqnarray*}
where the first inequality is by the choices of $t_k$, $\theta_k$, $\tau_k$ and $\sigma_k$, the second inequality is because $k\leq K-1$, and the last is because $k\geq\widehat K$ with $\widehat K$ in \eqref{eq:k0hatK_deterministic_muy0}. Hence, when $k\geq\widehat K$, the summands in the definition of $\mathcal{T}_{\|\vlambda^*\|_1+1,K}$ with indexes $k\geq\widehat K$ are less than or equal to zero. The right-hand side of  \eqref{eq:Tlambda_deterministic} is a constant independent of $K$, so $\mathcal{T}_{\|\vlambda^*\|_1+1,K}=O(1)$. In addition, we have $\widehat K\leq 0$ and thus $\mathcal{T}_{\|\vlambda^*\|_1+1,K}\leq 0$ when \eqref{eq:k0removehatK_muy0} holds.
\qedsymbol{}
\end{proof}

Theorem~\ref{thm:maintheorem_deterministic_muy0} shows that, when $\mu_f>0$ but $\mu_y=0$, Algorithm~\ref{alg:PDHG} finds an $\epsilon$-optimal solution for \eqref{eq:SIP} in $O(\epsilon^{-1})$ iterations. This iteration complexity holds for both smooth and non-smooth cases because of the dominating term, i.e., $\frac{\sqrt{56}L^{yy}_gD_y^2(\|\vlambda^*\|_1+1)}{K}=O(\frac{1}{K})$ in \eqref{eq:constraint_deterministic1_muy0}. This complexity is reasonable because, when $\mu_y=0$, $g_i^*(\bx)$ is not necessarily smooth. Hence, \eqref{eq:SIP} can be viewed as a strongly convex  ($\mu_f>0$) optimization problem with finitely many non-smooth constraints $g_i^*(\bx)\leq0$ for $i=1,\dots,m$ whose optimal complexity is indeed $O(\epsilon^{-1})$.

\begin{theorem}
\label{thm:maintheorem_deterministic_muf0}
Suppose $\mu_f=0$, $\mu_y>0$, and parameters $t_k$, $\theta_k$, $\tau_k$, $\sigma_k$ and $\gamma_k$ in Algorithm~\ref{alg:PDHG} are chosen as
\begin{eqnarray*}
    &t_k=1,\quad \theta_k=1,\\
    &\tau_k =\max\{\tau',\tau\sqrt{K}\},\quad \sigma_k =(k+k_0+1) \mu_y,\quad    \gamma_k=60m(M_g^x)^2,
\end{eqnarray*}
where $\tau'\geq\max\left\{8(L_f+1),8(L^{yx}_g+L^{xx}_g)(\|\vlambda^*\|_1+1)\right\}$, $\tau\geq \sqrt{(H_g^{x})^2(\|\vlambda^*\|_1+1)^2+(H_f)^2}$ and $k_0\geq \max\left\{\sqrt{56}L^{yy}_g/\mu_y,14L^{yx}_g/\mu_y\right\}$. Then
\small
\begin{align}
\label{eq:objective_deterministic1_muf0}
f(\bar\vx_{K})-f(\bx^*)\leq \frac{\tau'\|\vx^*-\vx_0\|^2}{2K}+\frac{\tau\|\vx^*-\vx_0\|^2}{2\sqrt{K}}+\frac{4H_f}{\sqrt{K}}
\end{align}
\normalsize
and
\small
\begin{align}
\label{eq:constraint_deterministic1_muf0}
\|[\vg^*(\bar\vx_{K})]_+\|_\infty
\leq&\frac{\tau'\|\vx^*-\vx_0\|^2}{2K}+\frac{\tau\|\vx^*-\vx_0\|^2}{2\sqrt{K}}
+\frac{30m(M_g^x)^2(\|\vlambda^*\|_1+1)^2}{K}+\frac{(k_0+1)\mu_yD_y^2(\|\vlambda^*\|_1+1)}{2K}\\\nonumber
&+\frac{4\sqrt{(H_g^{x})^2(\|\vlambda^*\|_1+1)^2+(H_f)^2}}{\sqrt{K}}+\frac{7(H_g^y)^2(\|\vlambda^*\|_1+1)\ln\left((K+k_0+1)/k_0\right)}{K\mu_y}.
\end{align}
\normalsize
\end{theorem}
\begin{proof}
It is easy to verify that \eqref{eq:thetaktk}, \eqref{eq:determine_cond4_a}, \eqref{eq:determine_cond4_b} and \eqref{eq:determine_cond4_c_simplified} hold. According to the definitions of $t_k$, $\theta_k$, $\tau_k$, $\sigma_k$ and $\gamma_k$, we have 
\begin{eqnarray}
\label{eq:determine_cond2_verified_myf0}
&&\frac{t_kL_f}{2}+\frac{6mt_{k+1}(\theta_{k+1}M_g^x)^2}{\gamma_{k+1}}+\frac{24mt_{k+2}(\theta_{k+2}M_g^x)^2}{\gamma_{k+2}}\\\nonumber
&=&\frac{L_f}{2}+\frac{6m(M_g^x)^2}{60m (M_g^x)^2 }+\frac{24m(M_g^x)^2}{60m (M_g^x)^2 }\leq \frac{t_k\tau_k}{16},
\end{eqnarray}
where the  inequality is because $\tau_k\geq \tau'\geq 8 (L_f+1)$. This means
\eqref{eq:determine_cond2_simplified} holds. In addition, 
\begin{eqnarray}
\label{eq:determine_cond3_verified_muf0}
\frac{7t_{k+1}(\theta_{k+1}L^{yy}_g)^2}{\sigma_{k+1}}
=\frac{7(L^{yy}_g)^2}{(k+k_0+2)\mu_y}
\leq\frac{(k+k_0+1)\mu_y}{8}
=  \frac{t_k\sigma_k}{8},
\end{eqnarray}
where the  inequality is because $k_0\geq \sqrt{56}L^{yy}_g/\mu_y$. This verifies \eqref{eq:determine_cond3}.  Hence, \eqref{eq:choiceinQ1} and \eqref{eq:choiceinQ2} hold by Corollary~\ref{thm:mainprop_deterministic_simplified}. 

By the definitions of $t_k$, $\theta_k$, $\tau_k$, $\sigma_k$ and $\gamma_k$, we have
\begin{align}
\label{eq:sumtk_muf0}
\sum_{k=0}^{K-1}t_k=&K,\\
\label{eq:sumtktauk_muf0}
\sum_{k=0}^{K-1}\frac{t_k}{\tau_k}\leq&\frac{K}{\tau\sqrt{K}}\leq \frac{K^{\frac{1}{2}}}{\sqrt{(H_g^{x})^2(\|\vlambda^*\|_1+1)^2+(H_f)^2}},\\
\label{eq:sumtkthetaksigmak_muf0}
\sum_{k=0}^{K}\frac{t_k\theta_k^2}{\sigma_k}=&\sum_{k=0}^{K}\frac{1}{(k+k_0+1)\mu_y}\leq \frac{\ln\left((K+k_0+1)/k_0\right)}{\mu_y}.
\end{align}
Furthermore, we claim that 
$\mathcal{T}_{\|\vlambda^*\|_1+1,K}\leq 0$ in \eqref{eq:choiceinQ2}. In fact, for any $k\geq0$, we have 
\begin{eqnarray*}
&&\frac{ 7t_{k+1}(\theta_{k+1}L^{yx}_g)^2(\|\vlambda^*\|_1+1)}{\sigma_{k+1}}+\frac{t_kL_g^{xx}(\|\vlambda^*\|_1+1)}{2}\\\nonumber
&=&\frac{ 7(L^{yx}_g)^2(\|\vlambda^*\|_1+1)}{(k+k_0+2)\mu_y}+\frac{L_g^{xx}(\|\vlambda^*\|_1+1)}{2} \\\nonumber
&\leq&\frac{ L^{yx}_g(\|\vlambda^*\|_1+1)}{2}+\frac{L_g^{xx}(\|\vlambda^*\|_1+1)}{2}
\leq \frac{t_k\tau_k}{16},
\end{eqnarray*}
where the first inequality is because $k_0\geq 14L^{yx}_g/\mu_y$ and the second inequality is because $t_k\tau_k\geq\tau'\geq8(L^{yx}_g+L^{xx}_g)(\|\vlambda^*\|_1+1)$. Conclusions \eqref{eq:objective_deterministic1_muf0} and \eqref{eq:constraint_deterministic1_muf0} are then obtained by applying \eqref{eq:sumtk_muf0}, \eqref{eq:sumtktauk_muf0} and \eqref{eq:sumtkthetaksigmak_muf0} as well as the definition of $t_0$, $\tau_0$, $\sigma_0$ and $\gamma_0$ to the right-hand sides of \eqref{eq:choiceinQ1} and \eqref{eq:choiceinQ2}. 
\qedsymbol{}
\end{proof}

When the problem is smooth, we have $H_g^{x}=H_g^{y}=H_f=0$, so we can set $\tau=0$ in Theorem~\ref{thm:maintheorem_deterministic_muf0}. In this case, Theorem~\ref{thm:maintheorem_deterministic_muf0} shows that Algorithm~\ref{alg:PDHG} finds an $\epsilon$-optimal solution for \eqref{eq:SIP} in $\tilde O(\epsilon^{-1})$ iterations when $\mu_f=0$ but $\mu_y>0$. This iteration complexity increases to $O(\epsilon^{-2})$ when the problem is nonsmooth.


\begin{theorem}
\label{thm:maintheorem_deterministic_muf0my0}
Suppose $\mu_f=0$, $\mu_y=0$, and parameters $t_k$, $\theta_k$, $\tau_k$, $\sigma_k$ and $\gamma_k$ in Algorithm~\ref{alg:PDHG} are chosen as
\begin{eqnarray*}
    &t_k=1,\quad \theta_k=1,\\
    &\tau_k =\max\{\tau',\tau\sqrt{K}\},\quad \sigma_k =\max\{\sqrt{56}L^{yy}_g,14L^{yx}_g,(H_g^y/D_y)\sqrt{K}\},\quad    \gamma_k=60(M_g^x)^2,
\end{eqnarray*}
where where $\tau'\geq\max\left\{8(L_f+1),8(L^{yx}_g+L^{xx}_g)(\|\vlambda^*\|_1+1)\right\}$ and $\tau\geq \sqrt{(H_g^{x})^2(\|\vlambda^*\|_1+1)^2+(H_f)^2}$. Then \eqref{eq:objective_deterministic1_muf0} holds and 
\small
\begin{align}
\label{eq:constraint_deterministic1_muf0my0}
\|[\vg^*(\bar\vx_{K})]_+\|_\infty
\leq&\frac{\tau'\|\vx^*-\vx_0\|^2}{2K}+\frac{\tau\|\vx^*-\vx_0\|^2}{2\sqrt{K}}
+\frac{30m(M_g^x)^2(\|\vlambda^*\|_1+1)^2}{K}\\\nonumber
&+\frac{\max\{\sqrt{14}L^{yy}_g,7L^{yx}_g\}D_y^2(\|\vlambda^*\|_1+1)}{K}+\frac{15 H_g^y D_y(\|\vlambda^*\|_1+1)}{2\sqrt{K}}\\\nonumber
&+\frac{4\sqrt{(H_g^{x})^2(\|\vlambda^*\|_1+1)^2+(H_f)^2}}{\sqrt{K}}.
\end{align}
\normalsize
\end{theorem}
\begin{proof}
It is easy to verify that \eqref{eq:thetaktk}, \eqref{eq:determine_cond4_a}, \eqref{eq:determine_cond4_b} and \eqref{eq:determine_cond4_c_simplified} hold. According to the definitions of $t_k$, $\theta_k$, $\tau_k$, $\sigma_k$ and $\gamma_k$, 
\eqref{eq:determine_cond2_verified_myf0} still holds. So does \eqref{eq:determine_cond2_simplified}. In addition,
\begin{eqnarray}
\label{eq:determine_cond3_verified_muf0my0}
\frac{7t_{k+1}(\theta_{k+1}L^{yy}_g)^2}{\sigma_{k+1}}
\leq\frac{7(L^{yy}_g)^2}{\sqrt{56}L^{yy}_g}
\leq\frac{t_k\sigma_k}{8},
\end{eqnarray}
where the second inequality is because $\sigma_k\geq \sqrt{56}L^{yy}_g$. This verifies \eqref{eq:determine_cond3}. Hence, \eqref{eq:choiceinQ1} and \eqref{eq:choiceinQ2} hold by Corollary~\ref{thm:mainprop_deterministic_simplified}. 

By the definitions of $t_k$, $\theta_k$, $\tau_k$, $\sigma_k$ and $\gamma_k$, we have \eqref{eq:sumtk_muf0}, \eqref{eq:sumtktauk_muf0} and 
\begin{align}
\label{eq:sumtkthetaksigmak_muf0my0}
\sum_{k=0}^{K}\frac{t_k\theta_k^2}{\sigma_k}\leq&\frac{K}{(H_g^y/D_y)\sqrt{K}}= \frac{D_y\sqrt{K}}{H_g^y}.
\end{align}
Furthermore, we claim that 
$\mathcal{T}_{\|\vlambda^*\|_1+1,K}\leq 0$ in \eqref{eq:choiceinQ2}. In fact, for any $k\geq0$, we have 
\begin{eqnarray}
\nonumber
&&\frac{ 7t_{k+1}(\theta_{k+1}L^{yx}_g)^2(\|\vlambda^*\|_1+1)}{\sigma_{k+1}}+\frac{t_kL_g^{xx}(\|\vlambda^*\|_1+1)}{2}\\\label{eq:TKnonpositive_muf0my0}
&\leq&\frac{ L^{yx}_g(\|\vlambda^*\|_1+1)}{2}+\frac{L_g^{xx}(\|\vlambda^*\|_1+1)}{2}
\leq \frac{t_k\tau_k}{16},
\end{eqnarray}
where the first inequality is because $\sigma_{k+1}\geq 14L^{yx}_g$ and the second inequality is because $t_k\tau_k\geq\tau'\geq8(L^{yx}_g+L^{xx}_g)(\|\vlambda^*\|_1+1)$. Conclusions \eqref{eq:objective_deterministic1_muf0} and \eqref{eq:constraint_deterministic1_muf0my0} are then obtained by applying \eqref{eq:sumtk_muf0}, \eqref{eq:sumtktauk_muf0} and \eqref{eq:sumtkthetaksigmak_muf0my0} as well as the definition of $t_0$, $\tau_0$, $\sigma_0$ and $\gamma_0$ to the right-hand sides of \eqref{eq:choiceinQ1} and \eqref{eq:choiceinQ2}. 
\qedsymbol{}
\end{proof}

When the problem is smooth, we have $H_g^{x}=H_g^{y}=H_f=0$, so we can set $\tau=\sigma=0$ in Theorem~\ref{thm:maintheorem_deterministic_muf0my0}. In this case, Theorem~\ref{thm:maintheorem_deterministic_muf0my0} shows that Algorithm~\ref{alg:PDHG} finds an $\epsilon$-optimal solution for \eqref{eq:SIP} in $O(\epsilon^{-1})$ iterations when $\mu_f=\mu_y=0$. This iteration complexity increases to $O(\epsilon^{-2})$ when the problem is nonsmooth.

\subsection{Convergence properties in stochastic case}
In this section, we present our main theoretical results on the convergence rate of Algorithm~\ref{alg:SPDHG}. Different from the deterministic case, the order of iteration complexity of our algorithm does not change with strong convexity parameters $\mu_f$ and $\mu_y$. In particular, Algorithm~\ref{alg:SPDHG} has an iteration complexity $O(\epsilon^{-2})$ no matter $\mu_f$ and $\mu_y$ are positive or not. This is a finding consistent with literature, e.g., \cite{boob2022stochastic}, and different from unconstrained stochastic optimization where the  complexity can be improved from $O(\epsilon^{-2})$ to $O(\epsilon^{-1})$ when $\mu_f>0$. For this reason, in this section, we just assume $\mu_f=\mu_y=0$.

Similar to Corollary \ref{thm:mainprop_deterministic_simplified}, we first derive from Proposition~\ref{thm:mainprop_deterministic} generic upper bounds on the objective gap and constraint violation of $\bar\vx_{K}$.
\begin{corollary}
\label{thm:mainprop_stochastic_simplified}
Suppose $\mu_f=\mu_y=0$ and  parameters $t_k$, $\theta_k$, $\tau_k$, $\sigma_k$ and $\gamma_k$ in Algorithm~\ref{alg:SPDHG} satisfies \eqref{eq:thetaktk},  \eqref{eq:determine_cond2}, \eqref{eq:determine_cond3}, \eqref{eq:determine_cond4_c} and the following two inequalities
\begin{align}
\label{eq:determine_cond4_a_mufmy0}
&  t_{k+1}\tau_{k+1}-t_{k}\tau_{k}\leq0,\\
\label{eq:determine_cond4_b_mufmy0}
& t_{k+1}\sigma_{k+1}- t_k\sigma_k \leq 0.
\end{align}
It holds that
\small
\begin{align}
\label{eq:choiceinQ1_stochastic}
\mathbb{E}\left[f(\bar\vx_{K})-f(\bx^*)\right]
\leq&\frac{1}{\sum_{k=0}^{K-1}t_k}\frac{t_0\tau_0}{2}\|\vx^*-\vx_0\|^2
+\frac{1}{\sum_{k=0}^{K-1}t_k}\sum_{k=0}^{K-1}\frac{4t_k(H_f)^2}{\tau_k}\\\nonumber
&+\frac{2\sigma_{f'}^2}{\sum_{k=0}^{K-1}t_k}\sum_{k=0}^{K-1}\frac{t_k}{\tau_k}
+\frac{30m\sigma_g^2}{\sum_{k=0}^{K-1}t_k}\sum_{k=0}^{K-1}\frac{t_k}{\gamma_k}
\end{align}
\normalsize
and
\small
\begin{align}
\label{eq:choiceinQ2_stochastic}
&\mathbb{E}\|[\vg^*(\bar\vx_{K})]_+\|_\infty\\\nonumber
\leq&\frac{1}{\sum_{k=0}^{K-1}t_k}\left[t_0\tau_0\|\bx-\bx_0\|^2+t_0\gamma_0(\|\vlambda^*\|_1+1)^2+t_0\sigma_0D_y^2(\|\vlambda^*\|_1+1)+ \mathbb{E}\mathcal{T}_{\|\vlambda^*\|_1+1,K}\right]\\\nonumber
&+\frac{4\left[(H_g^{x})^2(\|\vlambda^*\|_1+1)^2+(H_f)^2\right]}{\sum_{k=0}^{K-1}t_k}\sum_{k=0}^{K-1}\frac{t_k}{\tau_k}+\frac{7(H_g^y)^2(\|\vlambda^*\|_1+1)}{\sum_{k=0}^{K-1}t_k}\sum_{k=0}^{K}\frac{t_k\theta_k^2}{\sigma_k}\\\nonumber
&+\frac{(\|\vlambda^*\|_1+1)}{\sum_{k=0}^{K-1}t_k}
\sqrt{10m\sigma_{g'}^2D_y^2\sum_{k=0}^{K-1}t_k^2}
+\frac{2\sigma_{f'}^2+4\sigma_{g'}^2(\|\vlambda^*\|_1+1)^2}{\sum_{k=0}^{K-1}t_k}\sum_{k=0}^{K-1}\frac{t_k}{\tau_k}\\\nonumber
&+\frac{15m(\|\vlambda^*\|_1+1)\sigma_{g'}^2}{\sum_{k=0}^{K-1}t_k}\sum_{k=0}^{K-1}\frac{t_k}{\sigma_k}
+\frac{30m\sigma_g^2}{\sum_{k=0}^{K-1}t_k}\sum_{k=0}^{K-1}\frac{t_k}{\gamma_k},
\end{align}
\normalsize
where $\mathcal{T}_{\|\vlambda^*\|_1+1,K}$ is defined in \eqref{eq:TK} with $\Lambda=\|\vlambda^*\|_1+1$.
\end{corollary}
\begin{proof}
Conditions \eqref{eq:determine_cond4_a_mufmy0} and \eqref{eq:determine_cond4_b_mufmy0} are just \eqref{eq:determine_cond4_a} and \eqref{eq:determine_cond4_b} when $\mu_f=\mu_y=0$. Therefore, by Proposition~\ref{thm:mainprop_deterministic}, \eqref{eq:Qsum5} holds for any deterministic $\bx\in\cX$, stochastic $\by\in\cY$ and stochastic $\vlambda\in\mathbb{R}_+^m$ satisfying $\|\vlambda\|_1\leq \Lambda$ almost surely. Recall that $\vlambda_0=\mathbf{0}$ and observe that $\mathcal{T}_{0,K}\leq 0$. Applying \eqref{eq:Qsum5} to the right-hand side of \eqref{eq:Qsum_averaged_obj} with $\Lambda=0$ leads to \eqref{eq:choiceinQ1}. Recall \eqref{eq:lambda_for_barxK}, \eqref{eq:y_for_barxK} and \eqref{eq:lambda_for_barxK_bound}. Inequality \eqref{eq:choiceinQ2} is thus obtained by applying \eqref{eq:Qsum5} to the right-hand side of \eqref{eq:Qsum_averaged_cst} with $\Lambda= \|\vlambda^*\|_1+1$ and using the facts that $\|\widehat\vlambda_K-\vlambda_0\|=\|\widehat\vlambda_K\|\leq\|\vlambda^*\|_1+1$ and that $\|\widehat\by_K-\vy_0\|_{\widehat\vlambda_K}^2\leq \|\widehat\vlambda_K\|_1D_y^2$. 
\qedsymbol{}
\end{proof}

We then present the convergence property of Algorithm~\ref{alg:SPDHG} in the following theorem.

\begin{theorem}
\label{thm:maintheorem_deterministic_muf0my0_stoc}
Suppose $\mu_f=0$, $\mu_y=0$, and parameters $t_k$, $\theta_k$, $\tau_k$, $\sigma_k$ and $\gamma_k$ in Algorithm~\ref{alg:SPDHG} are chosen as
\small
\begin{eqnarray*}
    &t_k=1,\quad \theta_k=1,\\
 &\tau_k =\max\{\tau',\tau\sqrt{K}\},\quad \sigma_k =\max\left\{\sqrt{56}L^{yy}_g,14L^{yx}_g,\left(\sqrt{(H_g^y)^2+\sigma_{g'}^2}/D_y\right)\sqrt{K}\right\},\\  &\gamma_k=60(M_g^x)^2+\frac{60m\sigma_{g'}(2K-k)}{\sqrt{K}},
\end{eqnarray*}
\normalsize
where $\tau'\geq\max\left\{4(L_f+1+\sigma_{g'}),8(L^{yx}_g+L^{xx}_g)(\|\vlambda^*\|_1+1)\right\}$ and 
$$\tau\geq \sqrt{(H_g^{x})^2(\|\vlambda^*\|_1+1)^2+(H_f)^2+\sigma_{g'}^2(\|\vlambda^*\|_1+1)+\sigma_{f'}^2}.$$ 
Then 
\small
\begin{align}
\label{eq:objective_deterministic1_muf0my0_stoc}
\mathbb{E}\left[f(\bar\vx_{K})-f(\bx^*)\right]\leq& 
\frac{\tau'\|\vx^*-\vx_0\|^2}{2K}+\frac{\tau\|\vx^*-\vx_0\|^2}{2\sqrt{K}}+\frac{4H_f}{\sqrt{K}}+\frac{2\sigma_{f'}}{\sqrt{K}}+\frac{\sigma_{g'}}{2\sqrt{K}}
\end{align}
\normalsize
and
\small
\begin{align}
\label{eq:constraint_deterministic1_muf0my0_stoc}
&\mathbb{E}\|[\vg^*(\bar\vx_{K})]_+\|_\infty\\\nonumber
\leq&\frac{\tau'\|\vx^*-\vx_0\|^2}{2K}+\frac{\tau\|\vx^*-\vx_0\|^2}{2\sqrt{K}}
+\frac{60(M_g^x)^2(\|\vlambda^*\|_1+1)^2}{K}
+\frac{60m\sigma_{g'}(\|\vlambda^*\|_1+1)^2}{\sqrt{K}}\\\nonumber
&+\frac{\max\{\sqrt{56}L^{yy}_g,14L^{yx}_g\}D_y^2(\|\vlambda^*\|_1+1)}{K}+\frac{8\sqrt{(H_g^y)^2+\sigma_{g'}^2}D_y(\|\vlambda^*\|_1+1)}{\sqrt{K}}\\\nonumber
&+\frac{4\sqrt{(H_g^{x})^2(\|\vlambda^*\|_1+1)^2+(H_f)^2}}{\sqrt{K}}\\\nonumber
&+\frac{\sqrt{10m\sigma_{g'}D_y}(\|\vlambda^*\|_1+1)}{\sqrt{K}}
+\frac{4\sqrt{\sigma_{f'}^2+\sigma_{g'}^2(\|\vlambda^*\|_1+1)^2}}{\sqrt{K}}+\frac{15m(\|\vlambda^*\|_1+1)\sigma_{g'}D_y}{\sqrt{K}}
+\frac{\sigma_g}{2\sqrt{K}}.
\end{align}
\normalsize
\end{theorem}

\begin{proof}
It is easy to verify that \eqref{eq:thetaktk}, \eqref{eq:determine_cond4_a_mufmy0} and \eqref{eq:determine_cond4_b_mufmy0} hold. In addition, 
\small
\begin{eqnarray*}
&&t_{k+1}\gamma_{k+1}+\frac{8m\sigma_{g'}^2t_k}{\tau_k}\\
&\leq& 60(M_g^x)^2+\frac{60m\sigma_{g'}(2K-k-1)}{ \sqrt{K}}+\frac{60m\sigma_{g'}^2}{ \tau\sqrt{K}}\\
&\leq& 60(M_g^x)^2+\frac{60m\sigma_{g'}(2K-k)}{ \sqrt{K}}=t_{k}\gamma_{k},
\end{eqnarray*}
\normalsize
where the first inequality is by the definitions of $t_k$, $\theta_k$, $\gamma_k$ and $\tau_k$ and the second inequality is because $\tau\geq \sigma_{g'}$. This verifies \eqref{eq:determine_cond4_c}. By definitions of $t_k$, $\theta_k$, $\gamma_k$ and $\tau_k$ again, we have
\small
\begin{eqnarray*}
&&\frac{t_kL_f}{2}+\frac{6mt_{k+1}\theta_{k+1}^2(M_g^x)^2}{\gamma_{k+1}}+\frac{24mt_{k+2}\theta_{k+2}^2(M_g^x)^2}{\gamma_{k+2}}+\frac{18m\sigma_{g'}^2t_{k+1}}{\gamma_{k+1}}+\frac{12m\sigma_{g'}^2t_{k+2}}{\gamma_{k+2}}\\\nonumber
&\leq&\frac{L_f}{2}+ \frac{30(M_g^x)^2}{60(M_g^x)^2}+\frac{30 m\sigma_{g'}^2}{60m\sigma_{g'}}\\\nonumber
&\leq&\frac{L_f}{2}+ \frac{1}{2}+\frac{\sigma_{g'}}{2}\leq \frac{t_{k}\tau_{k}}{16},
\end{eqnarray*}
\normalsize
where the first inequality is obtained using the fact that $\frac{2K-k}{\sqrt{K}}\geq1$ and the second inequality is because $\tau_k\geq \tau'\geq 8 (L_f+1)$. This verifies \eqref{eq:determine_cond2}. Lastly, \eqref{eq:determine_cond3_verified_muf0my0} still holds. So does \eqref{eq:determine_cond3}. Hence, \eqref{eq:choiceinQ1_stochastic} and \eqref{eq:choiceinQ2_stochastic} hold by Corollary~\ref{thm:mainprop_stochastic_simplified}. 

By the definitions of $t_k$, $\theta_k$, $\tau_k$, $\sigma_k$ and $\gamma_k$ again, we have \eqref{eq:sumtk_muf0},  \eqref{eq:sumtkthetaksigmak_muf0my0} and the following three inequalities
\begin{align}
\label{eq:sumtksquare_muf0my0}
\sum_{k=0}^{K-1}t_k^2=&K,\\
\label{eq:sumtksigmak_stochastic_muf0}
\sum_{k=0}^{K-1}\frac{t_k}{\sigma_k}\leq&\frac{K}{\left(\sqrt{(H_g^y)^2+\sigma_{g'}^2}/D_y\right)\sqrt{K}}= \frac{D_y\sqrt{K}}{\sqrt{(H_g^y)^2+\sigma_{g'}^2}},\\
\label{eq:sumtktauk_stochastic_muf0}
\sum_{k=0}^{K-1}\frac{t_k}{\tau_k}\leq&\frac{K}{\tau\sqrt{K}}\leq \frac{K^{\frac{1}{2}}}{\sqrt{(H_g^{x})^2(\|\vlambda^*\|_1+1)^2+(H_f)^2+\sigma_{g'}^2(\|\vlambda^*\|_1+1)+\sigma_{f'}^2}},\\
\label{eq:sumtkgammak_stochastic_muf0}
\sum_{k=0}^{K-1}\frac{t_k}{\gamma_k}\leq&\frac{K\sqrt{K}}{60m\sigma_{g'}K}=\frac{\sqrt{K}}{60m\sigma_{g'}}.
\end{align}
In addition, \eqref{eq:TKnonpositive_muf0my0} still holds, so $\mathcal{T}_{\|\vlambda^*\|_1+1,K}\leq 0$ almost surely. 
Conclusions \eqref{eq:objective_deterministic1_muf0my0_stoc} and \eqref{eq:constraint_deterministic1_muf0my0_stoc} are then obtained by applying \eqref{eq:sumtk_muf0}, \eqref{eq:sumtkthetaksigmak_muf0my0}, \eqref{eq:sumtksquare_muf0my0}, \eqref{eq:sumtksigmak_stochastic_muf0}, \eqref{eq:sumtktauk_stochastic_muf0} and \eqref{eq:sumtkgammak_stochastic_muf0}  as well as the definition of $t_0$, $\tau_0$, $\sigma_0$ and $\gamma_0$ to the right-hand sides of \eqref{eq:choiceinQ1_stochastic} and \eqref{eq:choiceinQ2_stochastic}. 
\qedsymbol{}
\end{proof}

Theorem~\ref{thm:maintheorem_deterministic_muf0my0_stoc} shows that, using stochastic oracles, Algorithm~\ref{alg:SPDHG} finds an $\epsilon$-optimal solution for \eqref{eq:SIP} in $O(\epsilon^{-2})$ iterations. Here, we only focus on the case where $\mu_f=0$ and $\mu_y=0$ because the convergence rate does not change when $\mu_f$ or $\mu_y$ or both are positive. This finding is consistent with the results in stochastic optimization with finitely many constraints~\cite{boob2022stochastic}. 

Moreover, it is easy to see that Theorem~\ref{thm:maintheorem_deterministic_muf0my0} is a special case of Theorem~\ref{thm:maintheorem_deterministic_muf0my0_stoc} when the problem is deterministic, i.e., when $\sigma_f=\sigma_{g'}=\sigma_{g}=0$. Therefore, when we further have $H_g^{x}=H_g^{y}=H_f=0$,  we can set $\tau=\sigma=0$ and, according to \eqref{eq:objective_deterministic1_muf0my0_stoc} and \eqref{eq:constraint_deterministic1_muf0my0_stoc}, the iteration complexity of Algorithm~\ref{alg:SPDHG} is reduced $O(\epsilon^{-1})$.

\section{Experimental Results}
\label{sec:experiments}
In this section, we evaluate the numerical performance of the proposed AGSIP method on an instance of \eqref{eq:SIP}. All experiments are conducted in Matlab on a computer with the CPU 2GHz Quad-Core Intel Core i5.

\subsection{$\boldsymbol{\mu_f>0,\ \mu_y>0}$}
\label{subsec:strongly_convex}
We consider the instance of \eqref{eq:SIP} adapted from \cite{Calafiore2005}. Let $m=4$, $p=10$ and $q=10$ in \eqref{eq:SIP}. This instance is formulated as follows
\begin{equation*}
    \begin{split}
        \min_{\|\bx\|_\infty\leq 2}&-\mathbf{1}^\top \vx + \frac{0.1}{2}\|\bx\|_2^2 \\
        \text{s.t.}\quad &(\va_i+0.2\vy^i)^\top \vx - b_i - \frac{1}{2}{\by^i}^\top \bQ \by^i \leq 0,~\forall \vy^i~\text{s.t.}~\|\vy^i\|_\infty\leq 1\text{ for }i=1,\dots,4,
    \end{split}
\end{equation*}
where
\begin{equation*}
    \begin{split}
        \begin{bmatrix}
            \va_1^\top \\
            \va_2^\top \\
            \va_3^\top \\
            \va_4^\top
        \end{bmatrix}
        =
        \begin{bmatrix}
            -1 & 0 & -1 & 0 & 0 & -1 & -1 & 0 & -1 & 0\\
            0 & -1 & 0 & -1 & -1 & 0 & 0 & -1 & 0 & -1\\
            1 & 0 & 1 & 0 & 0 & 1 & 1 & 0 & 1 & 0\\
            0 & 1 & 0 & 1 & 1 & 0 & 0 & 1 & 0 & 1
        \end{bmatrix}
        \ ,\  
        \begin{bmatrix}
            b_1 \\
            b_2 \\
            b_3 \\
            b_4
        \end{bmatrix}
        =
        \begin{bmatrix}
            0 \\
            0 \\
            1\\
            1
        \end{bmatrix},
    \end{split}
\end{equation*}
and $\bQ= 0.1(\bq^\top \bq +\bI)$ with $\bq$ a 10-by-10 matrix of uniformly distributed random numbers between $0$ and $1$ and $\bI$ a 10-by-10 identity matrix. In this instance, $f(\bx)$ is strongly-convex in $\bx$ and $g_i(\bx,\by^i)$ is strongly-concave in $\by^i$ for fixed $\bx$.

We compare the AGSIP method with three baselines: the SIP-CoM algorithm \cite{Wei2020TheCA}, the smoothing penalty method \cite{xu2014solving} and the exchange method \cite{lopez2007semi}. At iteration $k$ of the SIP-CoM algorithm, the solution $\by_k$ to the lower problem can be solved using Matlab solver \textit{quadprog}.
Then a projected gradient step with step size $\eta_k$ is performed to update $\bx_k$ using either gradient $\nabla f(\vx_k)$ or subgradient $\partial_{x} [\max_{i}g_i(\vx_k,\vy_k^i)]$ depending on whether $g(\vx_k,\vy_k)\leq \delta_k$ or not. Here, $\delta_k\geq0$ is a tolerance parameter. The smoothing penalty method first approximates the constraint $g_i^*(\bx)\leq 0$ by $\gamma_i(\bx)\leq 0$, where
$$
\gamma_i(\bx)=\rho^{-1}\ln\left(\int_{\cY^i}\exp(\rho g_i(\bx,\by^i))d\by^i\right), \quad i=1,\dots,m,
$$
and $\rho>0$ is a smoothing parameter. Then it solves the resulting constrained optimization by a penalty method, namely, solving
\begin{equation*}
    \begin{split}
        \min_{\bx\in\cX} \quad &G_{\rho}^{\lambda}(\vx):= f(\vx)+\frac{\lambda}{2}\sum_{i=1}^m\left( \sqrt{\gamma_{i}^2(\vx)+\rho^{-1}} + \gamma_{i}(\vx) \right),
    \end{split}
\end{equation*}
where $\lambda>0$ is a penalty parameter. The accelerated gradient method~\cite{nesterov2018lectures} with line search is applied to the subproblem above. Since computing  $\gamma_i$ and its gradient requires calculating an integral, a set of points of size $5,000$ is randomly sampled from a uniform distribution over $\cY^i$ to approximate the integral. The set is sampled only once at the beginning of the algorithm and used in all iterations. Once condition (3.17) or (3.18) in \cite{xu2014solving} is satisfied, the accelerated gradient method is restarted with $\rho$ or $\gamma$ in the penalty term increased by a factor. We refer readers to Algorithm 3.1 in \cite{xu2014solving} for more details. In the exchange method, a set of points of size $100$ is first randomly sampled from a uniform distribution over $\cY^i$ for $i=1,\dots,4$ and the set of points for each $i$ is denoted as $\cY^i_0$. Then at iteration $k$, a solution $\bx_k$ is solved for the following problem
\begin{equation*}
        \min_{\bx\in\cX}\ f(\bx) \quad \text{s.t.}\quad g_i(\bx,\by^i) \leq 0,~\forall \vy^i\in \cY^i_k\text{ for }i=1,\dots,4.
\end{equation*}
Then for the fixed $\bx_k$, the solution $\by^i_k$ of the lower problem can be computed with Matlab solver. Then, if $\max_i g_i(\bx_k,\by^i_k)\leq 0$, $\bx_k$ is the optimal solution; otherwise, $\cY^i_{k+1}=\cY^i_k \cup \by^i_k$.

In our AGSIP algorithm, based on Theorem~\ref{thm:maintheorem_deterministic}, we set $t_k=k+k_0+1$, $\theta_k=\frac{k+k_0}{k+k_0+1}$, and set the step sizes $\sigma_k =\tau_k =\frac{k+k_0+1}{k_0}$, $\gamma_k=\frac{k_0}{k+k_0+1}$ where $k_0$ is chosen from $\{10^2,10^3,10^4,10^5\}$ that gives the best performance. In the SIP-CoM method, the step size is set to $\eta_k = \frac{C}{k+1}$ where $C$ is tuned from $\{0.01,0.1,1\}$, and the tolerance $\delta_k$ is set to $2\times 10^{-3}$. The output of the SIP-CoM method is the weighted average of $\bx_k$ at feasible steps with weights $\frac{\eta_k}{1-\eta_k}$.
In the smoothing penalty method, the initial values of $\rho$ and $\lambda$ are set to $10$ and $5$, respectively, and they will be increased by factors $1.1$ and $5$, respectively, after each restart.

For the four methods in comparison, we report the objective value $f(\bx_k)$ and the amount of constraint violation, i.e.,  $\max_{i=1,\dots,m}g_i^*(\bx_k)$, during the algorithms in Figures~\ref{convergence_time_strconvex} and~\ref{convergence_iter_strconvex}. The horizontal line represents CPU times (in seconds)  in Figure~\ref{convergence_time_strconvex}  and  the number of iterations in Figure~\ref{convergence_iter_strconvex}. For the smoothing penalty method, the number of iterations is the total number of iterations in the accelerated gradient algorithm performed across all restarts, including the ones performed during the line search. According to Figure~\ref{convergence_time_strconvex}, the smoothing penalty method require much longer CPU time than other three methods. This is because the smoothing penalty method needs a large sample of $\by^i$ to approximate the integral in $\gamma_i$ and the cost of computing the gradient of $G_{\rho}^{\lambda}(\vx)$ increases with the sample size. On the contrary, our method only needs to compute the gradients of $f$ and $g_i$ at one point per iteration. Both of the SIP-CoM algorithm and the exchange method require a solver to solve the lower-level problem, which is not always available. If the lower-level problem cannot be solved using a solver, the SIP-CoM algorithm requires a large sample of $\by^i$ per iteration to approximately solve the lower-level problem and the exchange method cannot be applied. 



\begin{figure}
\begin{center}
\includegraphics[width=0.49\linewidth]{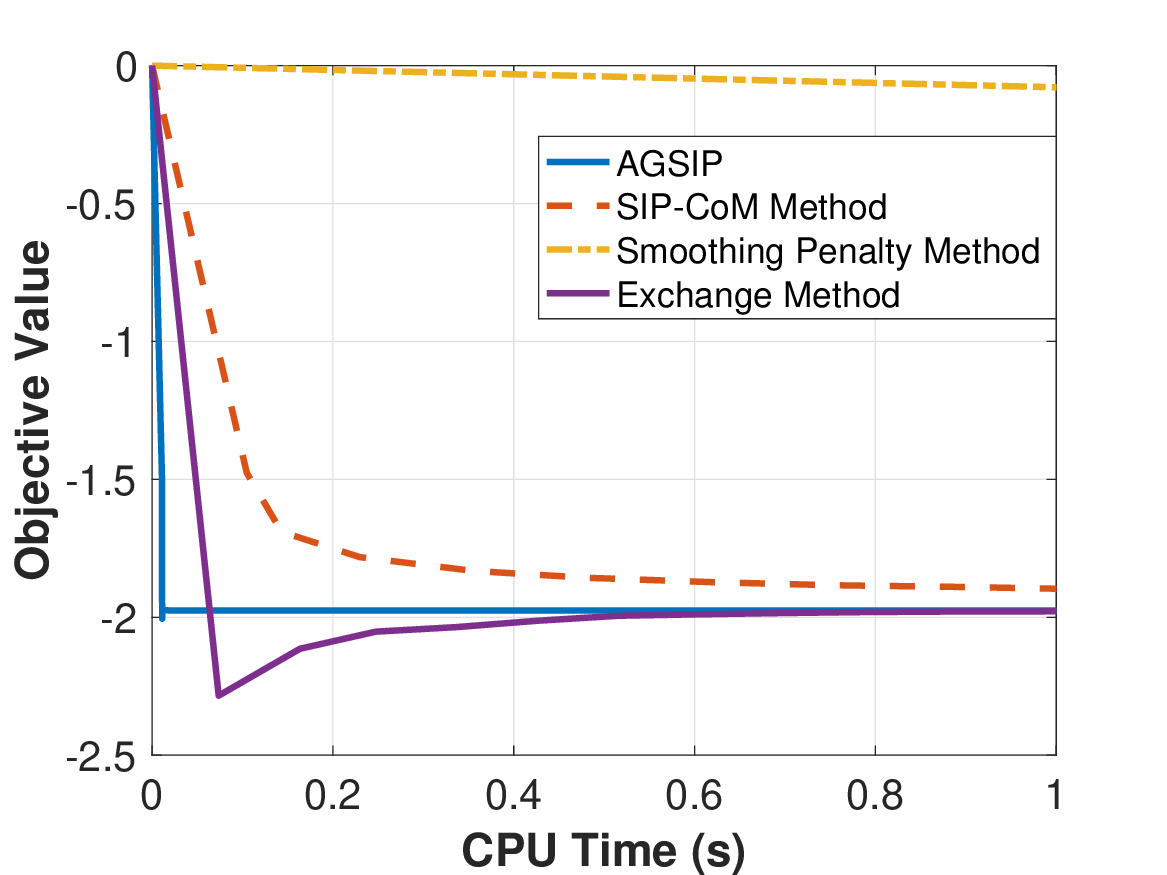}
\includegraphics[width=0.49\linewidth]{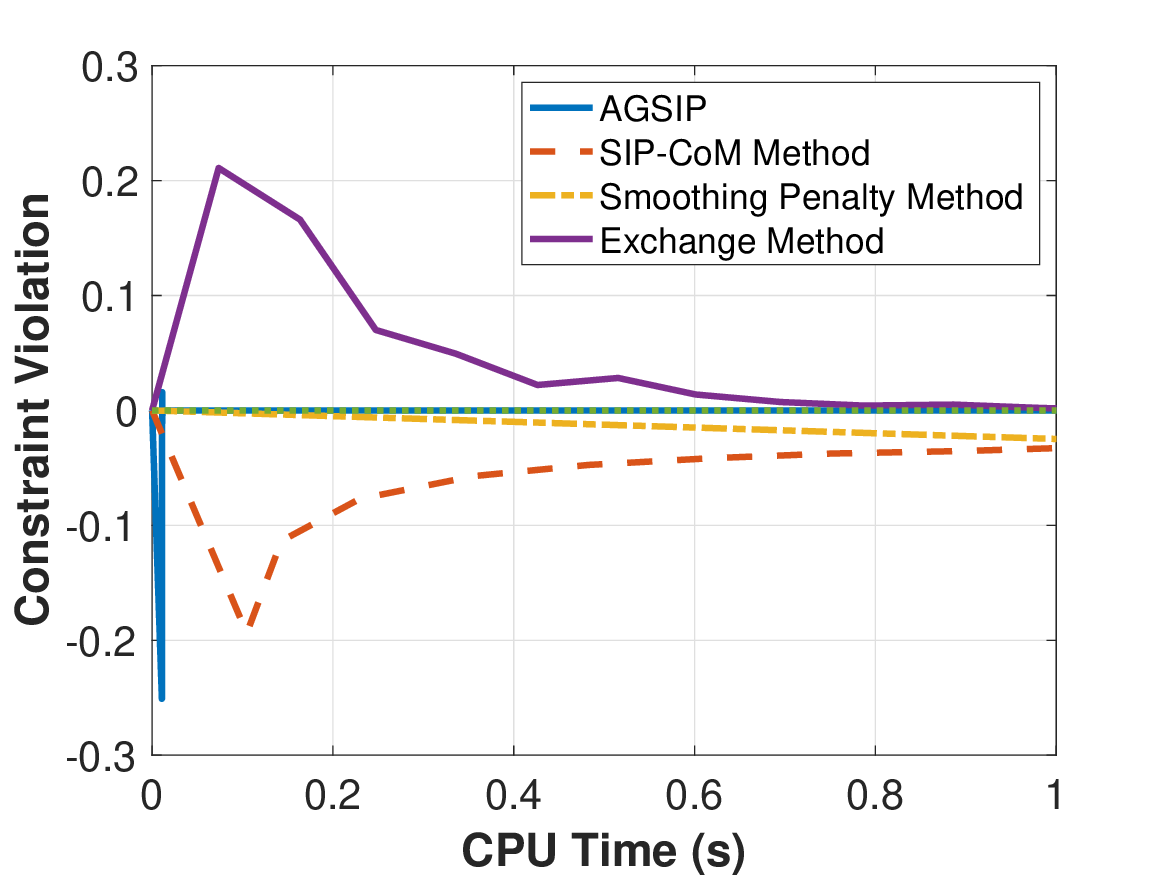}
\caption{Convergence of objective values and constraint violations with CPU time.}
\label{convergence_time_strconvex}
\end{center}
\end{figure}

\begin{figure}
\begin{center}
\includegraphics[width=0.49\linewidth]{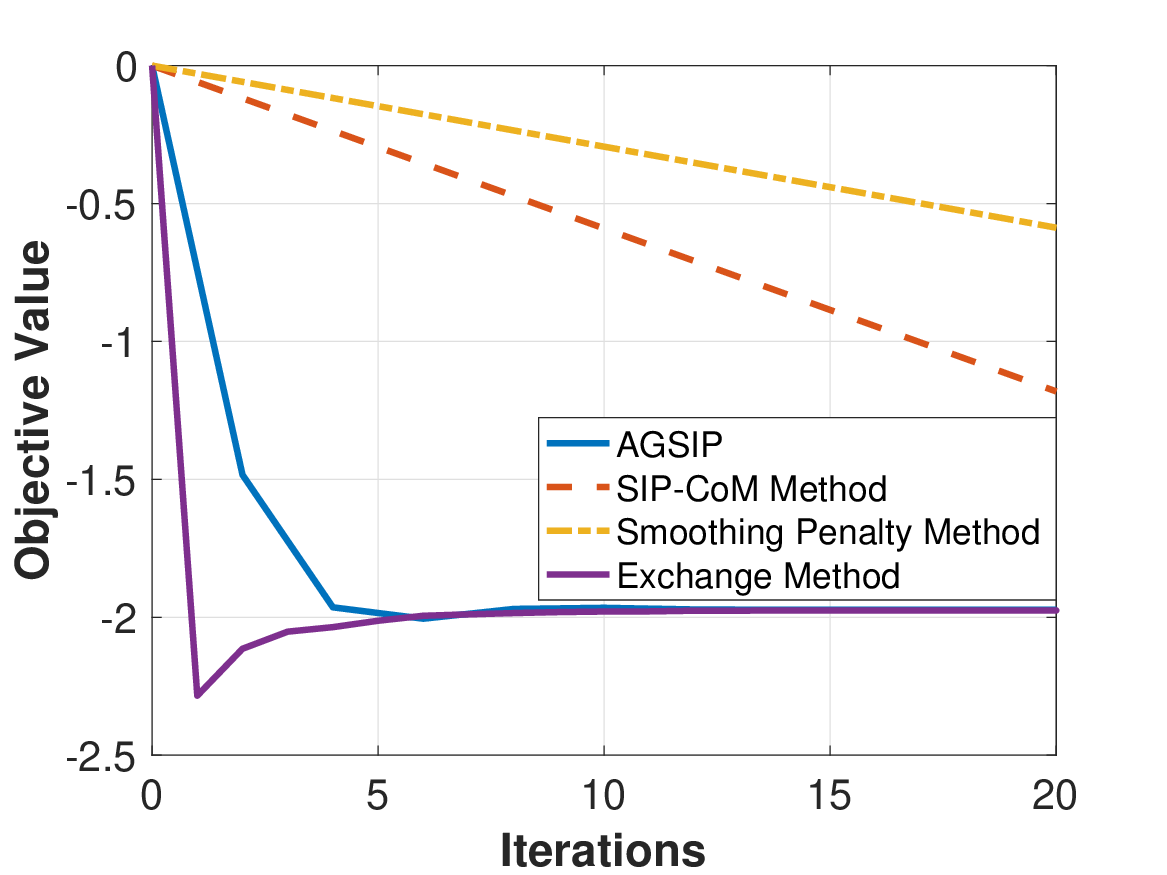}
\includegraphics[width=0.49\linewidth]{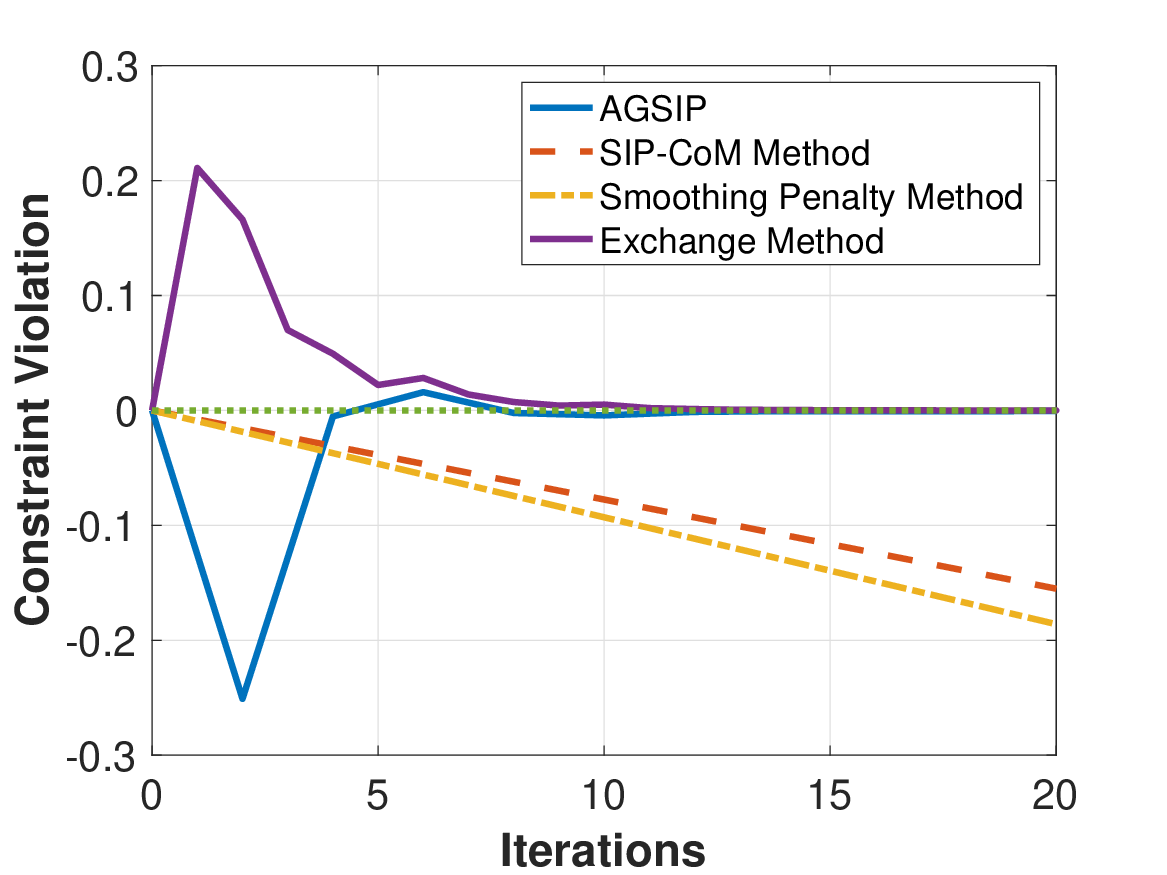}
\caption{Convergence objective values and constraint violations with number of iterations.}
\label{convergence_iter_strconvex}
\end{center}
\end{figure}

\subsection{$\boldsymbol{\mu_f=0,\ \mu_y=0}$}
\label{subsec:convex}

We consider the instance of \eqref{eq:SIP} adapted from \cite{Calafiore2005}. Let $m=4$, $p=10$ and  $q=10$ in \eqref{eq:SIP}. This instance is formulated as follows
\begin{equation*}
    \begin{split}
        \min_{\|\bx\|_\infty\leq 2}&-\mathbf{1}^\top \vx\quad
        \text{s.t.}\quad (\va_i+0.2\vy^i)^\top \vx - b_i\leq 0,~\forall \vy^i~\text{s.t.}~\|\vy^i\|_2\leq 1\text{ for }i=1,\dots,4,
    \end{split}
\end{equation*}
where $\va_i$ and $b_i$ are the same as in~\ref{subsec:strongly_convex}.

We compare the AGSIP method with three baselines the same as in~\ref{subsec:strongly_convex}. In the SIP-CoM algorithm and the exchange method, the lower-level problem have closed-form solutions. In our AGSIP algorithm, according to~\ref{thm:maintheorem_deterministic_muf0my0}, we fix the $\theta_k=t_k=1$, and set the step sizes $\sigma_k = \sigma$, $\gamma_k=\gamma$ and $\tau_k =\tau$ where $\sigma,\ \gamma,\ \tau$ are chosen from $\{0.01, 0.1, 1\}$ that gives the best performance. In the SIP-CoM method, the step size is set to $\eta_k=\frac{C}{\sqrt{k+1}}$ where $C$ is tuned from $\{0.01, 0.1, 1\}$, and the tolerance $\delta_k$ is set to $\frac{\delta}{\sqrt{k+1}}$ with $\delta$ chosen from $\{0.01, 0.1, 0.5\}$. In the smoothing penalty method, a set of points of size 50,000 is randomly sampled from a uniform distribution over $\mathcal{Y}^i$ to approximate the integral. The initial values of $\rho$ and $\lambda$ are set to $10$ and $5$, respectively, and they will be increased by factors $1.1$ and $5$, respectively, after each restart.

For the four methods in comparison, we report the objective value $f(\bx_k)$ and the amount of constraint violation, i.e.,  $\max_{i=1,\dots,m}g_i^*(\bx_k)$, during the algorithms in Figures~\ref{convergence_time} and~\ref{convergence_iter}. The horizontal line represents CPU times (in seconds) in Figure~\ref{convergence_time} and the number of iterations in Figure~\ref{convergence_iter}. For the smoothing penalty method, the number of iterations is the total number of iterations in the accelerated gradient algorithm performed across all restarts, including the ones performed during the line search. 
According to Figure~\ref{convergence_time}, our method AGSIP converges faster than the three baselines, but the advantage of our method is not significant when compared with the SIP-CoM method and the exchange method. The reason for that is the simple lower problem, which results in closed-form solutions for the SIP-CoM method and the exchange method. However, difficult lower-level problem without closed-form solution and cannot be solved with a solver are common in practice. In such cases, the SIP-CoM algorithm requires a large sample of $\by^i$ per iteration to approximately solve the lower-level problem and the exchange method cannot be applied. 



\begin{figure}
\begin{center}
\includegraphics[width=0.49\linewidth]{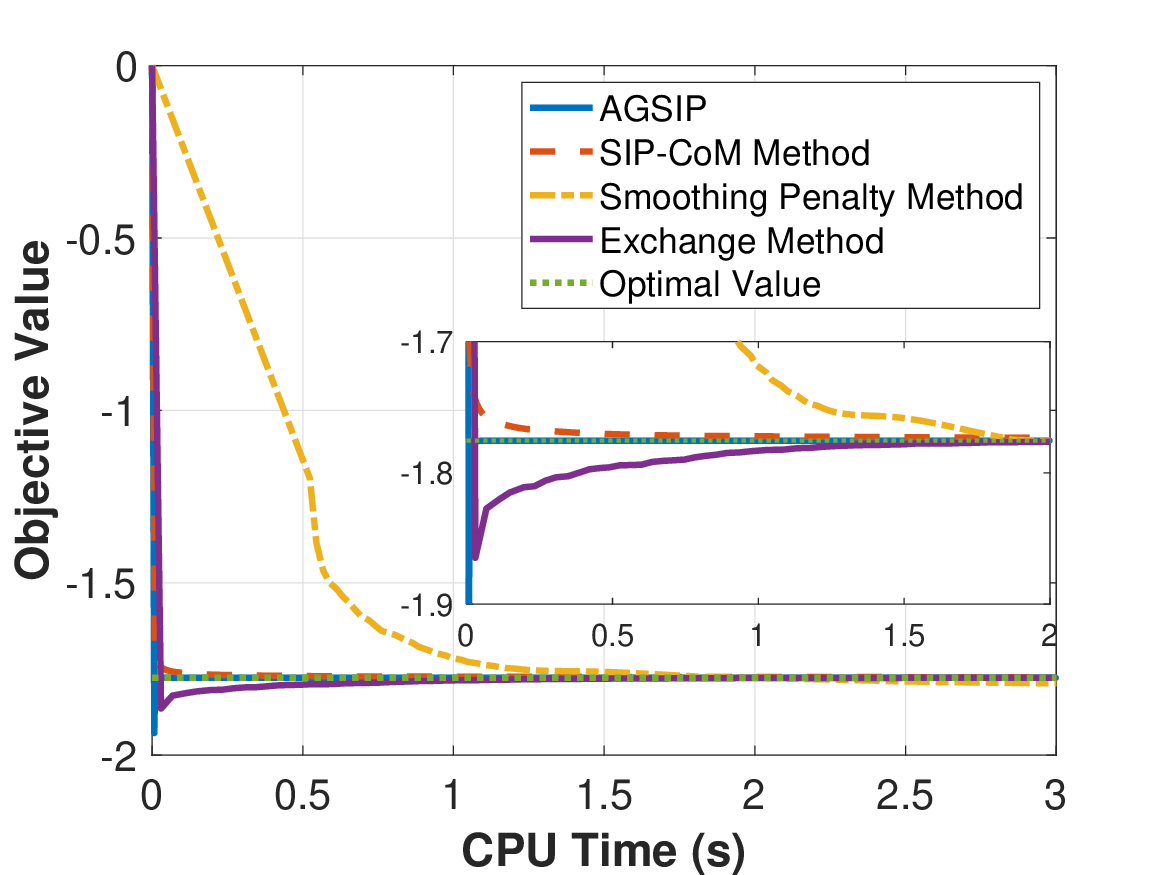}
\includegraphics[width=0.49\linewidth]{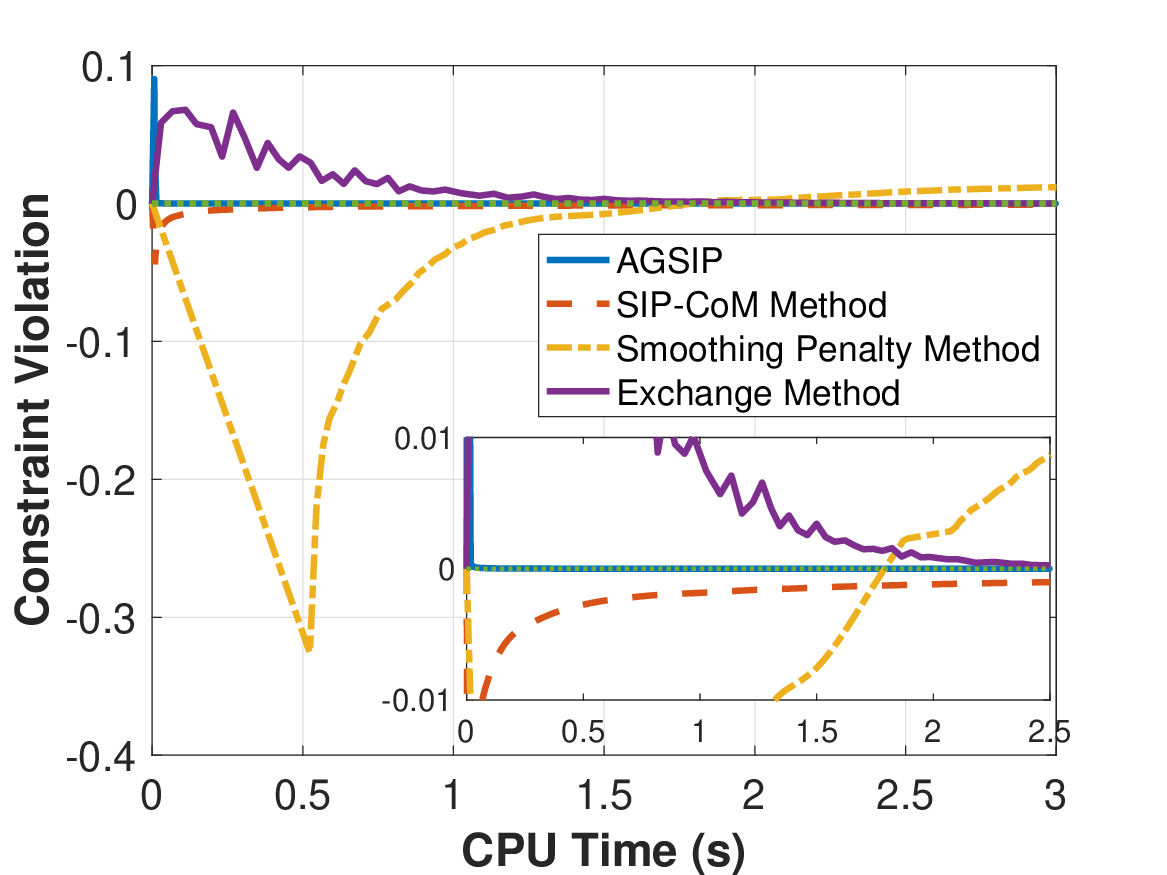}
\caption{Convergence of objective values and constraint violations with CPU time.}
\label{convergence_time}
\end{center}
\end{figure}

\begin{figure}
\begin{center}
\includegraphics[width=0.49\linewidth]{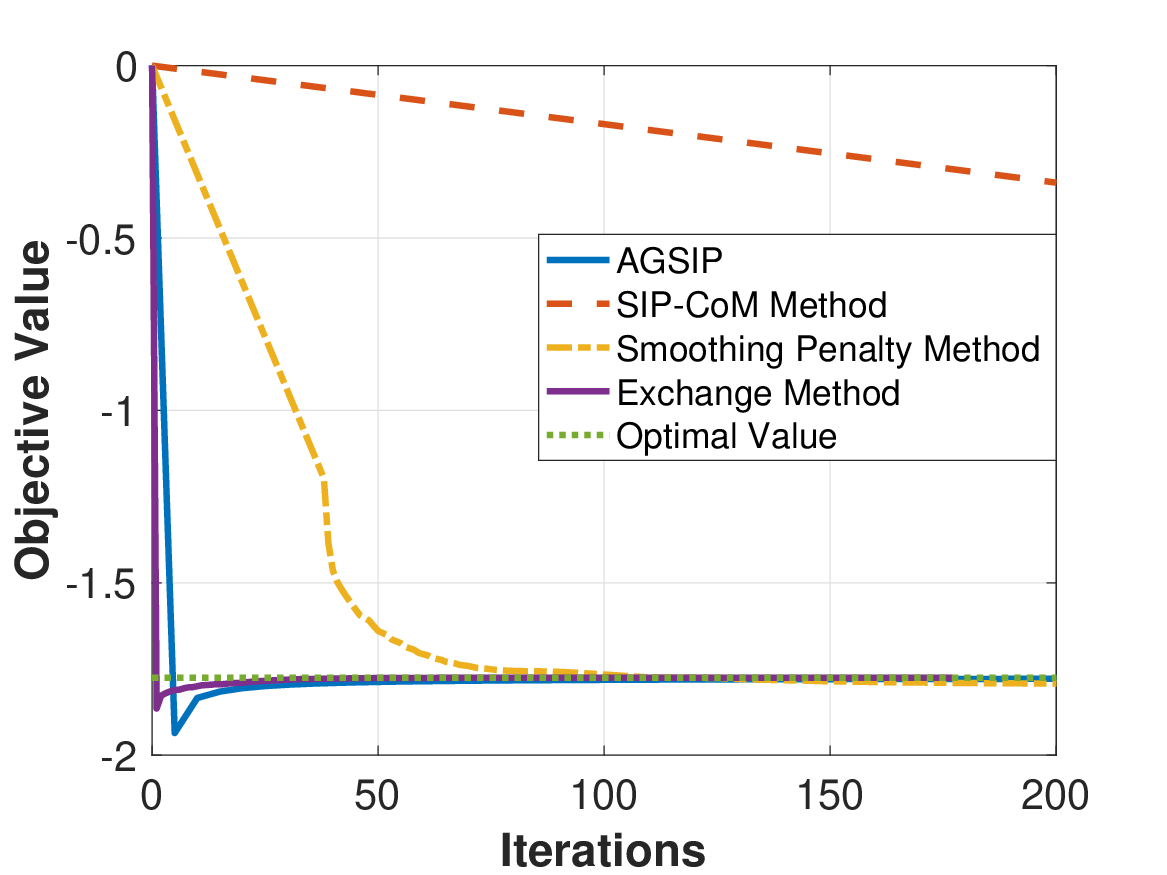}
\includegraphics[width=0.49\linewidth]{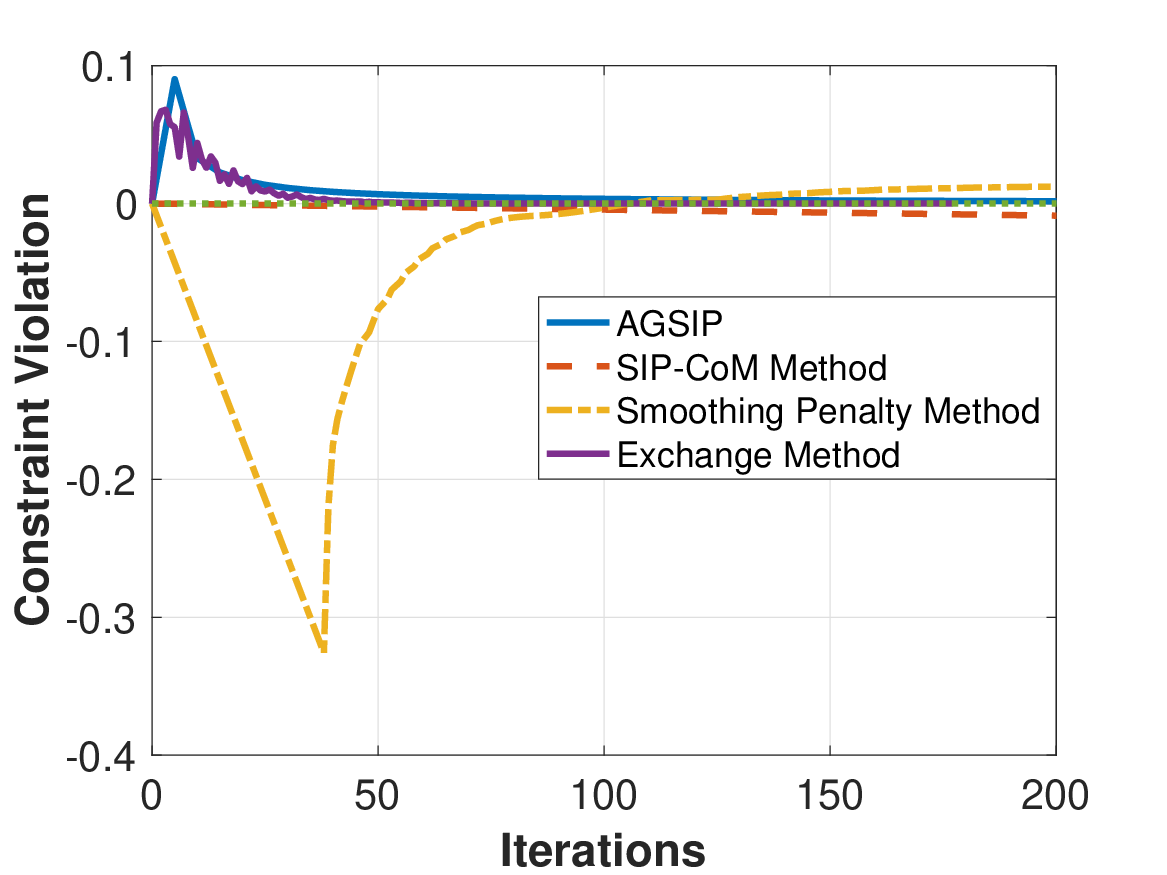}
\caption{Convergence objective values and constraint violations with number of iterations.}
\label{convergence_iter}
\end{center}
\end{figure}





\section{Conclusions}
\label{sec:conclusions}

Semi-infinite programming  has applications in transportation, robotics, statistics and machine learning but it is challenging to solved due to the infinite number of constraints. Under the assumption of convexity, an effective primal-dual gradient method is proposed where, in each iteration, momentum-driven gradient ascend and descend steps are performed. This approach is also extended using stochastic gradient and function oracles, allowing it to solve large-scale data-driven problems through data sampling. The iteration complexity of the proposed method for finding an $\epsilon$-optimal solution is establish under different strong convexity assumptions and smoothness assumptions, encompassing both deterministic and stochastic cases.

   \bibliographystyle{plain}
   \bibliography{ref, SIP}

\end{document}